\documentclass[10pt]{amsart}
\usepackage{amsmath, amsthm, amsfonts, amssymb, graphicx, hyperref, bm,verbatim,mathrsfs,siunitx,tikz,color}
\usepackage{stmaryrd,enumerate}
\theoremstyle{plain}
\newtheorem{thrm}{Theorem}[section]
\newtheorem{lmm}[thrm]{Lemma}
\newtheorem{prpstn}[thrm]{Proposition}

\newtheorem*{rmk}{Remark}
\newtheorem*{nprpstn}{Proposition}

\numberwithin{sblmm}{thrm} 
\numberwithin{equation}{section}

\newcommand{\Mod}[1]{\ (\mathrm{mod}\ #1)}

\begin{document}
\title{Primes with restricted digits}
\author{James Maynard}
\address{Magdalen College, Oxford, England, OX1 4AU}
\email{james.alexander.maynard@gmail.com}
\begin{abstract}
Let $a_0\in\{0,\dots,9\}$. We show there are infinitely many prime numbers which do not have the digit $a_0$ in their decimal expansion.

The proof is an application of the Hardy-Littlewood circle method to a binary problem, and rests on obtaining suitable `Type I' and `Type II' arithmetic information for use in Harman's sieve to control the minor arcs. This is obtained by decorrelating Diophantine conditions which dictate when the Fourier transform of the primes is large from digital conditions which dictate when the Fourier transform of numbers with restricted digits is large. These estimates rely on a combination of the geometry of numbers, the large sieve and moment estimates obtained by comparison with a Markov process.
\end{abstract}
\maketitle
%
%
%
%
%
%
%
\section{Introduction}
Let $a_0\in\{0,\dots,9\}$ and let 
\[
\mathcal{A}_1=\Bigl\{\sum_{0\le i\le k}n_i 10^i: n_i\in\{0,\dots,9\}\backslash\{a_0\},\,k\ge 0\Bigr\}
\]
be the set of numbers which have no digit equal to $a_0$ in their decimal expansion. The number of elements of $\mathcal{A}_1$ which are less than $x$ is $O(x^{1-c})$, where $c=\log{(10/9)}/\log{10}\approx 0.046>0$. In particular, $\mathcal{A}_1$ is a sparse subset of the natural numbers. A set being sparse in this way presents several analytic difficulties if one tries to answer arithmetic questions such as whether the set contains infinitely many primes. Typically we can only show that sparse sets contain infinitely many primes when the set in question possesses some additional multiplicative structure.

The set $\mathcal{A}_1$ has unusually nice structure in that its Fourier transform has a convenient explicit analytic description, and is often unusually small in size. There has been much previous work \cite{1,2,3,4,5,6,7} studying $\mathcal{A}_1$ and related sets by exploiting this Fourier structure. In particular the work of Dartyge and Mauduit \cite{Almost1,Almost2} shows the existence of infinitely many integers in $\mathcal{A}_1$ with at most $2$ prime factors, this result relying on the fact that $\mathcal{A}_1$ is well-distributed in arithmetic progressions \cite{Almost1,Level1,Level2}. We also mention the related work of Mauduit and Rivat \cite{MauduitRivat} who showed the sum of digits of primes is well-distributed, and the work of Bourgain \cite{Bourgain} which showed the existence of primes in the sparse set created by prescribing a positive proportion of the binary digits.

We show that there are infinitely many primes in $\mathcal{A}_1$. Our proof is based on a combination of the circle method, Harman's sieve, the method of bilinear sums, the large sieve, the geometry of numbers and a comparison with a Markov process. In particular, we make key use of the Fourier structure of $\mathcal{A}_1$, in the same spirit as the aforementioned works. Somewhat surprisingly, the Fourier structure allows us to successfully apply the circle method to a \textit{binary} problem.
%
%
%
%
\begin{thrm}\label{thrm:MainTheorem}
Let $X\ge 4$ and $\mathcal{A}=\{\sum_{0\le i\le k}n_i10^i< X:\, n_i\in\{0,\dots,9\}\backslash\{a_0\},\,k\ge 0\}$ be the set of numbers less than $X$ with no digit  in their decimal expansion equal to $a_0$. Then we have
\[
\#\{p\in\mathcal{A}\}\asymp \frac{\#\mathcal{A}}{\log{X}}\asymp \frac{X^{\log{9}/\log{10}}}{\log{X}}.
\]
\end{thrm}
%
%
%
%
Here, and throughout the paper, $f\asymp g$ means that there are absolute constants $c_1,c_2>0$ such that $c_1f<g<c_2f$.

Thus there are infinitely many primes with no digit $a_0$ when written in base $10$. Since $\#\mathcal{A}/X^{\log{9}/\log{10}}$ oscillates as $X\rightarrow\infty$, we cannot expect an asymptotic formula of the form $(c+o(1))X^{\log{9}/\log{10}}/\log{X}$. Nonetheless, we expect that
\[
\#\{p\in\mathcal{A}\}= (\kappa_\mathcal{A}+o(1))\frac{\#\mathcal{A}}{\log{X}},
\]
where
\begin{equation}
\kappa_\mathcal{A}=
\begin{cases}
\frac{10(\phi(10)-1)}{9\phi(10)},&\text{if $(10,a_0)=1$,}\\
\frac{10}{9},&\text{otherwise.}\\
\end{cases}
\label{eq:KappaDef}
\end{equation}
Indeed, there are $(\phi(10)\kappa_\mathcal{A}/10+o(1))\#\mathcal{A}$ elements of $\mathcal{A}$ which are coprime to 10, and $(1+o(1))X/\log{X}$ primes less than $X$ which are coprime to 10, and $(\phi(10)/10+o(1))X$ integers less than $X$ coprime to 10. Thus if the properties `being in $\mathcal{A}$' and `being prime' where independent for integers $n< X$ coprime to 10, we would expect $(\kappa_\mathcal{A}+o(1)) \#\mathcal{A}/\log{X}$ primes in $\mathcal{A}$. Theorem \ref{thrm:MainTheorem} shows this heuristic guess is within a constant factor of the truth, and we would be able to establish such an asymptotic formula if we had stronger `Type II' information. %

One can consider the same problem in bases other than 10, and with more than one excluded digit. The set of numbers less than $X$ missing $s$ digits in base $q$ has $\asymp X^{c}$ elements, where $c=\log(q-s)/\log{q}$. For fixed $s$, the density becomes larger as $q$ increases, and so the problem becomes easier. Our methods are not powerful enough to show the existence of infinitely many primes with two digits not appearing in their decimal expansion, but they can show that there are infinitely many primes with $s$ digits excluded in base $q$ provided $q$ is large enough in terms of $s$. Moreover, if the set of excluded digits possesses some additional structure this can apply to very thin sets formed in this way.
%
%
%
%
\begin{thrm}\label{thrm:ManyDigits}
Let $q$ be sufficiently large, and let $X\ge q$.

For any choice of $\mathcal{B}\subseteq\{0,\dots,q-1\}$ with $\#\mathcal{B}=s\le q^{23/80}$, let 
\[
\mathcal{A}'=\Bigl\{\sum_{0\le i \le k}n_i q^i< X:\,n_i\in\{0,\dots,q-1\}\backslash\mathcal{B},\,k\ge 0\Bigr\}
\]
 be the set of integers less than $X$ with no digit in base q in the set $\mathcal{B}$. Then we have
\[
\#\{p\in\mathcal{A}'\}\asymp \frac{X^{\log(q-s)/\log{q}}}{\log{X}}.
\]
In the special case when $\mathcal{B}=\{0,\dots,s-1\}$ or $\mathcal{B}=\{q-s,\dots,q-1\}$, this holds in the wider range $0\le s\le q-q^{57/80}$.
\end{thrm}
%
%
%
%
The final case of Theorem \ref{thrm:ManyDigits} when $\mathcal{B}=\{0,\dots,s-1\}$ and $s\approx q-q^{57/80}$ shows the existence many primes in a set of integers $\mathcal{A}'$ with $\#\mathcal{A}'\approx X^{57/80}=X^{0.7125}$, a rather thin set. The exponent here can be improved slightly with more effort.

The estimates in Theorem \ref{thrm:ManyDigits} can be improved to asymptotic formulae if we restrict $s$ slightly further. For general $\mathcal{B}$ with $s=\#\mathcal{B}\le q^{1/4-\delta}$ and any $q$ sufficiently large in terms of $\delta>0$ we obtain
\[
\#\{p\in\mathcal{A}'\}=(\kappa_\mathcal{B}+o(1))\frac{\#\mathcal{A}}{\log{X}},
\]
where, if $\mathcal{B}$ contains exactly $t$ elements coprime to $q$, we have
\[
\kappa_\mathcal{B}=
\frac{q(\phi(q)-t)}{\phi(q)(q-s)}.%
\]
In the case of just one excluded digit, we can obtain this asymptotic formula for $q\ge 12$. In the case of $\mathcal{B}=\{0,\dots,s-1\}$, we obtain the above asymptotic formula provided $s\le q-q^{3/4+\delta}$.

We expect several of the techniques introduced in this paper might be useful more generally in other digit-related questions about arithmetic sequences. Our general approach to counting primes in $\mathcal{A}$ and our analysis of the minor arc contribution might also be of independent interest, with potential application to other questions on primes involving sets whose Fourier transform is unrelated to Diophantine properties of the argument.
%
%
%
%
%
%
%
%
\section{Outline}\label{sec:Outline}
Our argument is fundamentally based on an application of the circle method. Clearly for the purposes of Theorem \ref{thrm:MainTheorem} we can restrict $X$ to a power of 10 for convenience. The number of primes in $\mathcal{A}$ is the number of solutions of the \textit{binary} equation $p-a=0$ over primes $p$ and integers $a\in\mathcal{A}$, and so is given by
\[
\#\{p\in\mathcal{A}\}=\frac{1}{X}\sum_{0\le a<X}S_{\mathcal{A}}\Bigl(\frac{a}{X}\Bigr)S_{\mathbb{P}}\Bigl(\frac{-a}{X}\Bigr),
\]
where
\begin{align*}
S_{\mathcal{A}}(\theta)&=\sum_{a\in\mathcal{A}}e(a\theta),\\
S_{\mathbb{P}}(\theta)&=\sum_{p<X}e(p\theta).
\end{align*}
We then separate the contribution from the $a$ in the `major arcs' which give our expected main term for $\#\{p\in\mathcal{A}\}$, and the $a$ in the `minor arcs' which we bound for an error term. 

The reader might be (justifiably) somewhat surprised by this, since it is well known that the circle method typically cannot be applied to binary problems. Indeed, one cannot generally hope for bounds better than `square-root cancellation'
\begin{align*}
S_{\mathbb{P}}(\theta)\ll X^{1/2},\\
S_{\mathcal{A}}(\theta)\ll \#\mathcal{A}^{1/2},
\end{align*}
for `generic' $\theta\in[0,1]$. Thus if one cannot exploit cancellation amongst the different terms in the minor arcs, we would expect that the $\gg X$ different `generic' $a$ in the sum above would contribute an error term which we can only bound as $O(X^{1/2}\#\mathcal{A}^{1/2})$, and this would dominate the expected main term.

It turns out that the Fourier transform $S_{\mathcal{A}}(\theta)$ has some somewhat remarkable features which cause it to typically have \textit{better} than square-root cancellation. (A closely related phenomenon is present and crucial in the work of Mauduit and Rivat \cite{MauduitRivat} and Bourgain \cite{Bourgain}.) Indeed, we establish the $\ell^1$ bound
\begin{equation}
\sum_{0\le a<X}\Bigl|S_{\mathcal{A}}\Bigl(\frac{a}{X}\Bigr)\Bigr|\ll \#\mathcal{A}\,X^{0.36}.\label{eq:L1Bound}
\end{equation}
which shows that for `generic' $a$ we have $S_{\mathcal{A}}(a/X)\ll \#\mathcal{A}/X^{0.64}\ll X^{0.32}$. This gives us a (small) amount of room for a possible successful application of the circle method , since now we might hope the `generic' $a$ would contribute a total $O(X^{0.82})$ if the bound $S_{\mathbb{P}}(a/X)\ll X^{1/2+\epsilon}$ held for all $a$ in the minor arcs, and this $O(X^{0.82})$ error term  is now smaller than the expected main term of size $\#\mathcal{A}^{1+o(1)}$. 

We actually get good asymptotic control over all moments (including fractional ones) of $S_{\mathcal{A}}(a/X)$ rather than just the first. By making a suitable approximation to $S_{\mathcal{A}}(\theta)$, we can re-interpret moments of this approximation as the average probability of restricted paths in  a Markov process, and obtain asymptotic estimates via a finite eigenvalue computation.

By combining an $\ell^2$ bound for $S_{\mathbb{P}}(a/X)$ with an $\ell^{1.526}$ bound for $S_{\mathcal{A}}(a/X)$, we are able to show that it is indeed the case that `generic' $a<X$ make a negligible contribution, and that we may restrict ourselves to $a\in\mathcal{E}$, some set of size $O(X^{0.36})$.

We expect that $S_{\mathbb{P}}(\theta)$ is large only when $\theta$ is close to a rational with small denominator, and $S_{\mathcal{A}}(\theta)$ is large when $\theta$ has a decimal expansion containing many 0's or 9's. Thus we expect the product to be large only when both of these conditions hold, which is essentially when $\theta$ is well approximated by a rational whose denominator is a small power of 10.

By obtaining suitable estimates for $\mathcal{A}$ in arithmetic progressions via the large sieve, one can verify that amongst all $a$ in the major arcs $\mathcal{M}$ where $a/X$ is well-approximated by a rational of small denominator we obtain our expected main term, and this comes from when $a/X$ is well-approximated by a rational with denominator 10.

Thus we are left to show when $a\in\mathcal{E}$ and $a/X$ is not close to a rational with small denominator, the product $S_{\mathcal{A}}(a/X)S_{\mathbb{P}}(-a/X)$ is small on average. By using an expansion of the indicator function of the primes as a sum of bilinear terms (similar to Vaughan's identity), we are led to bound expressions such as
\begin{equation}
\sum_{a_1,a_2\in\mathcal{E}\backslash\mathcal{M}}\Bigl|S_{\mathcal{A}}\Bigl(\frac{a_1}{X}\Bigr)S_{\mathcal{A}}\Bigl(\frac{a_1}{X}\Bigr)\Bigr|\sum_{n_1,n_2\le N}\min\Bigl(\frac{X}{N},\Bigl\|\frac{a_1n_1-a_2n_2}{X}\Bigr\|^{-1}\Bigr),\label{eq:BilinearSum}
\end{equation}
which is a weighted and averaged form of the typical expressions one encounters when obtaining a $\ell^\infty$ bound for exponential sums over primes. Here $\|\cdot\|$ is the distance to the nearest integer.

The double sum over $n_1,n_2$ in \eqref{eq:BilinearSum} is of size $O(N^2)$ for `typical' pairs $(a_1,a_2)$, and if it is noticeably larger than this then $a_1$ and $a_2$ must share some Diophantine structure. We find that the pair $(a_1,a_2)$ must lie close to the projection from $\mathbb{Z}^3$ to $\mathbb{Z}^2$ of some low height plane or low height line if this quantity is large, where the arithmetic height of the line or plane is bounded in terms of the size of the double sum. (For example, the diagonal terms $a_1=a_2$ give a large contribution and lie on a low height line, and $a_1,a_2$ which are both small give a large contribution and lie in a low height plane.)

This restricts the number and nature of pairs $(a_1,a_2)$ which can give a large contribution. Since we expect the size of $S_{\mathcal{A}}(a_1/X)S_{\mathcal{A}}(a_2/X)$ to be determined by digital rather than Diophantine conditions on $a_1,a_2$, we expect to have a smaller total contribution when restricted to these sets. By using the explicit description of such pairs $(a_1,a_2)$ we succeed in obtaining such a superior bound on the sum over these pairs. It is vital here that we are restricted to $a_1,a_2$ lying in the small set $\mathcal{E}$ (for points on a line) and outside of the set $\mathcal{M}$ of major arcs (for points in a lattice).

This ultimately allows us to get suitable bounds for \eqref{eq:BilinearSum} provided $N\in[X^{0.36},X^{0.425}]$. If this `Type II range' were larger, we would be able to express the indicator function of the primes as a combination of such bilinear expressions and easily controlled terms. We would then obtain an asymptotic estimate for $\#\{p\in\mathcal{A}\}$. Unfortunately our range is not large enough to do this. Instead we work with a minorant for the indicator function of the primes throughout our argument, which is chosen such that it is essentially a combination of bilinear expressions which do fall into this range. It is this feature which means we obtain a lower bound rather than an asymptotic estimate for the number of primes in $\mathcal{A}$.

Such a minorant is constructed via Harman's sieve, and, since it is essentially a combination of Type II terms and easily handled terms, we can obtain an asymptotic formula for elements of $\mathcal{A}$ weighed by it. This gives a lower bound 
\[
\#\{p\in\mathcal{A}\}\ge(c+o(1))\frac{\#\mathcal{A}}{\log{X}}
\]
for some constant $c$. We use numerical integration to verify that we (just) have $c>0$, and so we obtain our asymptotic lower bound for $\#\{p\in\mathcal{A}\}$. The upper bound is a simple sieve estimate.
\begin{rmk}
For the method used to prove Theorem \ref{thrm:MainTheorem}, strong assumptions such as the Generalized Riemann Hypothesis appear to be only of limited benefit. In particular, even under GRH one only gets pointwise bounds of the strength $S_\mathbb{P}(\theta)\ll X^{3/4+o(1)}$ for `generic' $\theta$, which is not strong enough to give a non-trivial minor arc bound on its own. The assumption of GRH and the above pointwise bound is sufficient to deal with the entire minor arc contribution in the regime where we obtain asymptotic formulae (i.e. when the base is sufficiently large).
\end{rmk}
%
%
%
%
%
%
\section{Notation}
We use the asymptotic notation $\ll,\gg$, $O(\cdot)$, $o(\cdot)$ throughout, denoting a dependence of the implied constant on a parameter $t$ by a subscript. As mentioned earlier, we use $f \asymp g$ to denote that both $f\ll g$ and $g\ll f$ hold. Throughout the paper $\epsilon$ will denote a single fixed positive constant which is sufficiently small; $\epsilon=10^{-100}$ would probably suffice. In particular, any implied constants may depend on $\epsilon$. We will assume that $X$ is always a suitably large integral power of 10 throughout. We will exclusively use the letter $p$ to denote a prime number, without always making this restriction explicit.

We will use the nonstandard notation that $n\sim X$ to mean that $n$ lies in the interval $(X/10,X]$ throughout the paper. 

Several variables will be assumed to be non-negative integers, without directly specifying this. Thus sums such as $\sum_{n<X}$ will be assumed to be over integers $n$ with $0\le n<X$, for example. The usage should be clear from the context.

It will be convenient to normalize the Fourier transform of $\mathcal{A}$, and to be able to view it at different scales. With this in mind, we define
\begin{equation}
F_{Y}(\theta)=Y^{-\log{9}/\log{10}}\Bigl|\sum_{n<Y}\mathbf{1}_{\mathcal{A}_1}(n)e(n\theta)\Bigr|.\label{eq:FyDef}
\end{equation}
 Whenever we encounter the function $F_Y$ we assume that $Y$ is a positive integral power of 10. (Or that they are powers of $q$ in Section \ref{sec:ManyPrimes}.) We use $\|\cdot\|$ to denote the distance to the nearest integer, and $\|\cdot\|_2$ to denote the standard Euclidean norm. We use $\mathbf{1}_{\mathcal{A}_1}$ for the indicator function of the set $\mathcal{A}_1$ of integers with restricted digits. Here $e(x)=e^{2\pi i x}$ is the complex exponential function.

We need to make use of various numerical estimates throughout the paper, some of which succeed only by a small margin. We have endeavored to avoid too many explicit calculations and we encourage the reader to not pay too much attention to the numerical constants appearing on a first reading.
%
%
%
%
%
%
\section{Structure of the paper}
In Section \ref{sec:Decomposition}, we use a sieve decomposition to reduce the proof of Theorem \ref{thrm:MainTheorem} to the proof of Proposition \ref{prpstn:FinalSieve} and Proposition \ref{prpstn:FinalTypeII}, which are asymptotic estimates for particular types of terms arising from sieve decompositions. These propositions are established in Section \ref{sec:Sieve}.

In Section \ref{sec:Sieve}, we use sieve theory to reduce the proof of Proposition \ref{prpstn:FinalSieve} and Proposition \ref{prpstn:FinalTypeII} to the proof of Proposition \ref{prpstn:TypeI} and Proposition \ref{prpstn:TypeII}, which are our `Type I' and `Type II' estimates. These will be established in Section \ref{sec:TypeI} and Section \ref{sec:TypeII} respectively.

In Section \ref{sec:TypeI} we use a large sieve argument to reduce the proof of our Type I estimate Proposition \ref{prpstn:TypeI} to that of Lemma \ref{lmm:LargeSieveTypeI} and Lemma \ref{lmm:LInfTypeI}, which are Fourier $\ell^\infty$ and $\ell^1$ bounds. These will be established in Section \ref{sec:Fourier}.

In Section \ref{sec:TypeII} we use the circle method and geometric decompositions to reduce the proof of our Type II estimate  Proposition \ref{prpstn:TypeII} to that of Proposition \ref{prpstn:Major}, Proposition \ref{prpstn:Generic} and Proposition \ref{prpstn:Exceptional}, which are our estimates for the `major arcs', the `generic minor arcs' and the `exceptional minor arcs'. These will be established in Sections \ref{sec:Major}, \ref{sec:Generic} and \ref{sec:Bilinear} respectively.

In Section \ref{sec:Fourier} we establish various Fourier estimates. In particular we establish Lemma \ref{lmm:LargeSieveTypeI} and Lemma \ref{lmm:LInfTypeI}, as well as several auxiliary lemmas which will be used in later sections.

In Section \ref{sec:Major} use results on primes in arithmetic progressions to establish our major arc estimate Proposition \ref{prpstn:Major}, making use of the estimates of Section \ref{sec:Fourier}.

In Section \ref{sec:Generic} we use Fourier moment bounds from Section \ref{sec:Fourier} to establish our generic minor arc estimate Proposition \ref{prpstn:Generic}.

In Section \ref{sec:Bilinear} we use the geometry of numbers to reduce the proof of the exceptional minor arc estimate Proposition \ref{prpstn:Exceptional} to the proof of Proposition \ref{prpstn:LatticeBound} and Proposition \ref{prpstn:LineBound}, which are estimates from frequencies constrained to lie in low height lattices or low height lines. These will be established in Section \ref{sec:Lattice} and Section \ref{sec:Line}.

In Section \ref{sec:Lattice} we establish our estimate for low height lattices  Proposition \ref{prpstn:LatticeBound}, using the estimates of Section \ref{sec:Fourier}.

In Section \ref{sec:Line} we establish our estimate for low height lines Proposition \ref{prpstn:LineBound} , using the geometric counting estimates and the results of Section \ref{sec:Fourier}. This completes the proof of Theorem \ref{thrm:MainTheorem}.

In Section \ref{sec:ManyPrimes}, we sketch the modifications in the argument required to establish Theorem \ref{thrm:ManyDigits}.

In particular, the dependency graph between the main statements in the proof of Theorem \ref{thrm:MainTheorem} is as follows:
  \begin{center}
\makebox[\textwidth]{\parbox{1.5\textwidth}{
\begin{center}
   \tikzstyle{interface}=[draw, text width=6em,
      text centered, minimum height=2.0em]
   \tikzstyle{daemon}=[draw, text width=7em,
      minimum height=2em, text centered, rounded corners]
   \tikzstyle{dots} = [above, text width=6em, text centered]
   \tikzstyle{wa} = [daemon, text width=6em,
      minimum height=2em, rounded corners]
   \tikzstyle{ur}=[draw, text centered, minimum height=0.01em]
 
   \def\blockdist{1.3}
   \def\edgedist{0.}

   \begin{tikzpicture}
 
      \node (wa)[interface]  {Theorem \ref{thrm:MainTheorem}};

      \path (wa.west)+(0.6,+1.2) node (d1)[daemon] {\footnotesize Proposition \ref{prpstn:FinalTypeII}};
      \path (wa.west)+(0.6,-1.2) node (d2)[daemon] {\footnotesize Proposition \ref{prpstn:FinalSieve}};
      \path (wa.west)+(-2.7,0.8) node (d3)[daemon] {\footnotesize Proposition \ref{prpstn:TypeI}};
      \path (wa.west)+(-2.7,-0.8) node (d4)[daemon] {\footnotesize Proposition \ref{prpstn:TypeII}};
      \path (wa.west)+(-5.7,2) node (d5)[daemon] {\footnotesize Lemma \ref{lmm:LargeSieveTypeI}};
      \path (wa.west)+(-5.7,1) node (d6)[daemon] {\footnotesize Lemma \ref{lmm:LInfTypeI}};
      \path (wa.west)+(-5.7,0) node (d7)[daemon] {\footnotesize Proposition \ref{prpstn:Major}};
      \path (wa.west)+(-5.7,-1) node (d8)[daemon] {\footnotesize Proposition \ref{prpstn:Generic}};
      \path (wa.west)+(-5.7,-2) node (d9)[daemon] {\footnotesize Proposition \ref{prpstn:Exceptional}};
      \path (wa.west)+(-8.7,-1.5) node (d10)[daemon] {\footnotesize Proposition \ref{prpstn:LatticeBound}};
      \path (wa.west)+(-8.7,-2.5) node (d11)[daemon] {\footnotesize Proposition \ref{prpstn:LineBound}};
      \path [draw, ->,>=stealth] (d1.south) -- node [above] {} (wa.north) ;
      \path [draw, ->,>=stealth] (d2.north) -- node [above] {} (wa.south) ;
      \path [draw, ->,>=stealth] (d3.east) -- node [above] {} (d1.west) ;
      \path [draw, ->,>=stealth] (d3.east) -- node [above] {} (d2.north west) ;
      \path [draw, ->,>=stealth] (d4.east) -- node [above] {} (d1.south west) ;
      \path [draw, ->,>=stealth] (d4.east) -- node [above] {} (d2.west) ;
      \path [draw, ->,>=stealth] (d5.east) -- node [above] {} (d3.north west) ;
      \path [draw, ->,>=stealth] (d6.east) -- node [above] {} (d3.west) ;
      \path [draw, ->,>=stealth] (d7.east) -- node [above] {} (d4.north west) ;
      \path [draw, ->,>=stealth] (d8.east) -- node [above] {} (d4.west) ;
      \path [draw, ->,>=stealth] (d9.east) -- node [above] {} (d4.south west) ;
      \path [draw, ->,>=stealth] (d10.east) -- node [above] {} (d9.north west) ;
      \path [draw, ->,>=stealth] (d11.east) -- node [above] {} (d9.south west) ;
      \path [draw, ->,>=stealth] (d3.south) -- node [above] {} (d4.north) ;

   \end{tikzpicture}
\end{center}}}
   \end{center}
%
%
%
%
%
%
\section{Basic estimates}
We will make frequent use of some well-known facts in analytic number theory without extra comment. In particular, we make use of the Prime Number Theorem in short intervals and arithmetic progressions with error term (see \cite[Chapter 22]{Davenport}, for example). This states that for any $A>0$ we have
\begin{equation}
\sum_{\substack{Y\le n\le Y+\Delta Y\\ n\equiv a\pmod{q}}}\Lambda(n)=\frac{\Delta Y}{\phi(q)}+O_A\Bigl(\frac{Y}{(\log{Y})^A}\Bigr)
\label{eq:PNT}
\end{equation}
provided $\Delta\ge (\log{Y})^{-A}$ and $q\le (\log{Y})^A$ and $\gcd(a,q)=1$.

We recall the following sieve estimate (see, for example, \cite[Theorem 7.11]{MontgomeryVaughan}): For $u>1+1/(\log{Y})^{1/2}$
\begin{equation}
\#\{n<Y:\,p|n\Rightarrow p\ge Y^{1/u}\}=(\omega(u)+o_u(1))\frac{u Y}{\log{Y}},
\label{eq:Buchstab}
\end{equation}
where $\omega(u)$ is the Buchstab function defined by the delay-differential equation
\begin{align*}
\qquad\omega(u)&=1/u, &&1\le u\le 2,\qquad\\
\omega'(u)&=\omega(u-1)-\omega(u), &&u>2.
\end{align*}
We recall some results from the geometry of numbers and Minkowski's theory of successive minima (see, for example, \cite[Page 110]{DavenportGeom}). A \textit{lattice} in $\mathbb{R}^k$ is a discrete subgroup of the additive group $\mathbb{R}^k$. For any lattice $\Lambda$ there is a Minkowski-reduced basis $\{\mathbf{v}_1,\dots,\mathbf{v}_r\}$ of linearly independent vectors in $\mathbb{R}^k$ such that
\[
\Lambda=\mathbf{v}_1\mathbb{Z}+\dots+\mathbf{v}_r\mathbb{Z},
\]
and for any $x_1,\dots,x_r\in\mathbb{R}$ we have
\[
\|x_1\mathbf{v}_1+\dots+x_r\mathbf{v}_r\|_2\asymp \sum_{i=1}^r\|x_i\mathbf{v}_i\|_2,
\]
and with $\|\mathbf{v}_1\|_2\cdots\|\mathbf{v}_r\|_2\asymp \det(\Lambda)$, where these implied constants depend only on the ambient dimension $k$. Here $\det(\Lambda)$ is the $r$-dimensional volume of the fundamental parallelepiped, given by
\[
\Bigl\{\sum_{i=1}^r x_i\mathbf{v}_i:\,x_1,\dots,x_r\in[0,1]\Bigr\}.
\]
We say $r$ is the \textit{rank} of the lattice. We see the properties of the Minkowski-reduced basis above indicate that each generating vector $\mathbf{v}_i$ has a positive proportion of its length in a direction orthogonal to all the other basis vectors.
%
%
%
%
%
%
\section{Sieve Decomposition and proof of Theorem \ref{thrm:MainTheorem}}\label{sec:Decomposition}
First, we prove Theorem \ref{thrm:MainTheorem} assuming two key propositions, given below. This reduces the problem to establishing Proposition \ref{prpstn:FinalSieve} and Proposition \ref{prpstn:FinalTypeII} which we do over the remaining sections.

As remarked in Section \ref{sec:Outline}, it suffices to consider $X$ as a power of 10. If $X=10^k$ we will think of all elements of $\mathcal{A}$ as having $k$ digits, none of which is equal to $a_0$. This is equivalent to slightly changing the definition of $\mathcal{A}$ in the case when $a_0=0$ (since it restricts $\mathcal{A}$ to $(X/10,X]$), but by considering $X$, $X/10$, $X/100$, $\dots$ we see that we can easily recover Theorem \ref{thrm:MainTheorem} for the original set $\mathcal{A}$ from this situation. 

We will make a decomposition of $\#\{p\in\mathcal{A}{}\}$ into various terms following Harman's sieve (see \cite{HarmanBook} for more details). Each of these terms can then be asymptotically estimated by Proposition \ref{prpstn:FinalSieve} or Proposition \ref{prpstn:FinalTypeII} (given below), or can be trivially bounded below by 0. To keep track of the terms in this decomposition we apply the same decomposition to the set 
\[
\mathcal{B}=\{0\le n< X\}
\]
 by considering a weighted sequence $w_n$.

Let $w_n$ be weights supported on non-negative integers $n<  X$ given by
\begin{equation}
w_n=\mathbf{1}_{\mathcal{A}{}}(n)-\frac{\kappa_\mathcal{A}\#\mathcal{A}{}}{\#\mathcal{B}{}}=\mathbf{1}_{\mathcal{A}{}}(n)-\frac{\kappa_\mathcal{A}\#\mathcal{A}{}}{X}\ge -\frac{\kappa_\mathcal{A}\#\mathcal{A}{}}{X}.
\label{eq:WDef}
\end{equation}
 (We recall that $\mathbf{1}_{\mathcal{A}}$ is the indicator function of $\mathcal{A}$, and $\kappa_\mathcal{A}$ is the constant given by \eqref{eq:KappaDef}.)  For a set $\mathcal{C}$ we define
\begin{align*}
\mathcal{C}_d&=\{c:\,c d\in\mathcal{C}\},\\
S(\mathcal{C},z)&=\#\{c\in\mathcal{C}:\,p|c\Rightarrow p>z\}.
\end{align*}
Given an integer $d>0$ and a real number $z>0$, let
\begin{equation}
S_d(z)=\sum_{\substack{n<X/d\\ p|n\Rightarrow p>z}}w_{n d}=S(\mathcal{A}{}_d,z)-\frac{\kappa_\mathcal{A}\#\mathcal{A}{}}{X}S(\mathcal{B}{}_d,z).
\label{eq:SWDef}
\end{equation}
We expect that $S_d(z)$ is typically small for a wide range of $d$ and $z$. The following two propositions show that this is the case for certain $d,z$.
%
%
%
%
\begin{prpstn}[Sieve asymptotic terms]\label{prpstn:FinalSieve}
Fix an integer $\ell\ge 0$. Let $\theta_1=9/25+2\epsilon$ and $\theta_2=17/40-2\epsilon$. Let $\mathcal{L}$ be a set of $O(1)$ affine linear functions $L:\mathbb{R}^\ell\rightarrow\mathbb{R}$. Then we have
\[
\sum_{\substack{X^{\theta_2-\theta_1}\le p_1\le \dots\le p_\ell \\ p_1\cdots p_\ell\le X^{1-\theta_1}}}^* S_{p_1\cdots p_\ell}(X^{\theta_2-\theta_1})=o_{\mathcal{L}}\Bigl(\frac{\#\mathcal{A}}{\log{X}}\Bigr),
\]
where $\sum^*$ indicates the summation is restricted by the conditions
\[
L\Bigl(\frac{\log{p}_1}{\log{X}},\,\dots\,,\,\frac{\log{p_\ell}}{\log{X}}\Bigr)\ge 0
\]
for all $L\in\mathcal{L}$.
\end{prpstn}
Proposition \ref{prpstn:FinalSieve} includes the case $\ell=0$, where we interpret the statement as
\begin{equation}
S_1(X^{\theta_2-\theta_1})=o\Bigl(\frac{\#\mathcal{A}}{\log{X}}\Bigr).
\label{eq:BasicSieveBound}
\end{equation}
%
%
%
%
%
%
%
\begin{prpstn}[Type II terms] \label{prpstn:FinalTypeII}
Fix an integer $\ell\ge1$. Let $\theta_1,\theta_2,\mathcal{L}$ be as in Proposition \ref{prpstn:FinalSieve}, and let $\mathcal{I}\subseteq\{1,\dots,\ell\}$ and $j\in\{1,\dots,\ell\}$. Then we have
\[
\sum_{\substack{X^{\theta_2-\theta_1}\le p_1\le \dots\le p_\ell \\ X^{\theta_1}\le\prod_{i\in\mathcal{I}}p_i\le X^{\theta_2} \\ p_1\cdots p_\ell\le X/p_j }}^*S_{p_1\cdots p_\ell}(p_j)=o_{\mathcal{L}}\Bigl(\frac{\#\mathcal{A}}{\log{X}}\Bigr),
\]
and
\[
\sum_{\substack{X^{\theta_2-\theta_1}\le p_1\le \dots\le p_\ell \\ X^{1-\theta_2}\le\prod_{i\in\mathcal{I}}p_j\le X^{1-\theta_1} \\ p_1\cdots p_\ell \le X/p_j }}^*S_{p_1\cdots p_\ell}(p_j)=o_{\mathcal{L}}\Bigl(\frac{\#\mathcal{A}}{\log{X}}\Bigr),
\]
where $\sum^*$ indicates the same restriction of summation to $L\ge 0$ for all $L\in\mathcal{L}$ as in Proposition \ref{prpstn:FinalSieve}.
\end{prpstn}
%
%
%
%
We note that by inclusion-exclusion the same result holds if some of the inequalities $L\ge 0$ are replaced by the strict inequality $L>0$.
%
%
%
%
\begin{proof}[Proof of Theorem \ref{thrm:MainTheorem} assuming Proposition \ref{prpstn:FinalSieve} and Proposition \ref{prpstn:FinalTypeII}]
Let $\theta_1=9/25+2\epsilon$ and $\theta_2=17/40-2\epsilon$ as in Proposition \ref{prpstn:FinalSieve}. 

We first consider the upper bound for Theorem \ref{thrm:MainTheorem}, which is essentially a standard sieve upper bound.  Since $\theta_2-\theta_1<1/2$, we have
\begin{align*}
\#\{p\in\mathcal{A}{}\}= S(\mathcal{A}{},X^{1/2})+O(X^{1/2})\le S(\mathcal{A}{},X^{\theta_2-\theta_1})+O(X^{1/2}).
\end{align*}
Thus, using \eqref{eq:BasicSieveBound} and the fact \eqref{eq:Buchstab} that there are $O(X/\log{X})$ integers in $[0,X]$ with no prime factors smaller than $X^{\theta_2-\theta_1}$, we have
\begin{align*}
\#\{p\in\mathcal{A}{}\}&\le S(\mathcal{A}{},X^{\theta_2-\theta_1})+O(X^{1/2})\\
&=\kappa_{\mathcal{A}}\frac{\#\mathcal{A}{}}{X}S(\mathcal{B}{},X^{\theta_2-\theta_1})+S_1(X^{\theta_2-\theta_1})+O(X^{1/2})\\
&=\kappa_{\mathcal{A}}\frac{\#\mathcal{A}{}}{X}\#\{n< X:\,p|n\Rightarrow p>X^{\theta_2-\theta_1}\}+o\Bigl(\frac{\#\mathcal{A}{}}{\log{X}}\Bigr)\\
&\ll \frac{\#\mathcal{A}{}}{\log{X}}.
\end{align*}
 Thus it suffices to establish the lower bound.

To simplify notation, we let $z_1\le z_2\le z_3\le z_4\le z_5\le z_6$ be given by
\begin{align*}
&z_1=X^{\theta_2-\theta_1},&\qquad &z_2=X^{\theta_1},&\qquad &z_3=X^{\theta_2},\\
&z_4=X^{1/2},&\qquad &z_5=X^{1-\theta_2},&\qquad &z_6=X^{1-\theta_1}.&
\end{align*}
We have
\[
\#\{p\in\mathcal{A}{}\}=\#\{p\in\mathcal{A}{}:\,p>X^{1/2}\}+O(X^{1/2})=S_1(z_4)+(1+o(1))\frac{\kappa_{\mathcal{A}}\#\mathcal{A}{}}{\log{X}}.
\]
Thus we wish to bound $S_1(z_4)$ from below. By Buchstab's identity (i.e. inclusion-exclusion on the least prime factor) we have
\[
S_1(z_4)=S_1(z_1)-\sum_{z_1<p\le z_4}S_p(p).
\]
The term $S_1(z_1)$ is $o(\#\mathcal{A}{}/\log{X})$ by \eqref{eq:BasicSieveBound} from Proposition \ref{prpstn:FinalSieve}. We split the sum over $p$ into ranges $(z_i,z_{i+1}]$, and see that all the terms with $p\in (z_2,z_3]$ are also negligible by Proposition \ref{prpstn:FinalTypeII}. This gives
\[
S_1(z_4)
=-\sum_{z_1<p\le z_2}S_p(p)
-\sum_{z_3<p\le z_4}S_p(p)
+o\Bigl(\frac{\#\mathcal{A}{}}{\log{X}}\Bigr).
\]
We wish to replace $S_p(p)$ by $S_p(\min(p,(X/p)^{1/2}))$. We note that these are the same when $p\le X^{1/3}$, but if $p>X^{1/3}$ then there are additional terms in $S_p((X/p)^{1/2})$ from primes in the interval $((X/p)^{1/2},p]$. For $\delta=1/(\log{X})^{1/2}$, by the prime number theorem and Proposition \ref{prpstn:FinalSieve}, we have
\begin{align}
0&\le\sum_{p<X^{1/2}}\Bigl(S(\mathcal{A}{}_p,\min(p,(X/p)^{1/2}))-S(\mathcal{A}{}_p,p)\Bigr)\nonumber\\
&\le\sum_{p<X^{1/2-\delta}}\sum_{\substack{(X/p)^{1/2}<q\le p\\ q p\in\mathcal{A}{}}}1+\sum_{X^{1/2-\delta}\le p\le X^{1/2}}S(\mathcal{A}{}_p,z_1)\nonumber\\
&\ll \sum_{\substack{a\in\mathcal{A}{}\\ a<X^{1-\delta} }}1+\frac{\#\mathcal{A}{}}{\log{X}}\sum_{X^{1/2-\delta}\le p<X^{1/2}}\frac{1}{p}\nonumber\\
&=o\Bigl(\frac{\#\mathcal{A}{}}{\log{X}}\Bigr).\label{eq:FactorReduction}
\end{align}
Here, and throughout this section, $q$ is restricted to being a prime number. Similarly, we get corresponding bounds for $S(\mathcal{B}{}_p,\min(p,(X/p)^{1/2}))$, and so we can replace $S_p(p)$ with $S_p(\min(p,(X/p)^{1/2}))$ at the cost of a small error. 

Using this, and applying Buchstab's identity again, we have
\begingroup
\allowdisplaybreaks
\begin{align*}
S_1(z_4)
&=-\sum_{z_1<p\le z_2}S_p(\min(p,(X/p)^{1/2}))
-\sum_{z_3<p\le z_4}S_p(\min(p,(X/p)^{1/2}))
+o\Bigl(\frac{\#\mathcal{A}{}}{\log{X}}\Bigr)\\
&=-\sum_{z_1<p\le z_2}S_p(z_1)
-\sum_{z_3<p\le z_4}S_p(z_1)
+\sum_{\substack{z_1<q\le p\le z_2\\ q\le (X/p)^{1/2}}}S_{p q}(q)\\
&\qquad+\sum_{\substack{z_3<p\le z_4 \\ z_1<q\le (X/p)^{1/2}}}S_{p q}(q)+o\Bigl(\frac{\#\mathcal{A}{}}{\log{X}}\Bigr).
\end{align*}
\endgroup
The first two terms above are asymptotically negligible by Proposition \ref{prpstn:FinalSieve}, and so this simplifies to
\begin{align}
S_1(z_4)
&=\sum_{\substack{z_1<q\le p\le z_2\\ q\le (X/p)^{1/2}}}S_{p q}(q)+\sum_{\substack{z_3<p\le z_4 \\ z_1<q\le (X/p)^{1/2}}}S_{p q}(q)+o\Bigl(\frac{\#\mathcal{A}{}}{\log{X}}\Bigr).
\label{eq:Decomp}
\end{align}
We perform further decompositions to the remaining terms in \eqref{eq:Decomp}. We first concentrate on the first term on the right hand. Splitting the ranges of $p q$ into intervals, and recalling those with a $p q$ in the interval $[z_2,z_3]$ or $[z_5,z_6]$ make a negligible contribution by Proposition \ref{prpstn:FinalTypeII}, we obtain
\begin{align}
\sum_{\substack{z_1<q\le p\le z_2\\ q\le (X/p)^{1/2}}}S_{p q}(q)
&=\sum_{\substack{z_1<q\le p\le z_2\\ q\le (X/p)^{1/2}\\ z_6< p q}}S_{p q}(q)
+\sum_{\substack{z_1<q\le p\le z_2\\ q\le (X/p)^{1/2}\\ z_3\le p q<z_5}}S_{p q}(q)\nonumber\\
&\qquad+\sum_{\substack{z_1<q\le p\le z_2\\ z_1\le p q<z_2}}S_{p q}(q)+o\Bigl(\frac{\#\mathcal{A}{}}{\log{X}}\Bigr).
\label{eq:Decomp1}
\end{align}
Here we have dropped the condition $q\le (X/p)^{1/2}$ in the final sum, since this is implied by $q\le p$ and $p q\le z_2$.
On recalling the definition \eqref{eq:WDef} of $w_n$, we can lower bound the first term of \eqref{eq:Decomp1} by dropping the non-negative contribution from the set $\mathcal{A}{}$ via $w_n\ge -\kappa_{\mathcal{A}}\#\mathcal{A}{}/X$. By partial summation, and using the estimate \eqref{eq:Buchstab}, this gives
\begingroup
\allowdisplaybreaks
\begin{align}
\sum_{\substack{z_1<q\le p\le z_2\\ q\le (X/p)^{1/2}\\  z_6<p q}}S_{p q}(q)
&\ge \frac{-\kappa_{\mathcal{A}}\#\mathcal{A}}{X}\sum_{\substack{z_1<q\le p\le z_2\\ q\le (X/p)^{1/2}\\  z_6<p q}}S(\mathcal{B}_{p q},q)\nonumber\\
&\ge  \frac{-\kappa_{\mathcal{A}}\#\mathcal{A}}{X}\sum_{\substack{z_1<q\le p\le z_2\\ q\le (X/p)^{1/2}\\  z_6<p q}}\sum_{\substack{n<X/p q\\ P^-(n)>q}}1\nonumber\\
&\ge  -(\kappa_{\mathcal{A}}+o(1))\#\mathcal{A}\sum_{\substack{z_1<q\le p\le z_2\\ q\le (X/p)^{1/2}\\  z_6<p q}}\frac{\omega\Bigl(\frac{\log{X/p q}}{\log{q}}\Bigr)}{p q \log{q} }+o\Bigl(\frac{\#\mathcal{A}}{\log{X}}\Bigr)\nonumber\\
&\ge -(1+o(1))\frac{\kappa_{\mathcal{A}}\#\mathcal{A}{}}{\log{X}}\iint\limits_{\substack{\theta_2-\theta_1<v<u<\theta_1\\ v<(1-u)/2\\ 1-\theta_1<u+v}}
\omega\Bigl( \frac{1-u-v}{v}\Bigr)\frac{d u d v}{u v^2}.
\label{eq:I1}
\end{align}
\endgroup
Here $\omega(u)$ is Buchstab's function, and $P^-(n)$ denotes the least prime factor of $n$.%

We perform further decompositions to the second term of \eqref{eq:Decomp1}, first splitting according to the size of $q^2 p$ compared with $z_6$.
\begin{equation}
\sum_{\substack{z_1<q\le p\le z_2\\ q\le (X/p)^{1/2}\\ z_3\le p q<z_5}}S_{p q}(q)
=\sum_{\substack{z_1<q\le p\le z_2\\ z_3\le p q<z_5\\ q^2 p < z_6}}S_{p q}(q)
+\sum_{\substack{z_1<q\le p\le z_2\\ z_3\le p q<z_5\\ z_6\le q^2 p\le X}}S_{p q}(q).
\label{eq:Decomp1B}
\end{equation}
For the second term of \eqref{eq:Decomp1B} when $q^2p$ is large, we first separate the contribution from products of three primes. By an essentially identical argument to when we replaced $S_p(p)$ by $S_p(\min(p,(X/p)^{1/2}))$ in \eqref{eq:FactorReduction}, we may replace $S_{p q}(q)$ by $S_{p q}(\min(q,(X/p q)^{1/2}))$ at the cost of a negligible error term (since $p q<z_6$). By Buchstab's identity we have (with $r$ restricted to being prime)
\[
\sum_{\substack{z_1<q\le p\le z_2\\ z_3\le p q<z_5\\ z_6\le q^2 p\le X}}S_{p q}(\min(q,(X/p q)^{1/2}))=\sum_{\substack{z_1<q\le p\le z_2\\ z_3\le p q<z_5\\ z_6\le q^2 p\le X}}S_{p q}((X/p q)^{1/2})+\sum_{\substack{z_1<q\le p\le z_2\\  z_3\le p q<z_5\\ z_6\le q^2 p\le X\\ q<r\le (X/p q)^{1/2} \\}}S_{p q r}(r).
\]
The first term above is counting products of exactly three primes, and for these terms we drop the contribution from $\mathcal{A}{}$ for a lower bound. By partial summation and the prime number theorem, this gives
\begin{equation}
\sum_{\substack{z_1<q\le p\le z_2\\ z_3\le p q<z_5\\ z_6\le q^2 p\le X}}S_{p q}((X/p q)^{1/2})
\ge-(1+o(1))\frac{\kappa_{\mathcal{A}}\#\mathcal{A}{}}{\log{X}}
\iint\limits_{\substack{\theta_2-\theta_1<v<u<\theta_1\\ \theta_2<u+v<1-\theta_2\\ 1-\theta_1<2v+u<1}}
\frac{d u d v}{u v(1-u-v)}.
\label{eq:I2}
\end{equation}
For the terms not coming from products of 3 primes, we split our summation according to the size of $qr$, noting that this is negligible if $qr\in[z_2,z_3]$ by Proposition \ref{prpstn:FinalTypeII}. For the terms with $qr\notin [z_2,z_3]$ we just take the trivial lower bound. Thus
\begin{align}
\sum_{\substack{z_1<q\le p\le z_2 \\ z_3\le p q<z_5\\ z_6\le q^2 p\le X\\ q<r\le (X/p q)^{1/2}}}S_{p q r}(r)&=\sum_{\substack{z_1<q\le p\le z_2 \\ z_3\le p q<z_5\\ z_6\le q^2 p\le X\\ q<r\le (X/p q)^{1/2}\\ qr<z_2}}S_{p q r}(r)+\sum_{\substack{z_1<q\le p\le z_2\\  z_3\le p q<z_5\\ z_6\le q^2 p\le X\\ q<r\le (X/p q)^{1/2}\\ qr>z_3}}S_{p q r}(r)\nonumber\\
&\qquad+o\Bigl(\frac{\#\mathcal{A}{}}{\log{X}}\Bigr)\nonumber\\
&\ge-(1+o(1))\frac{\kappa_{\mathcal{A}}\#\mathcal{A}{}}{\log{X}}
 \iiint\limits_{(u,v,w)\in\mathcal{R}_1}\omega\Bigl(\frac{1-u-v-w}{w}\Bigr)\frac{d u d v d w}{u v w^2}\label{eq:I3}\\
 &-(1+o(1))\frac{\kappa_{\mathcal{A}}\#\mathcal{A}{}}{\log{X}} \iiint\limits_{(u,v,w)\in\mathcal{R}_2}\omega\Bigl(\frac{1-u-v-w}{w}\Bigr)\frac{d u d v d w}{u v w^2},\label{eq:I4}
\end{align}
where $\mathcal{R}_1$ and $\mathcal{R}_2$ are given by
\begingroup
\allowdisplaybreaks
\begin{align*}
\mathcal{R}_1&=\Bigl\{(u,v,w):\,\theta_2-\theta_1<v<u<\theta_1,\,\theta_2<u+v<1-\theta_2,\,1-\theta_1<u+2v<1,\\
&\qquad v<w<(1-u-v)/2,\,v+w<\theta_1\Bigr\},\\
\mathcal{R}_2&=\Bigl\{(u,v,w):\,\theta_2-\theta_1<v<u<\theta_1,\,\theta_2<u+v<1-\theta_2,\,1-\theta_1<u+2v<1,\\
&\qquad v<w<(1-u-v)/2,\,v+w>\theta_2\Bigr\}.
\end{align*}
\endgroup
Together \eqref{eq:I2}, \eqref{eq:I3} and \eqref{eq:I4} give a suitable lower bound for the terms in \eqref{eq:Decomp1B} with $q^2p\ge z_6$.

When $q^2p<z_6$ we can apply two further Buchstab iterations, since then we can evaluate terms $S_{p q r}(z_1)$ with $r\le q\le p$ using Proposition \ref{prpstn:FinalSieve} as $p q r\le p q^2<z_6$. As before, we may replace $S_{p q}(q)$ by $S_{p q}(\min(q,(X/p q)^{1/2}))$ and $S_{p q r}(r)$ with $S_{p q r}(\min(r,(X/p q r)^{1/2}))$ at the cost of negligible error terms (since $p q r<z_6$). This gives
\begingroup
\allowdisplaybreaks
\begin{align*}
\sum_{\substack{z_1<q\le p\le z_2\\ q^2 p < z_6\\ z_3\le p q<z_5}}S_{p q}(q)&=\sum_{\substack{z_1<q\le p\le z_2\\ q^2 p < z_6\\ z_3\le p q<z_5}}S_{p q}(\min(q,(X/p q)^{1/2}))+o\Bigl(\frac{\#\mathcal{A}}{\log{X}}\Bigr)\\
&=\sum_{\substack{ z_1<q\le p\le z_2\\ q^2 p < z_6\\ z_3\le p q<z_5}}S_{p q}(z_1)
-\sum_{\substack{z_1<r\le q\le p\le z_2\\ q^2 p < z_6\\ z_3\le p q<z_5\\ r\le (X/p q)^{1/2} }}S_{p q r}(r)+o\Bigl(\frac{\#\mathcal{A}{}}{\log{X}}\Bigr)\\
&=%
o\Bigl(\frac{\#\mathcal{A}{}}{\log{X}}\Bigr)
-\sum_{\substack{z_1<r\le q\le p\le z_2\\ q^2 p < z_6\\ z_3\le p q<z_5\\ r\le (X/p q)^{1/2} }}S_{p q r}(\min(r,(X/p q r)^{1/2}))\\
&=o\Bigl(\frac{\#\mathcal{A}{}}{\log{X}}\Bigr)-\sum_{\substack{z_1<r\le q\le p\le z_2\\ q^2 p < z_6\\ z_3\le p q<z_5 \\ r\le (X/ p q)^{1/2}}}S_{p q r}(z_1)
+\sum_{\substack{z_1<s\le r\le q\le p\le z_2\\ q^2 p < z_6\\ z_3\le p q<z_5 \\ r^2 p q, s^2 r p q\le X }}S_{p q r s}(s)\\
&=o\Bigl(\frac{\#\mathcal{A}{}}{\log{X}}\Bigr)
+\sum_{\substack{z_1<s\le r\le q\le p\le z_2\\ q^2 p < z_6\\ z_3\le p q<z_5\\ r^2p q, s^2p q r\le X }}S_{p q r s}(s),
\end{align*}
\endgroup
where $r,s$ are restricted to primes in the sums above. Finally we see that any part of the final sum with a product of two of $p,q,r,s$ in $[z_2,z_3]$ can be discarded by Proposition \ref{prpstn:FinalTypeII}. Trivially lower bounding the remaining terms as we did before yields
\begin{align}
\sum_{\substack{z_1<s\le r\le q\le p\le z_2\\ q^2 p < z_6\\ z_3\le p q<z_5\\ r^2p q, s^2p q r\le X }}&S_{p q r s}(s)\nonumber\\
&\ge-(1+o(1))\frac{\kappa_{\mathcal{A}}\#\mathcal{A}{}}{\log{X}}
 \iiiint\limits_{(u,v,w,t)\in\mathcal{R}_3}\omega\Bigl(\frac{1-u-v-w-t}{t}\Bigr)\frac{d u d v d w d t}{u v w t^2},
 \label{eq:I5}
\end{align}
where $\mathcal{R}_3$ is given by
\begin{align*}
\mathcal{R}_3
=\Bigl\{(u,v,w,t):\,
&\theta_2-\theta_1<t<w<v<u<\theta_1,\,u+2v<1-\theta_1,u+v+2w<1,\\
&u+v+w+2t<1,\,\theta_2<u+v<1-\theta_2,\\
 &\{u+v,u+w,u+t,v+w,v+t,w+t\}\cap[\theta_1,\theta_2]=\emptyset\,\Bigr\}.
\end{align*}
This completes our decomposition of the terms from \eqref{eq:Decomp1B}, coming from the second term of \eqref{eq:Decomp1}. We note that we could have imposed various further restrictions such as $u+v+w\notin[\theta_1,\theta_2]$ in $\mathcal{R}_3$, but for ease of calculation we do not include these.

We perform decompositions to the third term of \eqref{eq:Decomp1} in a similar way to how we dealt with the second term. We have $q^2 p<(q p)^{3/2}< z_2^{3/2}< z_6$ so, as above, we can apply two Buchstab iterations and use Proposition \ref{prpstn:FinalSieve} to evaluate the terms $S_{p q r}(z_1)$ since we have $p q r\le p q^2<z_6$. Furthermore, we notice that terms with any of $p q r, p q s, p r s$, or $q r s$ in $[z_2,z_3]\cup[z_5,z_6]$ are negligible by Proposition \ref{prpstn:FinalTypeII}. %
This gives
\begingroup
\allowdisplaybreaks
\begin{align}
\sum_{\substack{z_1<q\le p\le z_2\\ z_1\le p q<z_2}}S_{p q}(q)%
&=\sum_{\substack{z_1<q\le p\le z_2\\ z_1\le p q<z_2}}S_{p q}(z_1)-\sum_{\substack{z_1<r\le q\le p\le z_2\\ z_1\le p q<z_2}}S_{p q r}(r)\nonumber\\%
&=o\Bigl(\frac{\#\mathcal{A}}{\log{X}}\Bigr)-\sum_{\substack{z_1<r\le q\le p\le z_2\\ z_1\le p q<z_2}}S_{p q r}(z_1)%
+\sum_{\substack{z_1<s<r<q<p<z_2\\ z_1<p q<z_2}}S_{p q r s}(s)\nonumber\\
&=\sum_{\substack{z_1<s<r<q<p<z_2\\ z_1<p q<z_2\\  p r q, p q s, p r s, q r s\notin[z_2,z_3] \\ p q r s\notin[z_2,z_3]\cup[z_5,z_6]}}S_{p q r s}(s)+o\Bigl(\frac{\#\mathcal{A}{}}{\log{X}}\Bigr)\nonumber\\
&\ge-(1+o(1))\frac{\kappa_{\mathcal{A}}\#\mathcal{A}{}}{\log{X}}
\iiiint\limits_{(u,v,w,t)\in\mathcal{R}_4}\omega\Bigl(\frac{1-u-v-w-t}{t}\Bigr)\frac{d u d v d w d t}{u v w t^2},\label{eq:I6}
\end{align}
\endgroup
where
\begin{align*}
\mathcal{R}_4=\Bigl\{(u,v,w,t):\, &\theta_2-\theta_1<t<w<v<u<\theta_1,\,u+v<\theta_1,\\
&u+v+w+t\notin[\theta_1,\theta_2]\cup[1-\theta_2,1-\theta_1],\\
&\{u+v+w,u+v+t,u+w+t,v+w+t\}\cap[\theta_1,\theta_2]=\emptyset\Bigr\}.
\end{align*}
We note that for $\mathcal{R}_4$ we have dropped different constraints to those we dropped in $\mathcal{R}_3$. 

Together \eqref{eq:I1}, \eqref{eq:I2}, \eqref{eq:I3}, \eqref{eq:I4}, \eqref{eq:I5} and \eqref{eq:I6} give our lower bound for all the terms occurring in \eqref{eq:Decomp1}, and so gives a lower bound for first term from \eqref{eq:Decomp} which covers all terms with $p\le z_2$. 

We are left to consider the second term from \eqref{eq:Decomp}, which is the remaining terms with $p\in(z_3,z_4]$. We treat these in a similar manner to those with $p\le z_2$. We first split the sum according to the size of $q p$. Terms with $q p\in [z_5,z_6]$ are negligible by Proposition \ref{prpstn:FinalTypeII}, so we are left to consider $q p\in(z_3,z_5)$ or $q p>z_6$. We then split the terms with $q p\in(z_3,z_5)$ according to the size of $q^2 p$ compared with $z_6$. This gives
\[
\sum_{\substack{z_3<p\le z_4 \\ z_1<q\le (X/p)^{1/2}}}S_{p q}(q)
=S_1+S_2+S_3+o\Bigl(\frac{\#\mathcal{A}{}}{\log{X}}\Bigr),
\]
where
\begingroup
\allowdisplaybreaks
\begin{align}
S_1
&=\sum_{\substack{z_3<p\le z_4 \\ z_1<q\le (X/p)^{1/2}\\ z_6<q p}}S_{p q}(q)\nonumber\\
&\ge-(1+o(1))\frac{\kappa_{\mathcal{A}}\#\mathcal{A}{}}{\log{X}}
\iint\limits_{\substack{\theta_2<u<1/2 \\ \theta_2-\theta_1<v<(1-u)/2 \\1-\theta_1<u+v}}\omega\Bigl(\frac{1-u-v}{v}\Bigr)\frac{d u d v}{u v^2},
\label{eq:I7}\\
S_2
&=\sum_{\substack{z_3<p\le z_4 \\ z_1<q\le (X/p)^{1/2}\\ z_3<q p<z_5\\ z_6\le q^2 p}}S_{p q}(q)\nonumber\\
&\ge -(1+o(1))\frac{\kappa_{\mathcal{A}}\#\mathcal{A}{}}{\log{X}}
\iint\limits_{\substack{\theta_2<u<1/2\\ \theta_2-\theta_1<v<(1-u)/2\\ \theta_2<u+v<1-\theta_2\\ 1-\theta_1<2v+u}}\omega\Bigl(\frac{1-u-v}{v}\Bigr)\frac{d u d v}{u v^2},
\label{eq:I8}
\end{align}
\endgroup
%
and where
\begin{align}
S_3=\sum_{\substack{z_3<p\le z_4\\ z_1<q\le (X/p)^{1/2}\\ z_3<q p<z_5 \\ q^2p <z_6}}S_{p q}(q).\nonumber
\end{align}
We apply two further Buchstab iterations to $S_3$ (we can handle the intermediate terms using Proposition \ref{prpstn:FinalSieve} as before since $q^2p<z_6$). As before, we may replace $S_{p q}(q)$ by $S_{p q}(\min(q,(X/p q)^{1/2}))$ and $S_{p q r}(r)$ by $S_{p q r}(\min(r,(X/p q r)^{1/2}))$ at the cost of a negligible error term (since $p q r<z_6$). This gives
\begingroup
\allowdisplaybreaks
\begin{align}
S_3&=\sum_{\substack{z_3<p\le z_4\\ z_1<q\le (X/p)^{1/2}\\ z_3<q p<z_5 \\ q^2p <z_6}}S_{p q}(\min(q,(X/p q)^{1/2}))+o\Bigl(\frac{\#\mathcal{A}}{\log{X}}\Bigr)\nonumber\\
&=\sum_{\substack{z_3<p\le z_4\\ z_1<q\le (X/p)^{1/2} \\ q^2p <z_6\\ z_3<q p<z_5}}S_{p q}(z_1)-\sum_{\substack{z_3<p\le z_4\\ z_1<r\le q\le (X/p)^{1/2} \\ q^2p <z_6\\ z_3<q p<z_5\\ r^2q p\le X}}S_{p q r}(\min(r,(X/p q r)^{1/2}))%
+o\Bigl(\frac{\#\mathcal{A}}{\log{X}}\Bigr)\nonumber\\
&=o\Bigl(\frac{\#\mathcal{A}}{\log{X}}\Bigr)-\sum_{\substack{z_3<p\le z_4\\ z_1<r\le q\le (X/p)^{1/2} \\ q^2p <z_6\\ z_3<q p<z_5\\ r^2q p\le X}}S_{p q r}(z_1)%
+\sum_{\substack{z_3<p\le z_4\\ z_1<s\le r\le q\le (X/p)^{1/2} \\ q^2p <z_6\\ z_3<q p<z_5\\ s^2q r p, r^2 q p\le X}}S_{p q r s}(s)\nonumber\\
&=\sum_{\substack{z_3<p<z_4\\ z_1<s\le r\le q\le (X/p)^{1/2} \\ q^2p<z_6 \\ z_3<q p<z_5\\ s^2q r p, r^2 q p\le X \\ p q, p r, p s, q r, q s, r s\notin[z_2,z_3]\cup[z_5,z_6]}}S_{p q r s}(s)+o\Bigl(\frac{\#\mathcal{A}{}}{\log{X}}\Bigr)\nonumber\\
&\ge -(1+o(1))\frac{\kappa_{\mathcal{A}}\#\mathcal{A}{}}{\log{X}}\iiiint\limits_{(u,v,w,t)\in\mathcal{R}_5}\omega\Bigl(\frac{1-u-v-w-t}{t}\Bigr)\frac{d u d v d w d t}{ u v w t^2},\label{eq:I9}
\end{align}
\endgroup
where
\begin{align*}
\mathcal{R}_5=\Bigl\{(u,v,w,t):\, &\theta_2-\theta_1<t<w<v,\, \theta_2<u<1/2,\, u+2v<1-\theta_1,\\
&u+v+2w<1,\, u+v+w+2t<1,\, \theta_2<u+v<1-\theta_2,\\
&\{u+v,u+w,u+t,v+w,v+t,w+t\}\notin[\theta_1,\theta_2]\,\Bigr\}.
\end{align*}
Together \eqref{eq:I7}, \eqref{eq:I8}, \eqref{eq:I9} give our lower bound for the second term from \eqref{eq:Decomp}, which is all the terms with $p\in[z_3,z_4]$. This completes our lower bound for $S_1(z_4)$.

Let $I_1,\dots,I_9$ denote the integrals in \eqref{eq:I1}, \eqref{eq:I2}, \eqref{eq:I3}, \eqref{eq:I4}, \eqref{eq:I5}, \eqref{eq:I6}, \eqref{eq:I7}, \eqref{eq:I8} and \eqref{eq:I9} respectively. Putting everything together, we obtain
\begin{align*}
\#\{p\in\mathcal{A}{}\}&=(1+o(1))\frac{\kappa_{\mathcal{A}}\#\mathcal{A}{}}{\log{X}}+S_1(z_4)\\
&\ge (1+o(1))\frac{\kappa_{\mathcal{A}}\#\mathcal{A}{}}{\log{X}}(1-I_1-I_2-I_3-I_4-I_5-I_6-I_7-I_8-I_9).
\end{align*}
In particular, we have
\begin{equation}
\#\{p\in\mathcal{A}{}\}\ge (1+o(1))\frac{\kappa_{\mathcal{A}}\#\mathcal{A}{}}{1000\log{X}}
\label{eq:LowerBound}
\end{equation}
provided that $I_1+\dots+I_9\le 0.999$. Numerical integration\footnote{A Mathematica\textregistered\;file detailing this computation is included with this article on \url{arxiv.org}.} then gives the following bounds on $I_1,\dots,I_9$ in the case when $\theta_1$ and $\theta_2$ in the definition of $I_1,\dots,I_9$ are replaced by $9/25$ and $17/40$ respectively.
\begin{align*}
I_1 & \le 0.02895, & I_2 & \le 0.35718, \\
I_3 & \le 0.01402, & I_4 & \le 0.04238, \\
I_5 & \le 0.05547,& I_6 & \le 0.06622,\\
I_7 & \le 0.21879,& I_8 & \le 0.20339,\\
I_9 & \le 0.00924.
\end{align*}
Thus in this case we have $I_1+\dots +I_9< 0.996$, and so by continuity we have $I_1+\dots +I_9< 0.996+O(\epsilon)$ when $\theta_1=9/25+2\epsilon$ and $\theta_2=17/40-2\epsilon$. Thus, taking $\epsilon$ suitably small, we see that \eqref{eq:LowerBound} holds, and so we have completed the proof of Theorem \ref{thrm:MainTheorem} for $X$ sufficiently large. If $X\ge 4$ is bounded by a constant, then Theorem \ref{thrm:MainTheorem} follows (after potentially adjusting the implied constants) on noting that either 2 or 3 is a prime in $\mathcal{A}$ and so Theorem \ref{thrm:MainTheorem} also holds for bounded $X\ge 4$.
\end{proof}
%
%
%
%
We note that there are various ways in which one can improve the numerical estimates, but we have restricted ourselves to the above decomposition in the interests of clarity. Judiciously employing further Buchstab decompositions would give small numerical improvements, for example.

Thus it suffices to establish Propositions \ref{prpstn:FinalSieve} and \ref{prpstn:FinalTypeII}.
%
%
%
%
%
%
%
%
%
\section{Sieve Asymptotics} \label{sec:Sieve}
In this section we prove Proposition \ref{prpstn:FinalSieve} and Proposition \ref{prpstn:FinalTypeII} assuming Proposition \ref{prpstn:TypeI} and Proposition \ref{prpstn:TypeII}, given below. This reduces the problem to proving standard `Type I' and `Type II' estimates. These propositions will then be proven in Sections \ref{sec:TypeI} and \ref{sec:TypeII}.

Before we state the propositions, we set up some extra notation. Let
\[
\mathcal{Q}_\ell(\eta)=\{(x_1,\dots,x_\ell)\in\mathbb{R}^\ell:\,\eta\le x_1\le \dots \le x_\ell,\,x_1+\dots+x_\ell=1\}.
\]
By a closed convex polytope in $\mathbb{R}^\ell$ we mean a region $\mathcal{R}$ defined by a finite number of non-strict affine linear inequalities in the coordinates (equivalently, this is the convex hull of a finite set of points in $\mathbb{R}^\ell$). Given a closed convex polytope $\mathcal{R}\subseteq \mathcal{Q}_\ell(\eta)$, we let
\[
\mathbf{1}_{\mathcal{R}}(a)=
\begin{cases}
1,\qquad &\text{if }a=p_1\cdots p_{\ell}\text{ for some $p_1,\dots,p_\ell$ with }\Bigl(\frac{\log{p_1}}{\log{a}},\dots,\frac{\log{p_\ell}}{\log{a}}\Bigr)\in\mathcal{R},\\
0,&\text{otherwise.}
\end{cases}
\]
We caution that $\mathbf{1}_{\mathcal{R}}$ counts numbers with a particular type of prime factorization, and should not be confused with $\mathbf{1}_{\mathcal{A}}$, the indicator function of the set $\mathcal{A}$. We recall $\mathcal{B}=\{n\in\mathbb{Z}:\, 0\le n< X\}$.

Our two key propositions that we will use are given below.
%
%
%
%
\begin{prpstn}[Type I estimate]\label{prpstn:TypeI}
Let $A>0$ and $Q\le X^{50/77}(\log{X})^{-2A-2}$.  Then we have
\[
\sum_{\substack{q<Q\\ (q,10)=1}}\Bigl|\#\{a\in\mathcal{A}:\,q|a,\,(a,10)=1\}-\kappa\frac{\#\mathcal{A}}{q}\Bigr|\ll_A \frac{\#\mathcal{A}}{(\log{X})^A},
\]
where
\[
\kappa=\begin{cases}
\frac{\phi(10)}{9},\qquad&\text{if }(a_0,10)\ne1,\\
\frac{\phi(10)-1}{9},&\text{if }(a_0,10)=1.
\end{cases}
\]
\end{prpstn}
%
%
%
%
\begin{prpstn}[Type II estimate] \label{prpstn:TypeII}
Let $\eta>0$, and let $\ell\le 2\eta^{-1}$. Let $\mathcal{R}\subseteq \mathcal{Q}_\ell(\eta)$ be a closed convex polytope in $\mathbb{R}^\ell$ which has the property that
\[
\mathbf{e}\in\mathcal{R}\Rightarrow \sum_{i\in\mathcal{I}} e_i\in\Bigl[\frac{9}{25}+\epsilon,\frac{17}{40}-\epsilon\Bigr]
\]
for some set $\mathcal{I}\subseteq\{1,\dots,\ell\}$. Then we have
\[
\sum_{\substack{a\in\mathcal{A}}}\mathbf{1}_{\mathcal{R}}(a)=\kappa_\mathcal{A}\frac{\#\mathcal{A}{}}{\#\mathcal{B}{}}\sum_{n<  X}\mathbf{1}_{\mathcal{R}}(n)+O_{\mathcal{R},\eta}\Bigl(\frac{\#\mathcal{A}}{\log{X}\log\log{X}}\Bigr),
\]
where
\[
\kappa_\mathcal{A}=
\begin{cases}
\frac{10(\phi(10)-1)}{9\phi(10)},\qquad&\text{if $(10,a_0)=1$,}\\
\frac{10}{9},&\text{otherwise.}\\
\end{cases}
\]
\end{prpstn}
%
%
%
%
Proposition \ref{prpstn:FinalTypeII} follows quickly from Proposition \ref{prpstn:TypeII}, but it will be convenient to establish a slightly more general version where the primes can be as small as $X^\eta$. 
\begin{lmm}[Type II terms, alternative formulation]\label{lmm:AltTypeII}
Fix an integer $\ell\ge1$ and a quantity $\eta>0$. Let $\theta_1=9/25+2\epsilon$, $\theta_2=17/40-2\epsilon$, and $\mathcal{L}$ be as in Proposition \ref{prpstn:FinalTypeII}, and let $\mathcal{I}\subseteq\{1,\dots,\ell\}$ and $j\in\{1,\dots,\ell\}$. Then we have
\[
\sum_{\substack{X^{\eta}\le p_1\le \dots\le p_\ell \\ X^{\theta_1}\le\prod_{i\in\mathcal{I}}p_i\le X^{\theta_2}\\ p_1\cdots p_\ell\le X/p_j}}^*S_{p_1\cdots p_\ell}(p_j)=o_{\mathcal{L},\eta}\Bigl(\frac{\#\mathcal{A}}{\log{X}}\Bigr),
\]
and
\[
\sum_{\substack{X^{\eta}\le p_1\le \dots\le p_\ell \\ X^{1-\theta_2}\le\prod_{i\in\mathcal{I}}p_j\le X^{1-\theta_1}\\ p_1\cdots p_\ell\le X/p_j}}^*S_{p_1\cdots p_\ell}(p_j)=o_{\mathcal{L},\eta}\Bigl(\frac{\#\mathcal{A}}{\log{X}}\Bigr),
\]
where $\sum^*$ indicates the same restriction of summation to $L\ge 0$ for all $L\in\mathcal{L}$ as in Proposition \ref{prpstn:FinalTypeII}.
\end{lmm}
As before, we note that by inclusion-exclusion the same result holds if some of the constraints $L\ge 0$ are replaced with $L>0$. We see Proposition \ref{prpstn:FinalTypeII} follows immediately from Lemma \ref{lmm:AltTypeII} on choosing $\eta=\theta_2-\theta_1$.
%
%
\begin{proof}[Proof of Lemma \ref{lmm:AltTypeII} assuming Proposition \ref{prpstn:TypeII}]
We just deal with the case when $\prod_{i\in\mathcal{I}}p_i\in[X^{\theta_1},X^{\theta_2}]$; the other case is entirely analogous with $\theta_1$ and $\theta_2$ simply replaced with $1-\theta_2$ and $1-\theta_1$ throughout. (Notice that if $\mathbf{e}\in\mathcal{R}\subseteq\mathcal{Q}_\ell(\eta)$ satisfies $\sum_{i\in\mathcal{I}}e_i\in[23/40+\epsilon,16/25-\epsilon]$, then $\sum_{i\notin\mathcal{I}}e_i\in[9/25+\epsilon,17/40-\epsilon]$. Thus the interval $[9/25+\epsilon,17/40-\epsilon]$ in Proposition \ref{prpstn:TypeII} can be replaced by the interval $[23/40+\epsilon,16/25-\epsilon]$, and so Proposition \ref{prpstn:TypeII} applies similarly in both cases.) 

Recall the definition \eqref{eq:SWDef} of $S_{d}(z)$. We see that $S_{p_1\cdots p_\ell}(p_j)$ is a sum of $w_n$ only involving integers $n$ with at most $1/\eta$ prime factors, since all prime factors are of size at least $ X^{\eta}$. The terms with exactly $r$ prime factors (for some $r\le 1/\eta$) are a sum of $w_{p_1\cdots p_r}$ over $p_1,\dots,p_r$ with the summation only restricted by a bounded number of linear inequalities on $\log{p_1}/\log{X},\dots,\log{p_r}/\log{X}$. (These are the previous restrictions on $p_1,\dots,p_\ell$, and the restriction $p_j\le p_{\ell+1}\le \dots \le p_r$). We may write the condition $X^{\eta}\le p_1$ and the restriction on the size of $\prod_{i\in\mathcal{I}}p_i$ and $\prod_{i=1}^\ell p_i$ as linear conditions only involving $\log{p_1}/\log{X},\dots,\log{p_\ell}/\log{X}$ with coefficients having constants depending only on $\eta$. Thus, after increasing $\mathcal{L}$ to include these conditions, it suffices to show that
\begin{equation}
\sum_{\substack{p_1\le \dots\le p_\ell \\ p_j\le p_{\ell+1}\le \dots \le p_r}}^* w_{p_1\cdots p_r}=o_{\mathcal{L},\eta}\Bigl(\frac{\#\mathcal{A}}{\log{X}}\Bigr),\label{eq:FactorSum}
\end{equation}
where $\sum^*$ indicates that the summation is restricted by the conditions
\begin{equation}
L\Bigl(\frac{\log{p_1}}{\log{X}},\dots ,\frac{\log{p_\ell}}{\log{X}}\Bigr)\ge 0\label{eq:LConstraints}
\end{equation}
for all $L\in\mathcal{L}$.

Let $\delta=1/\log\log{X}$. We first trivially discard the contribution from $n=p_1\cdots p_{r}<X^{1-\delta}$. Each $n$ appears $O_\eta(1)$ times in \eqref{eq:FactorSum}, so recalling the definition \eqref{eq:WDef} of $w_n$ and dropping the other constraints, the total contribution from such terms is
\begin{equation}
\ll_\eta \sum_{\substack{n\in \mathcal{A}\\ n<X^{1-\delta}}}1+\frac{\#\mathcal{A}}{\#\mathcal{B}}\sum_{n<X^{1-\delta}}1\ll \#\mathcal{A}^{1-\delta}+\frac{\#\mathcal{A}}{X^\delta}=o_\eta\Bigl(\frac{\#\mathcal{A}}{\log{X}}\Bigr).\label{eq:SmallTerms}
\end{equation}
Thus it is sufficient to show 
\begin{equation}
\sum_{\substack{p_1\le \dots\le p_\ell \\ p_j\le p_{\ell+1}\le \dots \le p_r\\ p_1\cdots p_r\ge X^{1-\delta}}}^* w_{p_1\cdots p_r}=o_{\mathcal{L},\eta}\Bigl(\frac{\#\mathcal{A}}{\log{X}}\Bigr).\label{eq:FactorSum2}
\end{equation}
Since we have the constraint $p_1\cdots p_\ell \le X/p_j\le X^{1-\eta}$, the result follows immediately if $r=\ell$ (if $\eta<\delta$ the result is trivial). Thus we may assume that $r>\ell$, so none of the constraints involve all the $p_i$. We now wish to replace $\log{p_i}/\log{X}$ with $\log{p_i}/\log{n}$ in the conditions \eqref{eq:LConstraints}. For $n\in[X^{1-\delta},X]$, we have
\[
\frac{\log{p_i}}{\log{X}}\le \frac{\log{p_i}}{\log{n}}\le (1+2\delta)\frac{\log{p_i}}{\log{X}},
\]
and so if exactly one of $L\Bigl(\frac{\log{p_1}}{\log{X}},\dots ,\frac{\log{p_\ell}}{\log{X}}\Bigr)$ and $L\Bigl(\frac{\log{p_1}}{\log{n}},\dots ,\frac{\log{p_\ell}}{\log{n}}\Bigr)$ is non-negative, we must have
\begin{equation}
\Bigr|L\Bigl(\frac{\log{p_1}}{\log{n}},\dots ,\frac{\log{p_\ell}}{\log{n}}\Bigr)\Bigl|\ll_\mathcal{L}\delta\label{eq:LSmall}.
\end{equation}
To bound the contribution of such terms, let $\gamma>0$ be a parameter and
\[
G(\gamma,L):=\sum_{\substack{n^{\eta}\le p_1,\dots,\,p_r \\ -\gamma\le L(\frac{\log{p_1}}{\log{n}},\dots ,\frac{\log{p_\ell}}{\log{n}})\le \gamma\\ n^{\theta_1}\le\prod_{i\in\mathcal{I}}p_i\le n^{\theta_2+\epsilon} }}\Bigl(1_{\mathcal{A}}(p_1\cdots p_r)+\frac{\#\mathcal{A}}{\#\mathcal{B}}1_{\mathcal{B}}(p_1\cdots p_r)\Bigr).
\]
(Here the summation is over all choices of primes $p_1,\dots,p_r$, and for any such choice $n=p_1\cdots p_r$. We do not restrict to $n\ge X^{1-\delta}$ in the summation.) We wish to show that if $\gamma=o_{L,\eta}(1)$ then $G(\gamma,L)=o_{L,\eta}(\#\mathcal{A}/\log{X})$, and we will do this by first thinking of $\gamma$ fixed but very small.

We split the sum into at most $r!=O_\eta(1)$ subsums where the variables are ordered (we potentially double-count the contribution from $p_i=p_{i'}$ for an upper bound). Thus, after relabelling the $p_i$, we see that
\[
G(\gamma,L)\ll \sup_{\substack{i_1,\dots,i_\ell\in\{1,\dots,r\}\\ \text{distinct}}}\sum_{\substack{n^{\eta}\le p_1\le \dots\le p_r \\ -\gamma\le L(\frac{\log{p_{i_1}}}{\log{n}},\dots ,\frac{\log{p_{i_\ell}}}{\log{n}})\le \gamma\\ n^{\theta_1}\le\prod_{i\in\mathcal{I}'}p_i\le n^{\theta_2+\epsilon}}}\Bigl(1_{\mathcal{A}}(p_1\cdots p_r)+\frac{\#\mathcal{A}}{\#\mathcal{B}}1_{\mathcal{B}}(p_1\cdots p_r)\Bigr)
\]
for some set $\mathcal{I}'\subseteq\{1,\dots,r\}$. Let $\mathcal{R}=\mathcal{R}(\gamma,L,\eta)\subseteq\mathcal{Q}_{r}(\eta)$ be given by
\[
\mathcal{R}=\Bigl\{(x_1,\dots,x_r)\in \mathcal{Q}_r(\eta):\, -\gamma\le L(x_{i_1},\dots,x_{i_\ell})\le \gamma,\,\sum_{i\in\mathcal{I}'}x_i\in [\theta_1,\theta_2+\epsilon]\Bigr\}.
\]
Then $\mathcal{R}$ satisfies the conditions of Proposition \ref{prpstn:TypeII}, so 
\begin{align*}
\sum_{\substack{n^{\eta}\le p_1\le \dots\le p_r \\ -\gamma\le L(\frac{\log{p_{i_1}}}{\log{n}},\dots ,\frac{\log{p_{i_\ell}}}{\log{n}})\le \gamma\\ n^{\theta_1}\le\prod_{i\in\mathcal{I}'}p_i\le n^{\theta_2+\epsilon}}}1_{\mathcal{A}}(p_1\cdots p_r)&=\sum_{n\in\mathcal{A}}\mathbf{1}_{\mathcal{R}}(n)\\
&=\frac{\#\mathcal{A}}{\#\mathcal{B}}\sum_{n<X}\mathbf{1}_{\mathcal{R}}(n)+o_\mathcal{R}\Bigl(\frac{\#\mathcal{A}}{\log{X}\log\log{X}}\Bigr).
\end{align*}
Thus
\[
G(\gamma,L)\ll  \frac{\#\mathcal{A}}{\#\mathcal{B}}\sup_{\substack{i_1,\dots,i_\ell\in\{1,\dots,r\}\\ \text{distinct}}}\sum_{n<X}\mathbf{1}_{\mathcal{R}}(n)+O_{L,\eta,\gamma}\Bigl(\frac{\#\mathcal{A}}{\log{X}\log\log{X}}\Bigr).
\]
By the Prime Number Theorem and partial summation, we have
\[
\sum_{n<X}\mathbf{1}_{\mathcal{R}}(n)=\frac{X}{\log{X}}\idotsint\limits_{(e_1,\dots,e_r)\in\mathcal{R}}\frac{d e_1\dots d e_{r-1}}{\prod_{i=1}^{r} e_i}+O_{\mathcal{R}}\Bigl(\frac{X}{(\log{X})^2}\Bigr).
\]
Since all components of elements of $\mathcal{R}$ are at least $\eta$, the integral is bounded by $\eta^{-r}$ times the $(r-1)-$dimensional volume of $\mathcal{R}$. Since $L$ involves at most $\ell\le r-1$ coordinates and $\mathcal{R}\subseteq [\eta,1]^r$, this volume is $O_{L,\eta}(\gamma)$.
Thus 
\[
G(\gamma,L)=O_{L,\eta}\Bigl(\frac{\gamma\#\mathcal{A}}{\log{X}}\Bigr)+O_{L,\eta,\gamma}\Bigl(\frac{\#\mathcal{A}}{\log{X}\log\log{X}}\Bigr).
\]
If $\gamma\rightarrow 0$ as $X\rightarrow\infty$ suitably slowly, we see that this shows that $G(\gamma,L)=o_{L,\eta}(\#\mathcal{A}/\log{X})$. But from the definition of $G$, we see that $G(\gamma,L)$ is non-decreasing in $\gamma$, so in fact we deduce that for any $\gamma=o_{L,\eta}(1)$ we have $G(\gamma,L)=o_{L,\eta}(\#\mathcal{A}/\log{X})$.

We see from \eqref{eq:LSmall} that the error introduced to \eqref{eq:FactorSum2} by replacing $\log{p_i}/\log{X}$ with $\log{p_i}/\log{n}$ in the conditions \eqref{eq:LConstraints} is $O(\sum_{L\in\mathcal{L}}G(\gamma,L))$ for some $\gamma\ll_\mathcal{L}\delta=o_\mathcal{L}(1)$. By the above discussion, this is $o_{\mathcal{L},\eta}(\#\mathcal{A}/\log{X})$, which is negligible.

After making this change, we may reintroduce the terms with $n<X^{1-\delta}$ at the cost of a negligible error by using the bound \eqref{eq:SmallTerms} again.
Thus
\[
\sum_{\substack{p_1\le \dots\le p_\ell \\ p_j\le p_{\ell+1}\le \dots \le p_r\\ p_1\cdots p_r \ge X^{1-\delta}}}^* w_{p_1\cdots p_r}=\sum_{\substack{p_1\le \dots\le p_\ell \\ p_j\le p_{\ell+1}\le \dots \le p_r}}^{**} w_{p_1\cdots p_r}+o_{\mathcal{L},\eta}\Bigl(\frac{\#\mathcal{A}}{\log{X}}\Bigr),
\]
where $\sum^{**}$ indicates the sum is constrained to
\[
L\Bigl(\frac{\log{p_1}}{\log{n}},\dots,\frac{\log{p_\ell}}{\log{n}}\Bigr)\ge 0
\]
for all $L\in\mathcal{L}$. Moreover, since we had the constraint $\prod_{i\in\mathcal{I}}p_i\in [X^{\theta_1},X^{\theta_2}]$ in \eqref{eq:LConstraints}, this second sum includes the constraint $\prod_{i\in\mathcal{I}}p_i\in [n^{\theta_1},n^{\theta_2}]$. We now split the summation into $O_\eta(1)$ subsums where the $p_i$ are totally ordered. After relabelling the coordinates, Proposition \ref{prpstn:TypeII} applies to each of these sums, since the linear constraints $L\ge 0$ for $L\in\mathcal{L}$ define a closed convex polytope (depending only on $\mathcal{L}$), and the ordering of the variables ensures that this lies within $\mathcal{Q}_r(\eta)$ (recall that the constraint $X^\eta\le p_1$ becomes $n^\eta\le p_1$, so all primes are at least $n^\eta$). The constraint $\prod_{i\in\mathcal{I}}p_i\in [n^{\theta_1},n^{\theta_2}]$ corresponds to the sum of a subset of the coordinates of all points in the polytope lying in $[\theta_1,\theta_2]$. Proposition \ref{prpstn:TypeII} shows that the contribution from each such sum is $o_{\mathcal{L},\eta}(\#\mathcal{A}/\log{X})$. Since there are $O_\eta(1)$ such sums, the total contribution is $o_{\mathcal{L},\eta}(\#\mathcal{A}/\log{X})$, giving the result.
\end{proof}
%
%
%
%
Our aim for the remainder of this section is to establish Proposition \ref{prpstn:FinalSieve} using Proposition \ref{prpstn:TypeI} and Proposition \ref{prpstn:TypeII}. %
%
%
%
%
%
%
%
We first establish an auxiliary lemma.%
%
%
%
%
\begin{lmm}[Fundamental Lemma]\label{lmm:FundamentalLemma}
For $\delta>0$ we have
\[
\sum_{\substack{d<X^{50/77-\epsilon}\\ p|d\Rightarrow p>X^{\delta} }}\Bigl|S(\mathcal{A}{}_{d}, X^{\delta})-\frac{\kappa_\mathcal{A}\#\mathcal{A}{}}{\#\mathcal{B}{}}S(\mathcal{B}{}_d,X^{\delta})\Bigr|
\ll \frac{\exp(-\delta^{-2/3})}{\log{X}}\#\mathcal{A}+\frac{\#\mathcal{A}}{(\log{X})^{100}}.
\]
The implied constant is independent of $\delta$.
\end{lmm}
%
%
%
%
\begin{proof}[Proof of Lemma \ref{lmm:FundamentalLemma} assuming Proposition \ref{prpstn:TypeI}]
If $\delta>\epsilon^4$ then since $S(\mathcal{C},X^t)$ is nonnegative and decreasing in $t$ for any set $\mathcal{C}$, we have
\begin{align*}
-\kappa_\mathcal{A}\frac{\#\mathcal{A}{}}{\#\mathcal{B}}S(\mathcal{B}{}_d,X^{\epsilon^4})&\le S(\mathcal{A}{}_d,X^\delta)-\kappa_\mathcal{A}\frac{\#\mathcal{A}{}}{\#\mathcal{B}{}}S(\mathcal{B}{}_d,X^\delta)\\
&\le \Bigl(S(\mathcal{A}{}_d,X^{\epsilon^4})-\frac{\kappa_\mathcal{A}\#\mathcal{A}{}}{\#\mathcal{B}{}}S(\mathcal{B}{}_d,X^{\epsilon^4})\Bigr)+\kappa_\mathcal{A}\frac{\#\mathcal{A}{}}{\#\mathcal{B}{}}S(\mathcal{B}{}_d,X^{\epsilon^4}).
\end{align*}
Since $S(\mathcal{B}{}_d,X^{\epsilon^4})\ll X/(d\log{X})$ for $d<X^{1-\epsilon}$ by \eqref{eq:Buchstab}, this gives
\[
\Bigl| S(\mathcal{A}{}_d,X^\delta)-\kappa_\mathcal{A}\frac{\#\mathcal{A}{}}{\#\mathcal{B}{}}S(\mathcal{B}{}_d,X^\delta)\Bigr|=\Bigl|S(\mathcal{A}{}_d,X^{\epsilon^4})-\kappa_\mathcal{A}\frac{\#\mathcal{A}{}}{\#\mathcal{B}{}}S(\mathcal{B}{}_d,X^{\epsilon^4})\Bigr|+O\Bigl(\frac{\#\mathcal{A}{}}{d\log{X}}\Bigr).
\]
By the rough number estimate \eqref{eq:Buchstab} again, we see that the sum of $1/d$ over $d<X$ with all prime factors bigger that $X^\delta$ is $O_\delta(1)$. Thus the result for $\delta>\epsilon^4$ follows from the result for $\delta=\epsilon^4$, so we may assume without loss of generality that $\delta\le \epsilon^4$.

Let
\[
\mathcal{A}'=\{a\in\mathcal{A}{}:\,(a,10)=1\}.
\]
Then $\#\mathcal{A}'=\kappa\#\mathcal{A}{}$, where $\kappa$ is the constant given in Proposition \ref{prpstn:TypeI}. Let $R_d(e)$ be defined by
\[
\#\{a\in\mathcal{A}_{d}':\, e|a\}=\frac{\kappa\#\mathcal{A}{}}{d e}+R_d(e).
\]
We put $q=d e$ and see from Proposition \ref{prpstn:TypeI} that for any $A>0$ the error terms $R_{d}(e)$ satisfy
\begin{align}
\sum_{\substack{d<X^{50/77-\epsilon}\\ p|d\Rightarrow p>X^{\delta}}}\sum_{\substack{e<X^{\epsilon/2}\\ (e,10)=1\\ p|e\Rightarrow p\le X^{\delta}}}R_d(e)
&\ll\sum_{\substack{q<X^{50/77-\epsilon/2}\\ (q,10)=1}}\Bigl|\#\mathcal{A}'_{q}-\frac{\kappa\#\mathcal{A}{}}{q}\Bigr|\nonumber\\
&\ll_A \frac{\#\mathcal{A}}{(\log{X})^A}.
\label{eq:SieveErrorBound}
\end{align}
By the fundamental lemma of sieve methods (see, for example, \cite[Theorem 6.9]{FriedlanderIwaniec}) we have
\[
S(\mathcal{A}'_d,X^\delta)=\Bigl(1-O\Bigl(\exp\Bigl(\frac{-\epsilon}{2\delta}\Bigr)\Bigr)\Bigr)\frac{\kappa\#\mathcal{A}{}}{d}\prod_{\substack{p\le X^\delta \\ p\nmid 10}}\Bigl(1-\frac{1}{p}\Bigr)+O\Bigl( \sum_{\substack{e<X^{\epsilon/2}\\ (e,10)=1 \\ p|e\Rightarrow p\le X^\delta}}R_d(e)\Bigr).
\]
Summing over $d$ and using the bound \eqref{eq:SieveErrorBound}, we obtain 
\begin{align*}
\sum_{\substack{d<X^{50/77-\epsilon}\\ p|d\Rightarrow p>X^{\delta}}}\Bigl|&S(\mathcal{A}'_{d},X^{\delta})-\frac{\kappa\#\mathcal{A}{}}{d}\prod_{\substack{p\le X^{\delta}\\ p\nmid 10}}\Bigl(1-\frac{1}{p}\Bigr)\Bigr|\\
&\ll \exp\Bigl(-\frac{\epsilon}{2\delta}\Bigr)\prod_{\substack{p\le X^{\delta}\\ p\nmid 10}}\Bigl(1-\frac{1}{p}\Bigr)\#\mathcal{A}{}\sum_{\substack{d<X^{50/77-\epsilon}\\ p|d\Rightarrow p>X^{\delta}}}\frac{1}{d}+\frac{\#\mathcal{A}}{(\log{X})^{100}}.
\end{align*}
The product in the final bound is $O(\delta^{-1}(\log{X})^{-1})$, and the inner sum over $d$ is seen to be $O(\delta^{-1})$ by an Euler product upper bound. Finally, since we are assuming that $\delta\le \epsilon^4$, we have that $\delta^{-2}\exp(-\epsilon/(2\delta))\ll \exp(-\delta^{-2/3})$. Thus
\begin{equation}
\sum_{\substack{d<X^{50/77-\epsilon}\\ p|d\Rightarrow p>X^{\delta}}}\Bigl|S(\mathcal{A}'_{d},X^{\delta})-\frac{\kappa\#\mathcal{A}{}}{d}\prod_{\substack{p\le X^{\delta}\\ p\nmid 10}}\Bigl(1-\frac{1}{p}\Bigr)\Bigr|\ll \frac{\exp(-\delta^{-2/3})\#\mathcal{A}}{\log{X}}+\frac{\#\mathcal{A}}{(\log{X})^{100}}.
\label{eq:FundSplit}
\end{equation}
An identical argument works for the set $\mathcal{B}'=\{n<  X:\,(n,10)=1\}$ instead of $\mathcal{A}'$. This gives
\begin{equation}
\sum_{\substack{d<X^{50/77-\epsilon}\\ p|d\Rightarrow p>X^{\delta}}}\Bigl|S(\mathcal{B}_{d}',X^{\delta})-\frac{\#\mathcal{B}'}{d}\prod_{\substack{p\le X^\delta \\ p\nmid 10}}\Bigl(1-\frac{1}{p}\Bigr)\Bigr|\ll \frac{\exp(-\delta^{-2/3})\#\mathcal{B}'}{\log{X}}+\frac{\#\mathcal{B}'}{(\log{X})^{100}}.
\label{eq:FundB}
\end{equation}
We see that for $(d,10)=1$ we have $S(\mathcal{A}'_d,X^{\delta})=S(\mathcal{A}{}_d,X^{\delta})$, that $S(\mathcal{B}'_d,X^{\delta})=S(\mathcal{B}{}_d,X^{\delta})$, and that $\#\mathcal{B}'=\phi(10)\#\mathcal{B}{}/10$. Thus, by the triangle inequality
\begin{align*}
\sum_{\substack{d<X^{50/77-\epsilon}\\ p|d\Rightarrow p>X^{\delta}}}&\Bigl|S(\mathcal{A}{}_{d},X^{\delta})-\frac{10\kappa\#\mathcal{A}{}}{\phi(10)\#\mathcal{B}{}}S(\mathcal{B}{}_d,X^{\delta})\Bigr|\\
&\le \sum_{\substack{d<X^{50/77-\epsilon}\\ p|d\Rightarrow p>X^{\delta}}}\Bigl|S(\mathcal{A}_{d}',X^{\delta})-\frac{\kappa\#\mathcal{A}{}}{d}\prod_{\substack{p\le X^\delta \\ p\nmid 10}}\Bigl(1-\frac{1}{p}\Bigr)\Bigr|\\
&\qquad+\frac{10\kappa\#\mathcal{A}{}}{\phi(10)\#\mathcal{B}{}}\sum_{\substack{d<X^{50/77-\epsilon}\\ p|d\Rightarrow p>X^{\delta}}}\Bigl|S(\mathcal{B}_{d}',X^{\delta})-\frac{\#\mathcal{B}'}{d}\prod_{\substack{p\le X^\delta \\ p\nmid 10}}\Bigl(1-\frac{1}{p}\Bigr)\Bigr|\\
&\qquad+\sum_{\substack{d<X^{50/77-\epsilon}\\ p|d\Rightarrow p>X^{\delta}}}\Bigl|\frac{\kappa\#\mathcal{A}{}}{d}\prod_{\substack{p\le X^\delta \\ p\nmid 10}}\Bigl(1-\frac{1}{p}\Bigr)-\frac{10\kappa\#\mathcal{A}{}\#\mathcal{B}'}{\phi(10)d \#\mathcal{B}{}}\prod_{\substack{p\le X^\delta \\ p\nmid 10}}\Bigl(1-\frac{1}{p}\Bigr)\Bigr|.
\end{align*}
We bound the first summation by \eqref{eq:FundSplit}, the second summation by \eqref{eq:FundB}, and note that since $\#\mathcal{B}'=\phi(10)\#\mathcal{B}{}/10$, the third summation is zero. Since $\kappa_\mathcal{A}=10\kappa/\phi(10)$, this gives
\[
\sum_{\substack{d<X^{50/77-\epsilon}\\ p|d\Rightarrow p>X^{\delta}}}\Bigl|S(\mathcal{A}{}_{d},X^{\delta})-\frac{\kappa_\mathcal{A}\#\mathcal{A}{}}{\#\mathcal{B}{}}S(\mathcal{B}{}_d,X^{\delta})\Bigr|\ll \frac{\exp(-\delta^{-2/3})}{\log{X}}\#\mathcal{A}+\frac{\#\mathcal{A}}{(\log{X})^{100}}.
\qedhere
\]
\end{proof}
%
%
%
%
Using Lemma \ref{lmm:FundamentalLemma} we can now prove Proposition \ref{prpstn:FinalSieve}.
%
%
%
%
\begin{proof}[Proof of Proposition \ref{prpstn:FinalSieve} assuming Lemma \ref{lmm:AltTypeII} and Lemma \ref{lmm:FundamentalLemma}]
Recall that $\theta_1=9/25+2\epsilon$, $\theta_2=17/40-2\epsilon$. Let $\theta:=\theta_2-\theta_1$, and let $\delta\ge 1/\log\log{X}$ be a small quantity which we will eventually choose to tend to 0 in a suitable manner. In particular, $\delta$ will be small compared with $\epsilon$. 

We first consider the contribution from $p_1\cdots p_\ell < X^{\theta_1}$. %
Given a set $\mathcal{C}$ and an integer $d$, we let
\begin{align*}
T_m(\mathcal{C};d)&=\sum_{\substack{X^{\delta}<p_m'\le \dots \le p_1'\le X^{\theta}\\ d p_1'\cdots p_m'\le X^{\theta_1}}}S(\mathcal{C}_{p_1'\cdots p_m'},X^{\delta}),\\
U_m(\mathcal{C};d)&=\sum_{\substack{X^{\delta}<p_m'\le \dots \le p_1'\le X^{\theta} \\ d p_1'\cdots p_m'\le X^{\theta_1}}}S(\mathcal{C}_{p_1'\cdots p_m'},p_m'),\\
V_m(\mathcal{C};d)&=\sum_{\substack{X^{\delta}<p_m'\le \dots \le p_1'\le X^{\theta} \\ X^{\theta_1}<d p_1'\cdots p_m'\le X^{\theta_1}p_m' }}S(\mathcal{C}_{p_1'\cdots p_m'},p_m').
\end{align*}
Buchstab's identity shows that
\[
U_m(\mathcal{C};d)=T_m(\mathcal{C};d)-U_{m+1}(\mathcal{C};d)-V_{m+1}(\mathcal{C};d).
\]
 We define $T_0(\mathcal{C};d)=S(\mathcal{C};X^{\delta})$ and $V_0(\mathcal{C};d)=0$. This gives for $d\le X^{\theta_1}$
\[
S(\mathcal{C},X^{\theta})=T_0(\mathcal{C};d)-V_1(\mathcal{C};d)-U_1(\mathcal{C};d)=\sum_{m\ge 0}(-1)^m(T_m(\mathcal{C};d)+V_m(\mathcal{C};d)).
\]
We apply the above decomposition to $\mathcal{A}{}_d$. This gives an expression with $O(\delta^{-1})$ terms since trivially $T_m(\mathcal{A}{}_{d};d)=U_m(\mathcal{A}{}_{d};d)=V_m(\mathcal{A}{}_{d};d)=0$ if $m>1/\delta$. Applying the same decomposition to $\mathcal{B}{}_{d}$, taking the weighted difference, and summing over $d=p_1\cdots p_\ell$ we obtain
\begin{align}
\sideset{}{'}\sum_{p_1,\dots,p_\ell} S(\mathcal{A}{}_{d},X^{\theta})&-\frac{\kappa_\mathcal{A}\#\mathcal{A}{}}{X} \sideset{}{'}\sum_{p_1,\dots,p_\ell} S(\mathcal{B}{}_{d},X^{\theta})\nonumber\\
&\ll \sum_{0\le m\ll 1/\delta}\,\sideset{}{'}\sum_{p_1,\dots,p_\ell} \Bigl|T_m(\mathcal{A}{}_{d};d)-\frac{\kappa_\mathcal{A}\#\mathcal{A}{}}{X} T_m(\mathcal{B}{}_{d};d)\Bigr|\nonumber\\
&+\sum_{0\le m\ll 1/\delta}\Bigl|\sideset{}{'}\sum_{p_1,\dots,p_\ell} \Bigl(V_m(\mathcal{A}{}_{d};d)-\frac{\kappa_\mathcal{A}\#\mathcal{A}{}}{X} V_m(\mathcal{B}{}_{d};d)\Bigr)\Bigr|.
\label{eq:TmVmSplit}
\end{align}
Here $\sum'$ indicates we are summing over all choices of $p_1,\dots,p_\ell$ which appear in the summation in Proposition \ref{prpstn:FinalSieve} with the additional condition that $d=p_1\cdots p_\ell < X^{\theta_1}$.

We note that $p_1,\dots,p_\ell\ge X^\theta$, so $d$ has $O(1)$ prime factors and any integer $e$ can be represented $O(1)$ times as $d p_1'\cdots p_m'$ for some primes $p_m'\le\dots \le p_1'$ and some choice of $p_1,\dots,p_\ell$ defining $d$. Thus, expanding the definition of $T_m$, if $\delta\le \epsilon$ we have
\begin{align}
\sum_{0\le m\ll 1/\delta}\,\sideset{}{'}\sum_{p_1,\dots,p_\ell}&\Bigl|T_m(\mathcal{A}{}_{d};d)-\frac{\kappa_\mathcal{A}\#\mathcal{A}{}}{\#\mathcal{B}{}} T_m(\mathcal{B}{}_d;d)\Bigr|\nonumber\\
&\ll \sum_{\substack{e<X^{\theta_1}\\ p|e\Rightarrow p>X^{\delta} }}\Bigl|S(\mathcal{A}{}_{e}, X^{\delta})-\frac{\kappa_\mathcal{A}\#\mathcal{A}{}}{\#\mathcal{B}{}}S(\mathcal{B}{}_e,X^{\delta})\Bigr|&\nonumber\\
&\ll \frac{\delta^{-1}\exp(-\delta^{-2/3})\#\mathcal{A}}{\log{X}}.
\label{eq:TmSums}
\end{align}
Here we applied by Lemma \ref{lmm:FundamentalLemma} in the last line, using $\delta\ge 1/\log\log{X}$.

We now consider the $V_m$ terms. We expand the definition of $V_m$ as a sum. We note that $p_m'\le X^\theta=X^{\theta_2-\theta_1}$, so the summation is constrained by $X^{\theta_1}\le d p_1'\cdots p_m'\le X^{\theta_2}$, which is our Type II constraint. We see that all terms have $d p_1'\cdots p_m'\le X/p_m'$, so we can insert this condition without changing the sum. We recall $p_1,\dots,p_\ell$ are constrained only by some linear conditions on $\log{p_1}/\log{X},\dots,\log{p_\ell}/\log{X}$. Thus we see that the sum is of the form considered in Lemma \ref{lmm:AltTypeII} with $\eta=\delta$, since all the conditions in the summation can be written as linear constraints on $\log{p_i}/\log{X}$ for $1\le i \le \ell$ and $\log{p_j'}/\log{X}$ for $1\le j\le m$. Thus, by Lemma \ref{lmm:AltTypeII}, we have
\begin{equation}
\sum_{m\ll \delta^{-1}}\Bigl|\sideset{}{'}\sum_{p_1,\dots,p_\ell}\Bigl(V_m(\mathcal{A}{}_d;d)-\frac{\kappa_\mathcal{A}\#\mathcal{A}{}}{\#\mathcal{B}{}} V_m(\mathcal{B}_d;d)\Bigr)\Bigr|=o_{\delta,\mathcal{L}}\Bigl(\frac{\#\mathcal{A}}{\log{X}}\Bigr).
\label{eq:VmSums}
\end{equation}
Putting together \eqref{eq:TmVmSplit}, \eqref{eq:TmSums} and \eqref{eq:VmSums}, we obtain
\[
\sideset{}{'}\sum_{p_1,\dots,p_\ell} S(\mathcal{A}{}_{d},X^{\theta})-\frac{\kappa_\mathcal{A}\#\mathcal{A}{}}{\#\mathcal{B}{}} \sideset{}{'}\sum_{p_1,\dots,p_\ell} S(\mathcal{B}{}_{d},X^{\theta})\ll \Bigl(\exp(-\delta^{-1/2})+o_{\delta,\mathcal{L}}(1)\Bigr)\frac{\#\mathcal{A}}{\log{X}}.
\]
Letting $\delta\rightarrow 0$ sufficiently slowly then gives the result for $d<X^{\theta_1}$.

The contribution from $d$ with $X^{\theta_2}< d< X^{1-\theta_2}$ can be handled by an identical argument, where instead of restricting to $d p_1'\cdots p_m'\le X^{\theta_1}$ and $X^{\theta_1}<d p_1'\cdots p_m'\le X^{\theta_1}p_m'$ in $T_m$, $U_m$ and $V_m$, we instead restrict to  $d p_1'\cdots p_m'\le X^{1-\theta_2}$ and $X^{1-\theta_2}<d p_1'\cdots p_m'\le X^{1-\theta_2}p_m'$ respectively. The terms corresponding to $V_m$ involve $a\in\mathcal{A}{}_{d p_1'\cdots p_m'}$ with $X^{1-\theta_2}<d p_1'\cdots p_m'\le X^{1-\theta_1}\le X/p_m'$, so can be handled by the second part of Lemma \ref{lmm:AltTypeII} instead of the first part. Since $50/77>1-17/40+2\epsilon=1-\theta_2$, the terms corresponding to $T_m$ can still be handled by Lemma \ref{lmm:FundamentalLemma}.

Finally, the contribution from $d$ with $X^{\theta_1}\le  d\le X^{\theta_2}$ or $X^{1-\theta_2}\le d\le X^{1-\theta_1}$ can be bounded almost immediately by Lemma \ref{lmm:AltTypeII}. One Buchstab iteration gives
\[
S_{d}(X^\theta)=S_d(X^\delta)-\sum_{X^\delta<p<X^\theta}S_{d p}(p).
\]
We put $d=p_1\cdots p_\ell$ and sum over $p_1,\dots,p_\ell$ satisfying the constraints imposed by $\mathcal{L}$ and such that $d\in [X^{1-\theta_2},X^{1-\theta_1}]$. The first term makes a negligible total contribution by Lemma \ref{lmm:FundamentalLemma} since $d\le X^{1-\theta_1}<X^{50/77-\epsilon}$. The second term makes negligible total contribution by Lemma \ref{lmm:AltTypeII} (noting that $d p\le X^{1-\theta_1+\theta}\le X^{1-\theta}\le X/p$). This gives the result when $d\in[X^{1-\theta_2}, X^{1-\theta_1}]$. The argument for $d\in [X^{\theta_1},X^{\theta_2}]$ is completely analogous. 

Together these cover the whole range $p_1\cdots p_\ell\le X^{1-\theta_1}$, giving the result.
\end{proof}
Thus, since Lemma \ref{lmm:AltTypeII} and Lemma \ref{lmm:FundamentalLemma} follow from Proposition \ref{prpstn:TypeI} and Proposition \ref{prpstn:TypeII}, it suffices to establish Proposition \ref{prpstn:TypeI} and Proposition \ref{prpstn:TypeII}.
%
%
%
%
%
%
%
%
%
%
\section{Type I estimate}\label{sec:TypeI}
%
%
%
%
In this section we establish our `Type I' estimate Proposition \ref{prpstn:TypeI}, assuming the more technical Lemmas \ref{lmm:LargeSieveTypeI} and \ref{lmm:LInfTypeI}, which we will establish later  in Section \ref{sec:Fourier}. We recall that Proposition \ref{prpstn:TypeI} describes the number of elements of $\mathcal{A}$ in arithmetic progressions to modulus up to $X^{50/77-\epsilon}\approx X^{0.65}$ on average.

Our Type I estimate is based on suitable bounds on the Fourier Transform
\[
S_{\mathcal{A}}(\theta)=\sum_{a\in\mathcal{A}}e(a\theta)
\]
of the set $\mathcal{A}$. We recall our definition of the function $F_Y$ from \eqref{eq:FyDef}, which is a normalized version of $S_\mathcal{A}$. In particular, $|S_\mathcal{A}(\theta)|=\#\mathcal{A}\cdot F_X(\theta)$. The two key lemmas which we use in this section are the following.
%
%
%
%
\begin{lmm}[Large sieve estimate]\label{lmm:LargeSieveTypeI}
We have
\begin{align*}
\sum_{q\le Q}\sum_{\substack{0<a<q\\ (a,q)=1}}F_{Y}\Bigl(\frac{a}{q}\Bigr)&\ll  Q^{54/77}+\frac{Q^2}{Y^{50/77}}.
\end{align*}
\end{lmm}
%
%
%
%
\begin{lmm}[$\ell^\infty$ bound]\label{lmm:LInfTypeI}
Let $q<Y^{1/3}$ be of the form $q=q_1q_2$ with $(q_1,10)=1$ and $q_1>1$, and let $|\eta|<Y^{-2/3}/2$. Then for any integer $a$ coprime with $q$ we have 
\[ 
F_{Y}\Bigl(\frac{a}{q}+\eta\Bigr)\ll \exp\Bigl(-c\frac{\log{Y}}{\log{q}}\Bigr)
\]
for some absolute constant $c>0$.
\end{lmm}
%
%
%
%
%
\begin{proof}[Proof of Proposition \ref{prpstn:TypeI} assuming Lemma \ref{lmm:LargeSieveTypeI} and Lemma \ref{lmm:LInfTypeI}]
By M\"obius inversion and using additive characters, we have for $(q,10)=1$
\begin{align*}
\#\mathcal{A}'_q&=\#\{a\in\mathcal{A}:\,q|a,\,(a,10)=1\}\\
&=\sum_{\substack{a\in\mathcal{A}\\ q|a}}\sum_{d|(10,a)}\mu(d)\\
&=\sum_{d|10}\mu(d)\sum_{a\in\mathcal{A}}\Bigl(\frac{1}{d q}\sum_{0\le b<d q}e\Bigl(\frac{a b}{d q}\Bigr)\Bigr)\\
&=\sum_{d|10}\frac{\mu(d)}{d q}\sum_{0\le b< d q}S_{\mathcal{A}}\Bigl(\frac{b}{d q}\Bigr).
\end{align*}
We write $b/d q=b'/d q'$ with $(b',q')=1$, and separate the terms with $q'=1$. We then let $b'/d q'=b''/d' q'$ with $(b'',d' q')=1$. For $(q,10)=1$ we see that this representation is unique for all $b,d$ under consideration. Thus
\begin{align*}
\#\mathcal{A}'_q&=\sum_{d|10}\frac{\mu(d)}{d q}\sum_{0\le b'< d}S_{\mathcal{A}}\Bigl(\frac{b'}{d}\Bigr)+O\Bigl(\sum_{d|10}\sum_{\substack{q'|q\\ q'>1}}\sum_{\substack{0\le b'< d q'\\ (b',q')=1}}\frac{1}{q}\Bigl|S_{\mathcal{A}}\Bigl(\frac{b'}{d q'}\Bigr)\Bigr|\Bigr)\\
&=\frac{1}{q}\#\{a\in\mathcal{A}:(a,10)=1\}+O\Bigl(\frac{\#\mathcal{A}}{q}\sum_{d'|10}\sum_{\substack{q'|q\\ q'>1}}\sum_{\substack{0\le b''< d' q'\\ (b'',d' q')=1}}F_X\Bigl(\frac{b''}{d' q'}\Bigr)\Bigr).
\end{align*}
We note that $\#\{a\in\mathcal{A}:(a,10)=1\}=\kappa\#\mathcal{A}$. Summing over $q<Q$ with $(q,10)=1$ and letting $q=q' q''$, we obtain
\begin{align}
\sum_{\substack{q<Q\\ (q,10)=1}}\Bigl|\#\mathcal{A}_q'-\frac{\kappa\#\mathcal{A}}{q}\Bigr|
&\ll \sum_{\substack{q<Q\\ (q,10)=1}}\frac{\#\mathcal{A}}{q}\sum_{d'|10}\sum_{\substack{q'|q\\ q'>1}}\sum_{\substack{0\le b''< d q'\\ (b'',d' q')=1}}F_X\Bigl(\frac{b''}{d' q'}\Bigr)\nonumber\\
&\ll \sum_{\substack{1<q'<Q\\ (q',10)=1}}\frac{\#\mathcal{A}}{q'}\sum_{d'|10}\sum_{\substack{0\le b''< d' q'\\ (b'',d' q')=1}}F_X\Bigl(\frac{b''}{d' q'}\Bigr)\sum_{q''<Q/q'}\frac{1}{q''}\nonumber\\
&\ll \#\mathcal{A}(\log{X})^2\sup_{\substack{Q_1\le Q\\ d'|10}}\frac{1}{Q_1}\sum_{\substack{q'\sim Q_1\\ (q',10)=1\\ q'>1}}\sum_{\substack{0\le b''< d' q'\\ (b'',d' q')=1}}F_X\Bigl(\frac{b''}{d' q'}\Bigr).
\label{eq:TypeIBound}
\end{align}
Here we recall our notation that $q'\sim Q_1$ means $q'\in (Q_1/10,Q_1]$. By Lemma \ref{lmm:LargeSieveTypeI} we have for any $d|10$ 
\[
\frac{1}{Q_1}\sum_{q\sim Q_1}\sum_{\substack{0\le a<d q\\ (a,d q)=1}}F_X\Bigl(\frac{a}{d q}\Bigr)\ll \frac{1}{Q_1^{23/77}}+\frac{Q_1}{X^{50/77}},
\]
which gives the required bound if $Q_1>(\log{X})^{4A+8}$ on recalling that $Q_1\le Q\le X^{50/77}(\log{X})^{-2A-2}$. In the case $Q_1\le (\log{X})^{4A+8}$ we instead use Lemma \ref{lmm:LInfTypeI}, which gives
\[
\frac{1}{Q_1}\sum_{\substack{q\sim  Q_1\\ (q,10)=1\\ q>1}}\sum_{\substack{a\le d q \\ (a,d q)=1}}F_X\Bigl(\frac{a}{d q}\Bigr)\ll Q_1\sup_{\substack{(a,q)=1\\ 1<q\le Q_1\\ (q,10)=1\\ d|10}}F_X\Bigl(\frac{a}{d q}\Bigr)\ll_A\frac{Q_1}{(\log{X})^{100(A+1)}}.
\]
Thus we see that the bound \eqref{eq:TypeIBound} is $O_A(\#\mathcal{A}/(\log{X})^A)$ in either case, as required.
\end{proof}
We are left to establish Proposition \ref{prpstn:TypeII} and Lemma \ref{lmm:LargeSieveTypeI} and Lemma \ref{lmm:LInfTypeI}.
%
%
%
%
%
%
%
\section{Type II estimate}\label{sec:TypeII}
%
%
%
%
In this section we reduce our `Type II' estimate to various major arc and minor arc estimates. In particular, we will reduce the proof of Proposition \ref{prpstn:TypeII} to the proof of Propositions \ref{prpstn:Major}, \ref{prpstn:Generic} and \ref{prpstn:Exceptional}. We first recall the statement of Propositon \ref{prpstn:TypeII} which allows us to count integers in $\mathcal{A}$ with a specific type of prime factorization provided such numbers always have a `conveniently sized' factor.
%
%
%
%
\begin{nprpstn}[Type II estimate Proposition \ref{prpstn:TypeII} restated] \label{prpstn:TypeII2}
Let $\eta>0$, and let $\ell\le 2\eta^{-1}$. Let $\mathcal{R}\subseteq \mathcal{Q}_\ell(\eta)$ be a closed convex polytope in $\mathbb{R}^\ell$ which has the property that
\[
\mathbf{e}\in\mathcal{R}\Rightarrow \sum_{i\in\mathcal{I}} e_i\in\Bigl[\frac{9}{25}+\epsilon,\frac{17}{40}-\epsilon\Bigr]
\]
for some set $\mathcal{I}\subseteq\{1,\dots,\ell\}$. Then we have
\[
\sum_{\substack{a\in\mathcal{A}}}\mathbf{1}_{\mathcal{R}}(a)=\kappa_\mathcal{A}\frac{\#\mathcal{A}{}}{\#\mathcal{B}{}}\sum_{n<  X}\mathbf{1}_{\mathcal{R}}(n)+O_{\mathcal{R},\eta}\Bigl(\frac{\#\mathcal{A}}{\log{X}\log\log{X}}\Bigr),
\]
where
\[
\kappa_\mathcal{A}=
\begin{cases}
\frac{10(\phi(10)-1)}{9\phi(10)},\qquad&\text{if $(10,a_0)=1$,}\\
\frac{10}{9},&\text{otherwise.}\\
\end{cases}
\]
\end{nprpstn}
%
%
%
%
%
To avoid technical issues due to the fact that $\sum_{n<Y}\mathbf{1}_{\mathcal{A}}(n)$ can fluctuate with $Y$, we will replace our counts $\mathbf{1}_{\mathcal{R}}(n)$ with a weight $\Lambda_{\mathcal{R}}$, where for a set $\mathcal{R}\subseteq [\eta,1]^\ell$ we define
\begin{equation}
\Lambda_\mathcal{R}(n)=\sum_{\substack{p_1,\dots,p_{\ell}\\ p_1\cdots p_\ell=n \\ (\frac{\log{p}_1}{\log{X}},\dots,\frac{\log{p_\ell}}{\log{X}})\in\mathcal{R}}}\prod_{i=1}^\ell \log{p_i}.
\label{eq:LambdaDef}
\end{equation}
We note that in $\Lambda_\mathcal{R}$ the conditions are on $\log{p_i}/\log{X}$, whereas in $\mathbf{1}_{\mathcal{R}}$ the conditions are on $\log{p_i}/\log{n}$. If every $\mathbf{e}\in\mathcal{R}$ has $e_1\le \dots \le e_\ell$ then at most one term occurs in the summation, so $\Lambda_{\mathcal{R}}$ simplifies to
\[
\Lambda_{\mathcal{R}}(n)=\begin{cases}
\prod_{i=1}^{\ell}\log{p_i},\qquad &\text{if }n=p_1\cdots p_\ell\text{ and }(\frac{\log{p}_1}{\log{X}},\dots,\frac{\log{p_\ell}}{\log{X}})\in\mathcal{R},\\
0,&\text{otherwise.}
\end{cases}
\]
We prove Proposition \ref{prpstn:TypeII} by an application of the Hardy-Littlewood circle method, whereby we study the functions
\begin{align*}
S_{\mathcal{A}}(\theta)=\sum_{a\in\mathcal{A} }e(a\theta),\qquad S_{\mathcal{R}}(\theta)=\sum_{n<X}\Lambda_{\mathcal{R}}(n)e(n\theta).
\end{align*}
Proposition \ref{prpstn:TypeII} then relies on the following three components.
%
%
%
%
\begin{prpstn}[Major arcs] \label{prpstn:Major}
Fix $\eta>0$ and let $\ell\in\mathbb{Z}$ satisfy $1\le \ell \le 2/\eta$. Let $\delta=(\log\log{X})^{-1}$, and let $\mathcal{R}_X=\mathcal{R}_X(a_1,\dots,a_{\ell-1})$ be given by
\[
\mathcal{R}_X=\Bigl\{\mathbf{e}\in\mathbb{R}^\ell: e_i\in (a_i,a_i+\delta]\text{ for }1\le i\le\ell-1,\, \sum_{i=1}^\ell e_i\le 1,\,e_\ell\ge \max\Bigl(\frac{\eta}{4},1-\sum_{i=1}^{\ell-1}a_i-\ell\delta\Bigr)\Bigr\},
\]
for some $a_1,\dots,a_{\ell-1}\in\mathbb{R}$ satisfying $\min_i a_i\ge \eta/2$ and $\sum_{i=1}^{\ell-1}a_i<1-\eta/2$.

Let $\mathcal{M}=\mathcal{M}(C)$ be given by
\[
\mathcal{M}=\Bigl\{0\le a<X:\, \Bigl|\frac{a}{X}-\frac{b}{q}\Bigr|\le \frac{(\log{X})^C}{X}\text{ for some integers $b,q$ with }q\le (\log{X})^C\Bigr\}.
\]
Then
\[
\frac{1}{X}\sum_{\substack{0\le a<X\\ a\in\mathcal{M}}}S_{\mathcal{A}}\Bigl(\frac{a}{X}\Bigr)S_{\mathcal{R}_X}\Bigl(\frac{-a}{X}\Bigr)=\kappa_\mathcal{A}\frac{\#\mathcal{A}}{X}\sum_{n< X}\Lambda_{\mathcal{R}_X}(n)+O_{C,\eta}\Bigl(\frac{\#\mathcal{A}}{(\log{X})^C}\Bigr).
\]
Here $\kappa_\mathcal{A}$ is the constant given in Proposition \ref{prpstn:TypeII}. The implied constant depends on $C$ and $\eta$, but not on $\mathcal{R}_X$ or $a_1\dots,a_{\ell-1}$.
\end{prpstn}
%
%
%
%
\begin{prpstn}[Generic minor arcs] \label{prpstn:Generic}
Fix $\eta>0$ and let $\ell\in\mathbb{Z}$ satisfy $1\le \ell \le 2/\eta$. Let $\mathcal{R}\subseteq\mathbb{R}^\ell$ be a closed convex polytope. Let $\mathcal{M}=\mathcal{M}(C)$ be as in Proposition \ref{prpstn:Major}.

Then there is some exceptional set $\mathcal{E}\subseteq[0,X]$ with
\[
\#\mathcal{E}\le X^{23/40},
\]
such that
\[
\frac{1}{X}\sum_{\substack{a<X\\ a\notin\mathcal{E}}}\Big|S_\mathcal{A}\Bigl(\frac{a}{X}\Bigr)S_{\mathcal{R}}\Bigl(\frac{-a}{X}\Bigr)\Big|\ll_\eta \frac{\#\mathcal{A}}{X^{\epsilon}}.
\]
The implied constant depends on $\eta$, but not on $\mathcal{R}$.
\end{prpstn}
%
%
%
%
\begin{prpstn}[Exceptional minor arcs] \label{prpstn:Exceptional}
Let $A>0$. Let $\eta$, $\ell$, $\mathcal{R}_X=\mathcal{R}_X(a_1,\dots,a_{\ell-1})$ and $\mathcal{M}=\mathcal{M}(C)$ be as given in Proposition \ref{prpstn:Major}. Let $a_1,\dots,a_{\ell-1}$ in the definition of $\mathcal{R}_X$ satisfy $\sum_{i\in\mathcal{I}}a_i\in [9/25+\epsilon/2,17/40-\epsilon/2]\cup[23/40+\epsilon/2,16/25-\epsilon/2]$ for some $\mathcal{I}\subseteq\{1,\dots,\ell-1\}$, and let $C=C(A,\eta)$ in the definition of $\mathcal{M}$ be sufficiently large in terms of $A$ and $\eta$. Let $\mathcal{E}\subseteq [0,X]$ be any set such that $\#\mathcal{E}\le X^{23/40}$. Then we have
\[
\frac{1}{X}\sum_{\substack{a\in\mathcal{E}\\ a\notin\mathcal{M}}}S_{\mathcal{A}}\Bigl(\frac{a}{X}\Bigr)S_{\mathcal{R}_X}\Bigl(\frac{-a}{X}\Bigr)\ll_{\eta,A}\frac{\#\mathcal{A}}{(\log{X})^{A}}.
\]
The implied constant depends on $\eta$ and $A$, but not on $\mathcal{R}_X$ or $a_1,\dots,a_{\ell-1}$.
\end{prpstn}
%
%
%
%
We expect the contribution from the major arcs $\mathcal{M}$ to give the main contribution. Proposition \ref{prpstn:Major} shows that we can get an asymptotic formula from frequencies in $\mathcal{M}$. Proposition \ref{prpstn:Generic} shows that most frequencies contribute negligibly, and that any significant contribution must come from some small exceptional set $\mathcal{E}$. (In view of Proposition \ref{prpstn:Major}, we must have $\mathcal{E}$ contains elements of $\mathcal{M}$ and so $\mathcal{E}$ is non-empty). We would expect that we can take $\mathcal{E}=\mathcal{M}$, but cannot quite show this. However, Proposition \ref{prpstn:Exceptional} shows that $\mathcal{E}\setminus\mathcal{M}$ contributes negligibly to our sum, which is sufficient for our purposes.
%
%
%
%
\begin{proof}[Proof of Proposition \ref{prpstn:TypeII} assuming Propositions \ref{prpstn:Major}, \ref{prpstn:Generic} and \ref{prpstn:Exceptional} and  Lemma \ref{lmm:FundamentalLemma}]
Let $\delta=(\log\log{X})^{-1}$. Clearly we may assume that $\delta$ is sufficiently small in terms of $\eta$, since otherwise the result is trivial. We note that $\ell\ge 2$, since the sum of coordinates of points in $\mathcal{R}$ is 1 but a non-trivial subset of them lies in $[9/25,17/40]$.
Given reals $a_1,\dots,a_{\ell-1}\ge 0$ and $\gamma>0$ and a set $\mathcal{S}\in\mathbb{R}^\ell$, let
\begin{align*}
\mathcal{C}(\mathbf{a};\gamma)&:=\Bigl(a_1,a_1+\gamma\Bigr]\times\dots \times \Bigl(a_{\ell-1},a_{\ell-1}+\gamma\Bigr],\\
\mathcal{C}^+(\mathbf{a};\gamma)&:=\Bigl\{\mathbf{e}\in[\eta/4,1]^\ell:\,(e_1,\dots,e_{\ell-1})\in\mathcal{C}(\mathbf{a};\gamma),\,\sum_{i=1}^\ell e_i\le 1,\,e_\ell\ge 1-\sum_{i=1}^{\ell-1}a_i-\ell\delta \Bigr\},\\
\tilde{\mathbf{1}}_{\mathcal{S}}(n)&:=\begin{cases}
1,\qquad &n=p_1\cdots p_\ell\text{ for some $p_1,\dots,p_{\ell}$ with }\Bigl(\frac{\log{p_1}}{\log{X}},\dots,\frac{\log{p_{\ell}}}{\log{X}}\Bigr)\in\mathcal{S},\\
0,&\text{otherwise.}
\end{cases}
\end{align*}
We see that $\mathbf{1}_{\mathcal{S}}$ and $\tilde{\mathbf{1}}_{\mathcal{S}}$ differ in that the denominators of the fractions are $\log{n}$ and $\log{X}$ respectively.

We cover $[\eta,1]^{\ell-1}$ by $O(\delta^{-(\ell-1)})$ disjoint hypercubes $\mathcal{C}(\mathbf{a},\delta)$ of side length $\delta$ (for example, we can take all $\mathbf{a}\in \{0,\delta,2\delta,\dots,\lceil \delta^{-1}\rceil\delta\}^{\ell-1}$). Let $\overline{\mathcal{R}}\subseteq[\eta,1]^{\ell-1}$ denote the projection of $\mathcal{R}$ onto the first $\ell-1$ coordinates (which is also a closed convex polytope). We see that if $n\in [X^{1-\delta^2},X]$ then $\log{n}$ and $\log{X}$ differ by a factor of at most $1-\delta^2$. In particular, if $\log{p_j}/\log{X}\in [a_j,a_j+\delta]$ then certainly $\log{p_j}/\log{n}\in [a_j,a_j+2\delta]$. This means that if $\mathcal{C}(\mathbf{a};2\delta)\subseteq\overline{\mathcal{R}}$ and $\log{p_j}/\log{X}\in [a_j,a_j+\delta]$ for all $j\le\ell-1$, then $\mathbf{1}_{\mathcal{R}}(p_1\cdots p_\ell)=1$ for all $p_\ell\in[X^{1-\delta^2}/p_1\cdots p_{\ell-1},X/p_1\cdots p_{\ell-1}]$. Thus for $n\in [X^{1-\delta^2},X]$
\begin{equation}
\mathbf{1}_{\mathcal{R}}(n)\tilde{\mathbf{1}}_{\mathcal{C}^+(\mathbf{a};\delta)}(n)=\begin{cases}
0,\qquad &\text{if $\overline{\mathcal{R}}\cap\mathcal{C}(\mathbf{a};2\delta)=\emptyset$},\\
\tilde{\mathbf{1}}_{\mathcal{C}^+(\mathbf{a};\delta)}(n),& \text{if $\mathcal{C}(\mathbf{a};2\delta)\subseteq\overline{\mathcal{R}}$},\\
O(\tilde{\mathbf{1}}_{\mathcal{C}^+(\mathbf{a};\delta)}(n)),&\text{otherwise}.
\end{cases}
\label{eq:RCubeSplit}
\end{equation}
If $\mathcal{C}(\mathbf{a};2\delta)\cap\overline{\mathcal{R}}\ne\emptyset$ but $\mathcal{C}(\mathbf{a};2\delta)\not\subseteq\overline{\mathcal{R}}$ then $\mathcal{C}(\mathbf{a};2\delta)$ intersects the boundary $\partial\overline{\mathcal{R}}$ of $\overline{\mathcal{R}}$. 

Since $\mathbf{1}_{\mathcal{R}}(n)$ is supported on $n$ with $\ell$ prime factors all at least $n^\eta$, if $n=p_1\cdots p_\ell\ge X^{1-\delta^2}$ and $\mathbf{1}_{\mathcal{R}}(n)=1$ then there is an $\mathbf{a}$ with $a_i\ge \eta/2$ such that $\tilde{\mathbf{1}}_{\mathcal{C}(\mathbf{a};\delta)}(p_1\cdots p_{\ell-1})=1$. Moreover, since $n\ge X^{1-\delta^2}$ we have $p_{\ell}\ge X^{1-\delta^2}/p_1\cdots p_{\ell-1}\ge X^{1-\sum_{i=1}^{\ell-1}a_i-\ell\delta}$, so in fact $\tilde{\mathbf{1}}_{\mathcal{C}^+(\mathbf{a};\delta)}(n)=1$. Since the cubes are disjoint, this happens for exactly one choice of $\mathbf{a}$. Therefore we have for any $n\in[X^{1- \delta^2},X]$
\[
\mathbf{1}_{\mathcal{R}}(n)=\sum_{\mathbf{a}}\tilde{\mathbf{1}}_{\mathcal{C}^+(\mathbf{a};\delta)}(n)\mathbf{1}_{\mathcal{R}}(n).
\]
Using this with \eqref{eq:RCubeSplit} to split the summation over hypercubes $\mathcal{C}$, we find
\begin{align*}
&\Bigl|\sum_{\substack{m\in\mathcal{A}\\ X^{1-\delta^2}<m<X}}\mathbf{1}_{\mathcal{R}}(m)-\frac{\kappa_\mathcal{A}\#\mathcal{A}}{X}\sum_{X^{1-\delta^2}<n< X}\mathbf{1}_{\mathcal{R}}(n)\Bigr|\\
&\le \sum_{\substack{\mathbf{a}\\ \mathcal{C}(\mathbf{a};2\delta)\subseteq\overline{\mathcal{R}}} }\Bigl|\sum_{\substack{m\in\mathcal{A}\\ X^{1-\delta^2}<m<X}}\tilde{\mathbf{1}}_{\mathcal{C}^+(\mathbf{a};\delta)}(m)-\frac{\kappa_\mathcal{A}\#\mathcal{A}}{X}\sum_{X^{1-\delta^2}<n< X}\tilde{\mathbf{1}}_{\mathcal{C}^+(\mathbf{a};\delta)}(n)\Bigr|\nonumber\\
&\quad+\sum_{\substack{\mathbf{a}\\ \mathcal{C}(\mathbf{a};2\delta)\cap\partial\overline{\mathcal{R}}\ne\emptyset} }O\Bigl(\sum_{\substack{m\in\mathcal{A}\\ X^{1-\delta^2}<m<X}}\tilde{\mathbf{1}}_{\mathcal{C}^+(\mathbf{a};\delta)}(m)+\sum_{X^{1-\delta^2}<n<X}\frac{\kappa_\mathcal{A}\#\mathcal{A}}{X}\tilde{\mathbf{1}}_{\mathcal{C}^+(\mathbf{a};\delta)}(n)\Bigr).
\end{align*}
Re-inserting terms with $m\le X^{1-\delta^2}$ and $n\le X^{1-\delta^2}$, we obtain
\begin{align}
\Bigl|\sum_{m\in\mathcal{A}}\mathbf{1}_{\mathcal{R}}(m)-\frac{\kappa_\mathcal{A}\#\mathcal{A}}{X}&\sum_{n< X}\mathbf{1}_{\mathcal{R}}(n)\Bigr|
\le \sum_{\substack{\mathbf{a}\\ \mathcal{C}(\mathbf{a};2\delta)\subseteq\overline{\mathcal{R}}} }\Bigl|\sum_{m\in\mathcal{A}}\tilde{\mathbf{1}}_{\mathcal{C}^+(\mathbf{a};\delta)}(m)-\frac{\kappa_\mathcal{A}\#\mathcal{A}}{X}\sum_{n< X}\tilde{\mathbf{1}}_{\mathcal{C}^+(\mathbf{a};\delta)}(n)\Bigr|\nonumber\\
&\qquad+\sum_{\substack{\mathbf{a}\\ \mathcal{C}(\mathbf{a};2\delta)\cap\partial\overline{\mathcal{R}}\ne\emptyset} }O\Bigl(\sum_{m\in\mathcal{A}}\tilde{\mathbf{1}}_{\mathcal{C}^+(\mathbf{a};\delta)}(m)+\sum_{n<X}\frac{\kappa_\mathcal{A}\#\mathcal{A}}{X}\tilde{\mathbf{1}}_{\mathcal{C}^+(\mathbf{a};\delta)}(n)\Bigr)\nonumber\\
&\qquad + O\Bigl(\sum_{\substack{m\in\mathcal{A}\\ m\le X^{1-\delta^2}}}1\Bigr)+O\Bigl(\frac{\#\mathcal{A}}{X}\sum_{n\le X^{1-\delta^2}}1\Bigr).
\label{eq:CubeSplit}
\end{align}
The final two terms above satisfy
\begin{equation}
\sum_{\substack{m\in\mathcal{A}\\ m\le X^{1-\delta^2}}}1+\kappa_{\mathcal{A}}\frac{\#\mathcal{A}}{X}\sum_{n\le X^{1-\delta^2}}1\ll \#\mathcal{A}^{1-\delta^2}+\frac{\#\mathcal{A}}{X^{\delta^2}}\ll \frac{\delta\#\mathcal{A}}{\log{X}}.
\label{eq:TrivialTerms}
\end{equation}
We now consider the contribution to \eqref{eq:CubeSplit} from $\mathcal{C}(\mathbf{a};2\delta)\cap\partial\overline{\mathcal{R}}\ne\emptyset$. Since $\mathcal{R}\subseteq[\eta,1]^\ell$, we must have $a_i\ge \eta/2$ and since the coordinates of points in $\mathcal{R}$ sum to 1 we also have $\sum_{i=1}^{\ell-1}a_i\le 1-\eta/2$. Since $\tilde{\mathbf{1}}_{\mathcal{C}^+(\mathbf{a};\delta)}(n)$ and $\Lambda_{\mathcal{C}^+(\mathbf{a};\delta)}(n)$ have the same support, which is restricted to integers with no factor less than $X^{\eta/4}$, we have $\tilde{\mathbf{1}}_{\mathcal{C}^+(\mathbf{a};\delta)}(n)\ll_\eta (\log{X})^{-\ell} \Lambda_{\mathcal{C}^+(\mathbf{a};\delta)}(n)$. Thus we have
\begin{align}
\sum_{m\in\mathcal{A}}\tilde{\mathbf{1}}_{\mathcal{C}^+(\mathbf{a};\delta)}(m)&+\frac{\kappa_\mathcal{A}\#\mathcal{A}}{X}\sum_{n<X}\tilde{\mathbf{1}}_{\mathcal{C}^+(\mathbf{a};\delta)}(n)\nonumber\\
&\ll_\eta \frac{1}{(\log{X})^\ell}\Bigl( \sum_{m\in\mathcal{A}}\Lambda_{\mathcal{C}^+(\mathbf{a};\delta)}(m)+\frac{\kappa_\mathcal{A}\#\mathcal{A}}{X}\sum_{n<X}\Lambda_{\mathcal{C}^+(\mathbf{a};\delta)}(n)\Bigr)\nonumber\\
&\le \frac{1}{(\log{X})^\ell}\Bigl|\sum_{m\in\mathcal{A}}\Lambda_{\mathcal{C}^+(\mathbf{a};\delta)}(m)-\frac{\kappa_\mathcal{A}\#\mathcal{A}}{X}\sum_{n<X}\Lambda_{\mathcal{C}^+(\mathbf{a};\delta)}(n)\Bigr|\nonumber\\
&\qquad +\frac{2}{(\log{X})^\ell}\frac{\kappa_\mathcal{A}\#\mathcal{A}}{X}\sum_{n<X}\Lambda_{\mathcal{C}^+(\mathbf{a};\delta)}(n).
\label{eq:BoundaryTerms}
\end{align}
Here we used the triangle inequality in the final line. By the prime number theorem, for any choice of $\mathbf{a}\in[0,2]^{\ell-1}$ we have
\begin{align*}
\sum_{n<X}\Lambda_{\mathcal{C}^+(\mathbf{a};\delta)}(n)&\le \sum_{\substack{p_1,\dots,p_{\ell-1}\\ p_i\in (X^{a_i},X^{a_i+\delta}]}}\Bigl(\prod_{i=1}^{\ell-1}\log{p_i}\Bigr)\sum_{p_\ell<X/p_1\cdots p_{\ell-1}}\log{p_{\ell}}\\
&\ll X\sum_{\substack{p_1,\dots,p_{\ell-1}\\ p_i\in (X^{a_i},X^{a_i+\delta}]}}\prod_{i=1}^{\ell-1}\frac{\log{p_i}}{p_i}\\
&\ll \delta^{\ell-1}X(\log{X})^{\ell-1}.
\end{align*}
Since $\mathcal{R}$ is a closed convex polytope, so is $\overline{\mathcal{R}}\subseteq\mathbb{R}^{\ell-1}$. Therefore there are $O_\mathcal{R}(\delta^{-(\ell-2)})$ hypercubes $\mathcal{C}(\mathbf{a};2\delta)$ which intersect $\partial \overline{\mathcal{R}}$. Thus the contribution to \eqref{eq:CubeSplit} from the final term of \eqref{eq:BoundaryTerms} is
\begin{align}
\ll \frac{\#\mathcal{A}}{X(\log{X})^\ell}\sum_{\substack{\mathbf{a}\\ \mathcal{C}(\mathbf{a};2\delta)\cap\partial\overline{\mathcal{R}}\ne\emptyset} }\sum_{n<X}\Lambda_{\mathcal{C}^+(\mathbf{a};\delta)}(n)
&\ll  \frac{\delta^{\ell-1}\#\mathcal{A}}{\log{X}}\sum_{\substack{\mathbf{a}\\ \mathcal{C}(\mathbf{a};2\delta)\cap\partial\overline{\mathcal{R}}\ne\emptyset} }1\nonumber\\
&\ll_\mathcal{R}  \frac{\delta\#\mathcal{A}}{\log{X}}.\label{eq:BoundaryErrorTerms}
\end{align}
We now consider the terms with $\mathcal{C}(\mathbf{a};2\delta)\subseteq\overline{\mathcal{R}}$. %
Since $\mathcal{R}\subseteq \mathcal{Q}_\ell(\eta)$, if $\mathbf{e}\in\mathcal{R}$ then $e_1\le \dots \le e_\ell$, so if $\mathbf{e}'\in\overline{\mathcal{R}}$ then $e_1'\le \dots \le e_{\ell-1}'$. Therefore, since $\mathcal{C}(\mathbf{a};2\delta)\subseteq\overline{\mathcal{R}}$, %
\begin{equation}
a_j+\delta<a_{j+1}\text{ for }j\in\{1,\dots,\ell-2\}.
\label{eq:CubeCond1}
\end{equation}
Since $\sum_{i=1}^\ell e_i=1$ and $e_{\ell-1}\le e_\ell$ for $\mathbf{e}\in\mathcal{R}$, if $\mathbf{e}'\in\overline{\mathcal{R}}$ then $e_{\ell-1}'\le 1-\sum_{i=1}^{\ell-1}e_i'$. Therefore, since $(a_1+2\delta,\dots,a_{\ell-1}+2\delta)\in\mathcal{C}(\mathbf{a};2\delta)\subseteq\overline{\mathcal{R}}$, we have 
\begin{equation}
a_{\ell-1}+2\delta\le 1-\sum_{i=1}^{\ell-1}a_i-(2\ell-2)\delta\le 1-\sum_{i=1}^{\ell-1}a_i-\ell\delta.
\label{eq:CubeCond2}
\end{equation}
Together \eqref{eq:CubeCond1} and \eqref{eq:CubeCond2} imply that at most one term occurs in the summation in $\Lambda_{\mathcal{C}^+(\mathbf{a};\delta)}$. Thus for such $\mathcal{C}(\mathbf{a};2\delta)$, since the coordinates are localized, we have
\begin{align}
\tilde{\mathbf{1}}_{\mathcal{C}^+(\mathbf{a};\delta)}(n)&=\frac{(1+O_\eta(\delta))\Lambda_{\mathcal{C}^+(\mathbf{a};\delta)}(n)}{(1-\sum_{i=1}^{\ell-1}a_i)(\prod_{i=1}^{\ell-1} a_i)(\log{X})^\ell}\nonumber\\
&=\frac{\Lambda_{\mathcal{C}^+(\mathbf{a};\delta)}(n)}{(1-\sum_{i=1}^{\ell-1}a_i)(\prod_{i=1}^{\ell-1} a_i)(\log{X})^\ell}+O_\eta(\delta \tilde{\mathbf{1}}_{\mathcal{C}^+(\mathbf{a};\delta)}(n)).
\label{eq:LambdaApprox}
\end{align}
Thus
\begin{align}
 \sum_{\substack{\mathbf{a}\\ \mathcal{C}(\mathbf{a};2\delta)\subseteq\overline{\mathcal{R}}} }&\Bigl|\sum_{m\in\mathcal{A}}\tilde{\mathbf{1}}_{\mathcal{C}^+(\mathbf{a};\delta)}(m)-\frac{\kappa_\mathcal{A}\#\mathcal{A}}{X}\sum_{n< X}\tilde{\mathbf{1}}_{\mathcal{C}^+(\mathbf{a};\delta)}(n)\Bigr|\nonumber\\
&\ll_\eta \frac{1}{(\log{X})^\ell} \sum_{\substack{\mathbf{a}\\ \mathcal{C}(\mathbf{a};2\delta)\subseteq\overline{\mathcal{R}}} }\Bigl|\sum_{m\in\mathcal{A}}\Lambda_{\mathcal{C}^+(\mathbf{a};\delta)}(m)-\frac{\kappa_\mathcal{A}\#\mathcal{A}}{X}\sum_{n< X}\Lambda_{\mathcal{C}^+(\mathbf{a};\delta)}(n)\Bigr|\nonumber\\
&\qquad +\delta\sum_{\substack{\mathbf{a}\\ \mathcal{C}(\mathbf{a};2\delta)\subseteq\overline{\mathcal{R}}} }\Bigl(\sum_{m\in\mathcal{A}}\tilde{\mathbf{1}}_{\mathcal{C}^+(\mathbf{a};\delta)}(m)+\frac{\kappa_\mathcal{A}\#\mathcal{A}}{X}\sum_{n< X}\tilde{\mathbf{1}}_{\mathcal{C}^+(\mathbf{a};\delta)}(n)\Bigr).
\label{eq:InternalTerms}
\end{align}
Since any $n=p_1\cdots p_\ell$ contributing to the second term above is counted at most once and has all prime factors at least $X^{\eta/4}$, we have
\begin{align}
\delta\sum_{\substack{\mathbf{a}\\ \mathcal{C}(\mathbf{a};2\delta)\subseteq\overline{\mathcal{R}}} }\Bigl(\sum_{m\in\mathcal{A}}\tilde{\mathbf{1}}_{\mathcal{C}^+(\mathbf{a};\delta)}(m)+\frac{\kappa_\mathcal{A}\#\mathcal{A}}{X}\sum_{n< X}\tilde{\mathbf{1}}_{\mathcal{C}^+(\mathbf{a};\delta)}(n)\Bigr)%
&\ll \delta S(\mathcal{A},X^{\eta/4})+\delta\frac{\#\mathcal{A}}{X}S(\mathcal{B},X^{\eta/4})\nonumber\\
&\ll_\eta \frac{\delta \#\mathcal{A}}{\log{X}}.
\label{eq:InternalErrorTerms}
\end{align}
Here we used Lemma \ref{lmm:FundamentalLemma} and \eqref{eq:Buchstab} in the final line. Combining \eqref{eq:TrivialTerms}, \eqref{eq:BoundaryTerms}, \eqref{eq:BoundaryErrorTerms}, \eqref{eq:InternalTerms} and \eqref{eq:InternalErrorTerms}, we find \eqref{eq:CubeSplit} is bounded by
\begin{equation*}
\ll_\eta \frac{1}{(\log{X})^{\ell}}\sum_{\substack{\mathbf{a}\\ \mathcal{C}(\mathbf{a};2\delta)\cap\overline{\mathcal{R}}\ne\emptyset} }\Bigl|\sum_{m\in\mathcal{A}}\Lambda_{\mathcal{C}^+(\mathbf{a};\delta)}(m)-\kappa_\mathcal{A}\frac{\#\mathcal{A}}{X}\sum_{n< X}\Lambda_{\mathcal{C}^+(\mathbf{a};\delta)}(n)\Bigr|+\frac{\delta\#\mathcal{A}}{\log{X}}.
\end{equation*}
Thus to establish Proposition \ref{prpstn:TypeII} it is sufficient to show that for any $A>0$, we have
\begin{equation}
\sum_{\substack{m\in\mathcal{A}}}\Lambda_{\mathcal{C}^+(\mathbf{a};\delta)}(m)=\frac{\#\mathcal{A}}{X}\sum_{n<  X}\Lambda_{\mathcal{C}^+(\mathbf{a};\delta)}(n)+O_{A,\eta}\Bigl(\frac{\#\mathcal{A}}{(\log{X})^A}\Bigr),
\label{eq:CubeSum}
\end{equation}
uniformly for every hypercube $\mathcal{C}(\mathbf{a};\delta)$ of side length $\delta$ with $\mathcal{C}(\mathbf{a};2\delta)\cap\overline{\mathcal{R}}\ne\emptyset$. 

Since $\sum_{i\in\mathcal{I}}e_i\in[9/25+\epsilon,17/40-\epsilon]$ if $\mathbf{e}\in\mathcal{R}$, by taking $\mathcal{J}=\mathcal{I}$ or $\mathcal{J}=\{1,\dots,\ell\}\backslash\mathcal{I}$, we must have that $\sum_{i\in\mathcal{J}}a_j\in[9/25+\epsilon/2,17/40-\epsilon/2]\cup[23/40+\epsilon/2,16/25-\epsilon/2]$ for some $\mathcal{J}\subseteq\{1,\dots,\ell-1\}$ for any $\mathbf{a}$ such that $\mathcal{C}(\mathbf{a};2\delta)\cap\mathcal{R}\ne \emptyset$. Since $\mathcal{R}\subseteq[\eta,1]^\ell$, we have $\min_i a_i\ge \eta/2$ and $\sum_{i=1}^{\ell-1}a_i<1-\eta/2$ if $\mathcal{C}(\mathbf{a};2\delta)\cap\mathcal{R}\ne \emptyset$. Thus all hypercubes under consideration satisfy the assumptions on $\mathcal{R}_X$ of Propositions \ref{prpstn:Major}-\ref{prpstn:Exceptional}.

By Fourier expansion we have
\[
\sum_{m\in\mathcal{A}}\Lambda_{\mathcal{C}^+(\mathbf{a};\delta)}(m)=\frac{1}{X}\sum_{0\le b< X}S_{\mathcal{A}}\Bigl(\frac{b}{X}\Bigr)S_{\mathcal{C}^+(\mathbf{a};\delta)}\Bigl(\frac{-b}{X}\Bigr).
\]
We split the summation over $b$ into the sets $\mathcal{M}$, $[0,X)\backslash(\mathcal{E}\cup\mathcal{M})$ and $\mathcal{E}\backslash\mathcal{M}$, where $\mathcal{M}$ is as given by Proposition \ref{prpstn:Major}, and $\mathcal{E}$ is the set who existence is asserted by Proposition \ref{prpstn:Generic}. We then apply Propositions \ref{prpstn:Major}, \ref{prpstn:Generic} and \ref{prpstn:Exceptional} respectively to each set in turn. Let $H_{\mathcal{C}^+}(\theta)=S_{\mathcal{A}}(\theta)S_{\mathcal{C}^+(\mathbf{a};\delta)}(-\theta)$. For $C$ in the definition of $\mathcal{M}$ sufficiently large in terms of $A$ and $\eta$, this gives
\begin{align*}
\sum_{m\in\mathcal{A}}\Lambda_{\mathcal{C}^+(\mathbf{a};\delta)}(m)
&=\frac{1}{X}\sum_{b\in\mathcal{M}}H_{\mathcal{C}^+}\Bigl(\frac{b}{X}\Bigr)
+\frac{1}{X}\sum_{b\notin\mathcal{E}\cup\mathcal{M} }H_{\mathcal{C}^+}\Bigl(\frac{b}{X}\Bigr)
+\frac{1}{X}\sum_{\substack{b\in\mathcal{E}\\ b\notin \mathcal{M}}}H_{\mathcal{C}^+}\Bigl(\frac{b}{X}\Bigr)\\
&=\kappa_\mathcal{A}\frac{\#\mathcal{A}}{X}\sum_{n< X}\Lambda_{\mathcal{C}^+(\mathbf{a};\delta)}(n)+O_{A,\eta}\Bigl(\frac{\#\mathcal{A}}{(\log{X})^A}\Bigr).
\end{align*}
This gives \eqref{eq:CubeSum}, and hence completes the proof of Proposition \ref{prpstn:TypeII}.
\end{proof}
Since Lemma \ref{lmm:FundamentalLemma} follows from Proposition \ref{prpstn:TypeI}, which in turn follows from Lemmas \ref{lmm:LargeSieveTypeI} and \ref{lmm:LInfTypeI}, we are left to establish Lemma \ref{lmm:LargeSieveTypeI}, Lemma \ref{lmm:LInfTypeI}, Proposition \ref{prpstn:Major}, Proposition \ref{prpstn:Generic} and Proposition \ref{prpstn:Exceptional}.
%
%
%
%
%
%
\section{Fourier Estimates} \label{sec:Fourier}
In this section we collect various distributional bounds on the Fourier transform
\[
S_{\mathcal{A}}(\theta)=\sum_{a\in\mathcal{A}}e(a\theta),
\]
which will underpin our later analysis. In particular, we establish Lemma \ref{lmm:LargeSieveTypeI} and Lemma \ref{lmm:LInfTypeI}, as well as several other related estimates. Specifically, Lemma \ref{lmm:LargeSieveTypeI} is a special case of Lemma \ref{lmm:LargeSieve}, and Lemma \ref{lmm:LInfTypeI} is the same as Lemma \ref{lmm:LInfBound}.

We recall our normalized version of $S_{\mathcal{A}}(\theta)$ from \eqref{eq:FyDef}
\[
F_{Y}(\theta)=Y^{-\log{9}/\log{10}}\Bigl|\sum_{n<Y}\mathbf{1}_{\mathcal{A}_1}(n)e(n\theta)\Bigr|.
\]
We recall that we assume $Y$ is an integral power of ten whenever we encounter $F_Y$ to avoid some unimportant technicalities. In particular,
\begin{equation}
F_Y(\theta)\le 1
\label{eq:FyBound}
\end{equation}
for all $\theta$ and $Y$. The key property of $F_Y$ which we exploit is that it has an exceptionally nice product form. If $Y=10^k$, then letting $n=\sum_{i=0}^{k-1}n_i 10^i$ have decimal digits $n_{k-1},\dots, n_0$, we find
\begin{align}
F_Y(\theta)&=\frac{1}{9^k}\Bigl|\sum_{n_0,\dots,n_{k-1}\in\{0,\dots,9\}\backslash \{a_0\}}e\Bigl(\sum_{i=0}^{k-1}n_i10^i\theta\Bigr)\Bigr|\nonumber\\
&=\prod_{i=0}^{k-1}\frac{1}{9}\Bigl|\sum_{n_i\in\{0,\dots,9\}\backslash \{a_0\}}e(n_i 10^i \theta)\Bigr|\nonumber\\
&=\prod_{i=1}^k\frac{1}{9}\Bigl|\frac{e(10^{i}\theta)-1}{e(10^{i-1}\theta)-1}-e(a_010^{i-1}\theta)\Bigr|.
\label{eq:ProductFormula}
\end{align}
We note that $F_Y$ is periodic modulo 1, and that the above product formula gives the identity
\begin{equation}
F_{UV}(\theta)=F_U(\theta)F_V(U\theta).
\label{eq:FFactorization}
\end{equation}
(We recall that we assume that $U$ and $V$ are both powers of 10 in such a statement.)
%
%
%
%
\begin{lmm}[$\ell^\infty$ bound, Lemma \ref{lmm:LInfTypeI} restated]\label{lmm:LInfBound}
Let $q<Y^{1/3}$ be of the form $q=q_1q_2$ with $(q_1,10)=1$ and $q_1>1$, and let $|\eta|<Y^{-2/3}/2$. Then for any integer $a$ coprime with $q$ we have 
\[ 
F_{Y}\Bigl(\frac{a}{q}+\eta\Bigr)\ll \exp\Bigl(-c\frac{\log{Y}}{\log{q}}\Bigr)
\]
for some absolute constant $c>0$.
\end{lmm}
%
%
%
%
\begin{proof}
From the bounds coming from truncated Taylor expansions, we have that
\begin{align*}
|e(n\theta)+e((n+1)\theta)|^2=2+2\cos(2\pi \|\theta\|)&\le 4-4\pi^2 \|\theta\|^2+4\pi^4 \|\theta\|^4/3\\
&\le 4-4\|\theta\|^2\le 4\exp(-\|\theta\|^2).
\end{align*}
We recall that $\|\cdot\|$ denotes the distance to the nearest integer. This implies that
\[
\Bigl|\sum_{n_i\in\{0,\dots,9\}\backslash \{a_0\}}e(n_i \theta)\Bigr|\le 7+2\exp(-\|\theta\|^2/2)\le 9\exp\Bigl(-\frac{\|\theta\|^2}{20}\Bigr).
\]
For the final inequality we used the convexity of $\exp(-x^2)$. We substitute this bound into our expression \eqref{eq:ProductFormula} for $F_Y$, which gives for $Y=10^k$
\begin{align*}
F_{Y}(t)&=\prod_{i=0}^{k-1}\frac{1}{9}\Bigl|\sum_{n_i\in\{0,\dots,9\}\backslash \{a_0\}}e(n_i 10^it)\Bigr|\\
&\le \exp\Bigl(-\frac{1}{20}\sum_{i=0}^{k-1}\|10^i t\|^2\Bigr).
\end{align*}
If $t=a/q_1q_2$ with $q_1>1$, $(q_1,10)=1$ and $(a,q_1)=1$, then $\|10^i t\|\ge 1/q_1q_2$ for all $i$. Similarly, if $t=a/q_1q_2+\eta$ with $a,q_1,q_2$ as above, with $|\eta|<Y^{-2/3}/2$ and with $q=q_1q_2<Y^{1/3}$ then for $i\le k/3$ we have $\|10^i t\|\ge 1/q-10^i|\eta|\ge 1/2q$. However, if $\|10^i t\|<1/20$ then $\|10^{i+1}t\|=10\|10^i t\|$. Thus, for any interval $\mathcal{I}\subseteq[0,k/3]$ of length $\log{q}/\log{10}$, there must be some integer $i\in\mathcal{I}$ such that $\|10^i (a/q+\eta)\|>1/200$. This implies that
\[
\sum_{i=0}^k\Bigl\|10^i\Bigl(\frac{a}{q}+\eta\Bigr)\Bigr\|^2\ge \frac{1}{10^5}\Bigl\lfloor\frac{\log{Y}}{3\log{q}}\Bigr\rfloor.
\]
Substituting this into the bound for $F$, and recalling we assume $q<Y^{1/3}$ gives the result.
\end{proof}
%
%
%
%
\begin{lmm}[Markov moment bound]\label{lmm:Markov}
Let $J$ be a positive integer. Let $\lambda_{t,J}$ be the largest eigenvalue of the $10^J\times10^J$ matrix $M_{t}$, given by
\[
(M_{t})_{i,j}=
\begin{cases}
G(a_1,\dots,a_{J+1})^t, &\text{if }i-1=\sum_{\ell=1}^J a_{\ell+1}10^{\ell-1},\, j-1=\sum_{\ell=1}^J a_\ell 10^{\ell-1}\\
&\text{ for some }a_1,\dots,a_{J+1}\in\{0,\dots 9\},\\
0, &\text{otherwise,}
\end{cases}
\]
where
\[
 G(t_0,\dots,t_{J})=\sup_{|\gamma|\le 10^{-J-1} }\frac{1}{9}\Bigl|\frac{e(\sum_{j=0}^{J}t_{j}10^{-j}+10\gamma)-1}{e(\sum_{j=0}^{J}t_{j}10^{-j-1}+\gamma)-1}-e\Bigl(\sum_{j=0}^{J}\frac{a_0 t_{j}}{10^{j+1}}+a_0\gamma\Bigr)\Bigr|.
\]
Then we have that
\[
\sum_{0\le a<10^k}F_{10^k}\Bigl(\frac{a}{10^k}\Bigr)^t\ll_{J,t} \lambda_{t,J}^k.
\]
\end{lmm}
%
%
%
%
\begin{proof}
We recall the product formula \eqref{eq:FFactorization} with $Y=10^k$
\[
F_{Y}(\theta)=\prod_{i=1}^k\frac{1}{9}\Bigl|\frac{e(10^i\theta)-1}{e(10^{i-1}\theta)-1}-e(a_010^{i-1}\theta)\Bigr|,
\]
where we interpret the term in parentheses as $9$ if $\|10^{i-1}\theta\|=0$. Writing $\theta=\sum_{i=1}^k t_i 10^{-i}$ for $t_i\in\{0,\dots,9\}$, we see that the $(k-j)^{th}$ term in the product depends only on $t_{k-j},\dots,t_k$. Moreover, the value of the term is mainly dependent on the first few of these digits by continuity. Thus we may approximate the absolute value of $F_Y(\theta)$ by a product where the $j^{th}$ term depends only on $t_{j},\dots,t_{j+J}$ for some constant $J$. Explicitly, we have
\begin{align*}
F_{Y}\Bigl(\sum_{i=1}^{k}\frac{t_i}{10^i}\Bigr)
&\le \prod_{i=1}^k\sup_{|\gamma|\le 10^{-J-1} }\frac{1}{9}\Bigl|\frac{e(\sum_{j=0}^{J}\frac{t_{i+j}}{10^{j}}+10\gamma)-1}{e(\sum_{j=0}^{J}\frac{t_{i+j}}{10^{j+1}}+\gamma)-1}-e\Bigl(a_0\sum_{j=0}^{J}\frac{t_{i+j}}{10^{j+1}}+a_0\gamma\Bigr)\Bigr|\\
&=\prod_{i=1}^k G(t_i,\dots,t_{i+J}),
\end{align*}
where we put $t_j=0$ for $j>k$.

With this formulation we can interpret the above bound in terms of the probability of a walk on $\{0,\dots,9,\infty\}^k$. Let $t\in\mathbb{R}$ be given. Consider an order-$J$ Markov chain $X_1,X_2,\dots$ where for $a,a_1,\dots,a_n\in\{0,\dots,9\}$ we have for $n>J$
\[
\mathbb{P}(X_{n}=a|X_{n-i}=a_i\text{ for }1\le i\le J)= c G(a,a_1,a_2,\dots,a_J)^t
\]
for some suitably small constant $c$ (so that the probability that $X_n\in\{0,\dots,9\}$ is less than 1). To make this a genuine Markov chain we choose the probability that $X_n=\infty$ given $X_{n-1},\dots,X_{n-J}$ to be such that the probabilities add up to 1, and if $X_n=\infty$ then we have that $X_{n+1}=\infty$ with probability 1.

Then we have that
\[
F_{Y}\Bigl(\sum_{i=1}^k\frac{a_i}{10^{i-1}}\Bigr)^t\le c^{-k}\mathbb{P}(X_i=a_{k+J+1-i}\text{ for }J< i\le k+J|X_1=\dots=X_J=0).
\]
The sum (over all paths in $\{0,\dots,9\}^k$) of the probabilities of paths is a linear combination of the entries in the $k^{th}$ power of the transition matrix restricted to $\{0,\dots,9\}$. Thus such a moment estimate is a linear combination of the $k^{th}$ power of the eigenvalues of this matrix. This allows us to estimate any moment of $F_{Y}(a/Y)$ over $a\in[0,Y)$ uniformly for all $k$ by performing a finite eigenvalue calculation. In particular, this gives us a (arbitrarily good as $J$ increases) numerical approximation to the distribution function of $F_Y$.

Explicitly, let $M_{t}$ be the $10^J\times10^J$ matrix given by
\[
(M_{t})_{i,j}=
\begin{cases}
G(a_1,\dots,a_{J+1})^t, &\text{if }i-1=\sum_{\ell=1}^J a_{\ell+1}10^{\ell-1},\, j-1=\sum_{\ell=1}^J a_\ell 10^{\ell-1}\\
&\text{ for some }a_1,\dots,a_{J+1}\in\{0,\dots, 9\},\\
0, &\text{otherwise,}
\end{cases}
\]
and let $\lambda_{t,J}$ be the absolute value of the largest eigenvalue of $M_t$. Since $G(t_1,\dots,t_{J+1})>0$ for all $t_1,\dots,t_{J+1}$, we have that $M_t$ is irreducible, and so each eigenspace corresponding to an eigenvalue of modulus $\lambda_{t,J}$ has dimension 1 by the Perron-Frobenius Theorem. Let $(M_t)_{i,j}=m_{i,j}$. By expanding out the $k^{th}$ power, we have
\[
(M_t^k)_{i,j}=\sum_{i_1,\dots,i_{k-1}\in\{0,\dots,10^J-1\}}m_{i,i_1}m_{i_1,i_2}\cdots m_{i_{k-1},j}.
\]
We recall that $m_{i,j}=0$ unless there is $a_1,\dots,a_{J+1}\in\{0,\dots,9\}$ such that
\begin{align*}
i-1&=a_2+10a_3+\dots+10^{J-1}a_{J+1},\\
j-1&=a_1+10a_2+\dots+10^{J-1}a_J.
\end{align*}
Thus the product $m_{i,i_1}m_{i_1,i_2}\cdots m_{i_{k-1},j}$ is non-zero only if there are $a_1,\dots,a_{k+J}\in\{0,\dots,9\}$ such that
\begin{align*}
j-1&=a_1+10a_2+\dots+10^{J-1}a_J,\\
i_{k-1}-1&=a_2+10a_3+\dots+10^{J-1}a_{J+1},\\
\vdots\\
i_1-1&=a_k+10a_{k+1}+\dots+10^{J-1}a_{J+k-1},\\
i-1&=a_{k+1}+10a_{k=2}+\dots+10^{J-1}a_{J+k}.
\end{align*}
If this is the case then we have
\[
m_{i,i_1}m_{i_1,i_2}\cdots m_{i_{k-1},j}=\prod_{i=1}^k G(a_i,a_{i+1},\dots,a_{i+J})^t.
\]
Thus, fixing $i=1$ so that $a_{k+1}=\dots=a_{J+k}=0$, and summing over $j$, we have that
\begin{align*}
\sum_{j=0}^{10^J-1}(M_{t}^k)_{1,j}&=\sum_{i_1,\dots,i_{k-1},j\in\{0,\dots,10^J-1\}}m_{1,i_1}m_{i_1,i_2}\cdots m_{i_{k-1},j}\\
&=\sum_{\substack{a_1,\dots,a_k\in\{0,\dots,9\}\\ a_{k+1}=\dots =a_{k+J}=0}}G(a_1,\dots,a_{J+1})^t\cdots G(a_k,\dots,a_{k+J})^t\\
&\ge \sum_{a=0}^{10^k-1}F_Y\Bigl(\frac{a}{10^k}\Bigr)^t.
\end{align*}
On the other hand, by the eigenvalue expansion of $M_{t}$, we have
\[
\sum_{j=0}^{10^J-1}(M_t^k)_{1,j}\ll_{t,J}\lambda_{t,J}^k.
\]
This gives the result.
\end{proof}
%
%
%
%
\begin{lmm}[$\ell^1$ bound]\label{lmm:L1Bound}
We have for any $k\in\mathbb{N}$
\[
\sum_{\mathbf{t}\in\{0,\dots,9\}^k}\prod_{i=1}^k G(t_i,\dots,t_{i+4})\ll 10^{27k/77}.
\]
In particular, we have for $Y_1\asymp Y_2\asymp Y_3$
\[
\sup_{\beta\in\mathbb{R}}\sum_{a<Y_1}F_{Y_2}\Bigl(\beta+\frac{a}{Y_3}\Bigr)\ll Y_1^{27/77},
\]
and
\[
\int_0^1 F_{Y}(t)d t\ll \frac{1}{Y^{50/77}}. 
\]
\end{lmm}
Here $27/77\approx 0.35$ is slightly larger than 1/3, and $50/77\approx 0.65$.
%
%
%
%
\begin{proof}
This follows from Lemma \ref{lmm:Markov} and a numerical bound on $\lambda_{1,4}$. Specifically, by Lemma \ref{lmm:Markov} taking $J=4$ we find
\begin{equation}
\sum_{\mathbf{t}\in\{0,\dots,9\}^k}\prod_{i=1}^{k-4} G(t_i,\dots,t_{i+J})\le\sum_{j} (M_1^{k-4})_{1,j}\ll \lambda_{1,4}^{k}.
\label{eq:GL1Bound}
\end{equation}
A numerical calculation\footnote{A Mathematica\textregistered\;file detailing this computation is included with this article on \url{arxiv.org}.} reveals that
\begin{equation}
\lambda_{1,4}< 2.24190< 10^{27/77}
\label{eq:Lambda14Calc}
\end{equation}
for all choices of $a_0\in\{0,\dots,9\}$. Thus, letting $Y=10^k$ we have $\lambda_{1,4}^k<Y^{27/77}$, which gives the first result.  

For the second bound, let $U_1=\max(1,Y_3/Y_2)$. Since $Y_3\asymp Y_2$, we have $U_1\ll 1$. Any $a<Y_1$ can be written as $a=a_1+U_1a_2+Y_3 a_3$ for some $0\le a_1< U_1\ll 1$, $0\le a_2<Y_3/U_1=\min(Y_3,Y_2)$ and $0\le a_3< Y_1/Y_3\ll 1$. Since there are $O(1)$ choices of $a_1,a_3$ and these can be absorbed into the supremum over $\beta$, we see that it suffices to show
\[
\sup_{\beta\in\mathbb{R}}\sum_{a_2<\min(Y_2,Y_3)}F_{Y_2}\Bigl(\beta+\frac{a_2}{Y_2}\Bigr)\ll Y_2^{27/77}.
\]
Since $F_{Y_2}\ge 0$ we can extend the summation to $a_2<Y_2$. Thus without loss of generality we may assume that $Y_1=Y_2=Y_3=Y=10^k$. We see that
\begin{align}
F_Y\Bigl(\sum_{i=1}^{k}\frac{t_i}{10^i}+\eta\Bigr)
&\le\prod_{i=1}^{k-4}\Bigl( G(t_i,\dots,t_{i+4})+O(10^{i-1}\eta) \Bigr)\nonumber\\
&=(1+O_J(Y\eta))\prod_{i=1}^{k-4} G(t_i,\dots,t_{i+4}).\label{eq:FGBound}
\end{align}
Here we used the fact that $G(t_i,\dots,t_{i+4})$ is bounded away from 0  for all $t_1,\dots,t_{k}\in\{0,\dots,9\}$ since it is the maximal absolute value of a trigonometric polynomial over an interval. Since $F$ is periodic modulo 1 we see that
\[
\sup_{\beta\in\mathbb{R}}\sum_{\mathbf{t}\in\{0,\dots,9\}^k}F_Y\Bigl(\sum_{i=1}^{k}\frac{t_i}{10^i}+\beta\Bigr)=\sup_{\eta\in[0,Y^{-1}]}\sum_{\mathbf{t}\in\{0,\dots,9\}^k}F_Y\Bigl(\sum_{i=1}^{k}\frac{t_i}{10^i}+\eta\Bigr),
\]
and so the second bound of the lemma follows from \eqref{eq:FGBound}, \eqref{eq:GL1Bound} and \eqref{eq:Lambda14Calc} on letting $a=\sum_{i=1}^k t_i/10^i$. For the final bound we integrate \eqref{eq:FGBound} over $\eta\in[0,Y^{-1}]$ and sum over $t_1,\dots,t_k\in\{0,\dots,9\}$, giving
\begin{align*}
\int_0^1 F_Y(t)d t&=\sum_{a=0}^{Y-1}\int_0^{1/Y}F_Y(a/Y+\eta)d\eta\\
&\ll \frac{1}{Y}\sum_{\mathbf{t}\in\{0,\dots,9\}^k}\prod_{i=1}^{k-4} G(t_i,\dots,t_{i+4})\\
&\ll \frac{1}{Y^{50/77}}.\qedhere
\end{align*}
\end{proof}
%
%
%
%
\begin{lmm}[$235/154^{th}$ moment bound]\label{lmm:DigitDistribution}
We have that
\[
\#\Bigl\{0\le a<Y:\,F_Y\Bigl(\frac{a}{Y}\Bigr)\sim \frac{1}{B}\Bigr\}
\ll B^{235/154}Y^{59/433}.
\]
\end{lmm}
Here $235/154\approx 1.5$ and $59/433\approx 0.14$.
%
%
%
%
We recall that $n\sim X$ means that $X/10<n\le X$.
\begin{proof}
This follows from Lemma \ref{lmm:Markov} and a numerical bound for $\lambda_{235/154,4}$. Explicitly, we take $J=4$ and $Y=10^k$. By Lemma \ref{lmm:Markov} we have
\begin{align*}
\#\Bigl\{0\le a<Y:\,F_Y\Bigl(\frac{a}{Y}\Bigr)\sim \frac{1}{B}\Bigr\}&\le  B^{235/154}\sum_{0\le a<Y}F_{Y}\Bigl(\frac{a}{Y}\Bigr)^{235/154}\\
&\ll B^{235/154}\lambda_{235/154,4}^k.
\end{align*}
A numerical calculation\footnote{A Mathematica\textregistered\;file detailing this computation is included with this article on \url{arxiv.org}.} reveals that
\[
\lambda_{235/154,4}<1.36854<10^{59/433},
\]
for all choices of $a_0\in\{0,\dots,9\}$. Substituting this in the bound above gives the result.
\end{proof}
%
%
%
%
\begin{lmm}[Large sieve estimates]\label{lmm:LargeSieve}
We have
\begin{align*}
\sup_{\beta\in\mathbb{R}}\sum_{a\le q}\sup_{|\eta|< \delta}F_{Y}\Bigl(\frac{a}{q}+\beta+\eta\Bigr)&\ll \Bigl(1+\delta q\Bigr)\Bigl(q^{27/77}+ \frac{q}{Y^{50/77}}\Bigr),\\
\sup_{\beta\in\mathbb{R}}\sum_{q\le Q}\sum_{\substack{0<a<q\\ (a,q)=1}}\sup_{|\eta|<\delta}F_{Y}\Bigl(\frac{a}{q}+\beta+\eta\Bigr)&\ll  \Bigl(1+\delta Q^2\Bigr)\Bigl(Q^{54/77}+\frac{Q^2}{Y^{50/77}}\Bigr),
\end{align*}
and for any integer $d$, we have
\begin{align*}
\sup_{\beta\in\mathbb{R}}\sum_{\substack{q\le Q\\ d|q}}\sum_{\substack{0<a<q\\ (a,q)=1}}\sup_{|\eta|<\delta}F_{Y}\Bigl(\frac{a}{q}+\beta+\eta\Bigr)&\ll \Bigl(1+\frac{\delta Q^2}{d}\Bigr) \Bigl(\Bigl(\frac{Q^2}{d}\Bigr)^{27/77}+\frac{Q^2}{d Y^{50/77}}\Bigr).
\end{align*}
\end{lmm}
%
%
%
%
\begin{proof}
For each $a\le q$, let $|\eta_{a}|$ maximize $F_U(a/q+\eta)$ over $|\eta|<\delta$. Since the fractions $a/q$ are all separated from one another by at least $1/q$, we have for any $t$
\[
\#\Bigl\{a\le q:\,\eta_a+\frac{a}{q}\in \Bigl[t-\frac{1}{2q},t+\frac{1}{2q}\Bigr]\Bigr\}\ll 1+q\delta.
\]
Thus, considering $t=b/q-\beta$, we see that
\begin{equation}
\sum_{a\le q}\sup_{|\eta|<\delta}F_{U}\Bigl(\frac{a}{q}+\beta+\eta\Bigr)\ll (1+q\delta)\sum_{b\le q}\sup_{|\eta|\le 1/2q}F_{U}\Bigl(\frac{b}{q}+\eta\Bigr).
\label{eq:DeltaBound}
\end{equation}
We have that
\[
F_{U}(t)=F_{U}(s)+\int_s^t F_{U}'(v)d v.
\]
Thus integrating over $s\in [t-\gamma,t+\gamma]$ for some $\gamma>0$,  we have
\[
F_{U}(t)\ll \frac{1}{\gamma}\int_{t-\gamma}^{t+\gamma}F_{U}(s)d s+\int_{t-\gamma}^{t+\gamma}|F_{U}'(s)|d s.
\]
This implies that
\begin{equation*}
\sup_{|\eta|\le \gamma}F_U(t+\eta)\ll \frac{1}{\gamma}\int_{t-2\gamma}^{t+2\gamma}F_U(s)ds+\int_{t-2\gamma}^{t+2\gamma}|F_U'(s)|ds.
\end{equation*}
Taking $\gamma=1/2q$, we obtain
\begin{align}
\sum_{b\le q}\sup_{|\eta|\le 1/2q}F_{U}\Bigl(\frac{b}{q}+\eta\Bigr)&\ll \sum_{b\le q}\Bigl(Q\int_{b/q-1/q}^{b/q+1/q}F_U(s)d s+\int_{b/q-1/q}^{b/q+1/q}|F'_U(s)|d s\Bigr)\nonumber \\
&\ll q\int_0^1F_{U}(t)d t+\int_{0}^{1}|F_{U}'(t)|d t.
\label{eq:GallagherBound}
\end{align}
Writing $U=10^u$ and $n=\sum_{i=0}^{u-1}n_i 10^i$, we see that
\[
|F_{U}'(t)|=\frac{2\pi}{9^{u}}\Bigl|\sum_{n< 10^u} n\mathbf{1}_{\mathcal{A}}(n)e(n t)\Bigr|.
\]
Writing $n=\sum_{j=0}^{u-1}n_j10^{j-1}$ and using the triangle inequality, we have
\begin{align*}
|F_U'(t)|&\le \frac{2\pi }{ 9^{u}}\sum_{j=0}^{u-1}10^j\Bigl|\sum_{0\le n_j<10}n_j\mathbf{1}_{\mathcal{A}}(n_j)e(n_j 10^{j}t)\Bigr|\prod_{\substack{0\le i\le u-1\\ i\ne j}}\Bigl|\sum_{0\le n_i<10}\mathbf{1}_{\mathcal{A}}(n_i)e(n_i 10^i t)\Bigr|\\
&\ll \frac{10^u}{9^u}\sup_{j\le u}\prod_{\substack{0\le i\le u-1\\ i\ne j}}\Bigl|\sum_{0\le n_i<10}\mathbf{1}_{\mathcal{A}}(n_i)e(n_i 10^i t)\Bigr|.
\end{align*}
We recall the function $G$ from Lemma \ref{lmm:Markov}. Since $G(t_1,\dots,t_{1+J})$ is bounded away from 0, we see that for $\eta\ll U^{-1}$
\begin{align*}
\Bigl|F'_{U}\Bigl(\sum_{i=1}^u\frac{t_i}{10^i}+\eta\Bigr)\Bigr|
&\ll U \prod_{i=1}^u\Bigl( G(t_i,\dots,t_{i+J})+O(10^i\eta)\Bigr)\\
&\ll (U+O(U^2\eta))\prod_{i=1}^u G(t_i,\dots,t_{i+J}).
\end{align*}
Thus, integrating over $\eta\in[0,U^{-1}]$, taking $J=4$, and using Lemma \ref{lmm:L1Bound}, we obtain
\begin{align}
\int_0^1|F'_U(t)|d t\ll \sum_{\mathbf{t}\in\{0,\dots,9\}^u}\prod_{i=1}^u G(t_i,\dots,t_{i+4})\ll U^{27/77}.
\label{eq:FpBound}
\end{align}
By Lemma \ref{lmm:L1Bound} we have
\begin{equation}
\int_0^1F_U(t)d t\ll \frac{1}{U^{50/77}}.\label{eq:FIntBound}
\end{equation}
Combining \eqref{eq:FIntBound}, \eqref{eq:FpBound}, \eqref{eq:GallagherBound} and \eqref{eq:DeltaBound}, we obtain
\[
\sum_{a\le q}\sup_{|\eta|<\delta}F_U\Bigl(\frac{a}{q}+\beta+\eta\Bigr)\ll \Bigl(1+\delta q\Bigr)\Bigl(U^{27/77}+\frac{q}{U^{50/77}}\Bigr).
\]
Combining this with the trivial bound
\[
F_{Y}(t)\le F_{U}(t)
\]
for $U\le Y$, and choosing $U$ maximally subject to $U\le q$ and $U\le Y$  gives the first result of the lemma.

The other bounds follow from entirely analogous arguments. In particular we note that for $(a,q)=1$, $q<Q$, the numbers $a/q$ are separated from one another by $1/Q^2$, and those with $d|q$ are separated from each other by $d/Q^2$, so we have the equivalent of \eqref{eq:DeltaBound} with $\delta q$ replaced by $\delta Q^2$ or $\delta Q^2/d$ and $|\eta|\le 1/2q$ replaced by $|\eta|\le 1/2Q^2$ or $|\eta|\le d/2Q^2$.
\end{proof}
%
%
%
%
\begin{lmm}[Hybrid Bounds] \label{lmm:Hybrid}
Let $E\ge 1$. Then we have
\begin{align*}
\sum_{a\le q}\sum_{\substack{|\eta|\le E/Y\\ (\eta+a/q)Y\in\mathbb{Z}}}F_Y\Bigl(\frac{a}{q}+\eta\Bigr)&\ll (q E)^{27/77}+\frac{q E}{Y^{50/77}},\\
\sum_{\substack{q<Q\\ d|q}}\sum_{\substack{a\le q\\ (a,q)=1}}\sum_{\substack{|\eta|\le E/Y\\ (\eta+a/q)Y\in\mathbb{Z}}}F_Y\Bigl(\frac{a}{q}+\eta\Bigr)&\ll  \Bigl(\frac{Q^2E}{d}\Bigr)^{27/77}+\frac{Q^2E}{d Y^{50/77}}.
\end{align*}
\end{lmm}
%
%
%
%
In the above lemma, we emphasize that $a,q,d$ are all integers, bu the summation over $\eta$ is over real numbers which are well-spaced from the condition $Y(\eta+a/q)\in\mathbb{Z}$.
%
%
%
%
\begin{proof}
We first note that the summand $a/q+\eta$ runs through fractions $b/Y$ with $|b|\le E+Y$ since we have the condition $(\eta+a/q)Y\in\mathbb{Z}$. Each fraction $b/Y$ is represented $O(1+\min(q E/Y,q))$ times, since if $a_1/q+\eta_1=a_2/q+\eta_2$ then $a_2=a_1+O(q E/Y)$ and $\eta_2$ is determined by $a_1,a_2,\eta_1$. There are $O(1+E/Y)$ choices of $b$ giving the same fraction $\Mod{1}$, and since $F_Y$ is periodic $\Mod{1}$ these all give the same value of $F_Y(b/Y)$. Thus we may consider only $b<Y$ with each fraction $b/Y$ occurring $O((1+E/Y)\min(q E/Y,q))$ times. Thus we see that if $10 q E\ge Y$ then
\begin{align*}
\sum_{a\le q}\sum_{\substack{|\eta|\le E/Y\\ (\eta+a/q)Y\in\mathbb{Z}}}F_Y\Bigl(\frac{a}{q}+\eta\Bigr)&\ll \min\Bigl(\frac{q E}{Y},q\Bigr)\Bigl(1+\frac{E}{Y}\Bigr)\sum_{0\le b<Y}F_Y\Big(\frac{b}{Y}\Bigr)\\
&\ll \frac{q E}{Y}\sum_{0\le b<Y}F_Y\Big(\frac{b}{Y}\Bigr).
\end{align*}
In this case the result now follows from Lemma \ref{lmm:L1Bound}. Thus we may assume $q E<Y/10$.

Using the product formula \eqref{eq:FFactorization}, we have for $Y\ge UV$ powers of 10
\[
F_{Y}(\theta)=F_{U}(\theta)F_{V}(U \theta)F_{Y/UV}(UV\theta).
\]
We also have the trivial bound $F_{V}(U\theta)\le 1$ of \eqref{eq:FyBound}. For $UV\le Y$ and $|\eta|<E/Y$ these give
\[
F_{Y}\Bigl(\frac{a}{q}+\eta\Bigr)\le F_{Y/UV}\Bigl(\frac{U V a}{q}+UV\eta\Bigr)\sup_{|\gamma|\le E/Y}F_{U}\Bigl(\frac{a}{q}+\gamma\Bigr).
\]
We choose $V$ and then $U$ to be the largest powers of 10 such that $V\le Y/q E$ and $U\le Y/V E$. Note that this choice gives $U,V\ge 1$ since $q E<Y/10$ and $q,E\ge 1$. Thus
\begin{align*}
\sum_{a\le q}\sum_{\substack{|\eta|\le E/Y\\ (\eta+a/q)Y\in\mathbb{Z}}}F_Y\Bigl(\frac{a}{q}+\eta\Bigr)& \le \sum_{a\le q}\sup_{|\gamma|\le E/Y}F_{U}\Bigl(\frac{a}{q}+\gamma\Bigr)\sum_{\substack{|\eta|\le E/Y\\ (\eta+a/q)Y\in\mathbb{Z}}}F_{Y/UV}\Bigl(\frac{U V a}{q}+UV\eta\Bigr)\\
&\le \Sigma_1\Sigma_2,
\end{align*}
where
\begin{align*}
\Sigma_1&=\sum_{a\le q}\sup_{|\gamma|\le E/Y}F_{U}\Bigl(\frac{a}{q}+\gamma\Bigr),\\
\Sigma_2&=\sup_{\beta\in\mathbb{R}}\sum_{\substack{|\eta|\le E/Y\\ Y(\eta+\beta)\in\mathbb{Z}}}F_{Y/UV}\Bigl(U V\beta+UV\eta\Bigr)\\
&\le \sup_{\beta'\in\mathbb{R}}\sum_{a\le 2E}F_{Y/UV}\Bigl(\beta'+\frac{U V a}{Y}\Bigr).
\end{align*}
Since we chose $U$ and $V$ maximally, we have $V\ge Y/10q E$, so $q/100\le U\le 10q$. Since $q E<Y/10$, we may extend the supremum in $\Sigma_1$ to $\gamma\le 1/10q$ for an upper bound. Thus, by Lemma \ref{lmm:LargeSieve} we have
\[
\Sigma_1\ll q^{27/77}.
\]
Similarly, since $Y/UV\asymp E$, by Lemma \ref{lmm:L1Bound} we have
\[
\Sigma_2\ll E^{27/77}.
\]
Putting this together gives the first result.

The second bound follows from an entirely analogous argument. We first split the argument depending on whether $Q^2E/d\ge Y/10$ or not, and use the final bound of Lemma \ref{lmm:LargeSieve} instead of the first bound to handle $\Sigma_2$.
\end{proof}
%
%
%
%
%
The argument giving the first bound of Lemma \ref{lmm:Hybrid} is essentially sharp if the $\ell^1$ bounds used in the proof are sharp and if $q$ is a divisor of a power of 10 or if $Q E\ge Y$. When $Q E\le Y^{1-\epsilon}$ and $q$ is not a divisor of a power of 10, however, we trivially bounded a factor $F_V(U(a/q+\eta))$ by 1 in the proof, which we expect not to be tight. Lemma \ref{lmm:Hybrid2} below allows us to obtain superior bounds (in certain ranges) provided the denominators do not have large powers of 2 or 5 dividing them.
%
%
%
%
\begin{lmm}[Alternative Hybrid Bound] \label{lmm:Hybrid2}
Let $D,E,Y,Q_1\ge 1$ be integral powers of 10 with $DE\ll Y$. Let $q_1\sim Q_1$ with $(q_1,10)=1$ and let $d\sim D$ satisfy $d|10^u$ for some $u\ge 0$. Let
\[
S=S(d,q_1,Q_2,E,Y)=\sum_{\substack{q_2\sim  Q_2\\ (q_2,10)=1}}\sum_{\substack{a< d q_1q_2\\ (a,d q_1q_2)=1}}\sum_{\substack{|\eta|\le E/Y\\ (\eta+a/q_1q_2d)Y\in\mathbb{Z}}}F_Y\Bigl(\frac{a}{d q_1q_2}+\eta\Bigr).
\]
Then we have
\[
S\ll (D E)^{27/77}(Q_1Q_2^2)^{1/21}+\frac{E^{5/6}D^{3/2}Q_1Q_2^2}{Y^{10/21}}.
\]
In particular, if  $q=d q'$ with $(q',10)=1$ and $d|10^u$ for some integer $u\ge 0$, then we have
\[
\sum_{\substack{a< q\\ (a,q)=1}}\sum_{\substack{|\eta|\le E/Y\\ (\eta+a/q)Y\in\mathbb{Z}}}F_Y\Bigl(\frac{a}{q}+\eta\Bigr)\ll (d E)^{27/77}q^{1/21}+\frac{E^{5/6} d^{3/2} q}{Y^{10/21}}.
\]
\end{lmm}
%
%
%
%
For example, if $(q,10)=1$ and $q E$ is a sufficiently small power of $Y$, then we improve the first bound $(q E)^{27/77}$ of Lemma \ref{lmm:Hybrid} in the $q$-aspect to $E^{27/77}q^{1/21}$. This improvement is important for our later estimates.
%
%
%
%
\begin{proof}
Choose $E'\asymp E$ and $D'\asymp D$ with $E',D'\ge 1$ integral powers of 10 such that $E' D'\le Y$. Let $V$ be the largest integral power of 10 such that $V^2\le Y/D' E'$. Since $D' E'\le Y$ we have that $V\ge1$. Let $d=d_1d_2d_3$ where $d_3=(d,D')$ and $d_2d_3=(d,VD')$.

By the periodicity of $F$ modulo one, the fact $(q_1q_2,d)=1$, and the Chinese remainder theorem, we have
\begin{align}
&\sum_{\substack{a< d q_1q_2\\ (a,d q_1q_2)=1}}\sum_{\substack{|\eta|\le E/Y\\ (\eta+a/q_1q_2d)Y\in\mathbb{Z}}}F_Y\Bigl(\frac{a}{d q_1 q_2}+\eta\Bigr)\nonumber\\
&=\sum_{\substack{a'< q_1q_2\\ (a',q_1q_2)=1}}\,\underset{(b_1+d_1b_2+d_1d_2b_3,d)=1}{\sum_{b_1< d_1}\,\sum_{b_2< d_2}\,\sum_{b_3< d_3}}\,\sideset{}{'}\sum_{|\eta|\le E/Y}F_Y\Bigl(\frac{a'}{q_1q_2}+\frac{b_1}{d_1d_2d_3}+\frac{b_2}{d_2d_3}+\frac{b_3}{d_3}+\eta\Bigr),\label{eq:FirstSplit}
\end{align}
where the dash on $\sum'$ indicates that $\eta$ is summed over all reals satisfying
\[
\Bigl(\eta+\frac{a'}{ q_1q_2}+\frac{b_1}{d_1d_2d_3}+\frac{b_2}{d_2d_3}+\frac{b_3}{d_3}\Bigr)Y\in\mathbb{Z}.
\]
By \eqref{eq:FFactorization}, we have $F_{E' D' V^2}(t)=F_{D'}(t)F_{V^2}(D' t)F_{E'}(D' V^2t)$. Since $D' E' V^2\le Y$, we have $F_Y(t)\le F_{D' E' V^2}(t)$. Thus, since $F$ is periodic modulo $1$ and $d_3|D'$ and $d_2d_3|VD'$, we have
\begin{align*}
F_{Y}&\Bigl(\frac{a'}{q_1q_2}+\frac{b_1}{d_1d_2d_3}+\frac{b_2}{d_2d_3}+\frac{b_3}{d_3}+\eta\Bigr)\\
&\le F_{E'}\Bigl(\beta_1+D' V^2\eta\Bigr)\sup_{|\gamma|\le E/Y}F_{D'}\Bigl(\beta_2+\frac{b_3}{d_3}+\gamma\Bigr)F_{V^2}\Bigl(D'\beta_2+D' \gamma\Bigr),
\end{align*}
where
\[
\beta_1=D' V^2\Bigl(\frac{a'}{q_1q_2}+\frac{b_1}{d_1d_2d_3}\Bigr),\qquad \beta_2=\frac{a'}{q_1q_2}+\frac{b_1}{d_1d_2d_3}+\frac{b_2}{d_2d_3}.
\]
Moreover, by \eqref{eq:FFactorization} and Cauchy-Schwarz, we have
\[
F_{V^2}(\theta)=F_V(\theta)F_V(V\theta)\le F_V(\theta)^2+F_V(V\theta)^2.
\]
Since $d_2d_3|D' V$, this gives
\[
F_{V^2}\Bigl(D'\beta_2+D' \gamma\Bigr)\le F_V\Bigl(D'\beta_2+D' \gamma\Bigr)^2+F_V\Big(\beta_3+D' V\gamma\Bigr)^2.
\]
where
\[
\beta_3=\frac{D' V a'}{q_1q_2}+\frac{b_1(D' V/d_2d_3)}{d_1}.
\]
These give
\begin{align*}
&\sum_{\substack{a'< q_1q_2\\ (a',q_1q_2)=1}}\,\underset{(b_1+d_1b_2+d_1d_2b_3,d)=1}{\sum_{b_1< d_1}\,\sum_{b_2< d_2}\,\sum_{b_3< d_3}}\,\sideset{}{'}\sum_{|\eta|\le E/Y}F_Y\Bigl(\frac{a'}{q_1q_2}+\frac{b_1}{d_1d_2d_3}+\frac{b_2}{d_2d_3}+\frac{b_3}{d_3}+\eta\Bigr)\\
&\ll \Sigma_1 \sum_{\substack{a'< q_1q_2\\ (a',q_1q_2)=1}}\,\underset{(b_1+d_1b_2+d_1d_2b_3,d)=1}{\sum_{b_1< d_1}\,\sum_{b_2< d_2}\,\sum_{b_3< d_3}}\sup_{|\gamma|\le E/Y}F_{D'}\Bigl(\beta_2+\frac{b_3}{d_3}+\gamma\Bigr)F_V\Bigl(D'\beta_2+D' \gamma\Bigr)^2\\
&+\Sigma_1 \sum_{\substack{a'< q_1q_2\\ (a',q_1q_2)=1}}\,\underset{(b_1+d_1b_2+d_1d_2b_3,d)=1}{\sum_{b_1< d_1}\,\sum_{b_2< d_2}\,\sum_{b_3< d_3}}\sup_{|\gamma|\le E/Y}F_{D'}\Bigl(\beta_2+\frac{b_3}{d_3}+\gamma\Bigr)F_V\Big(\beta_3+D' V\gamma\Bigr)^2,
\end{align*}
where
\[
\Sigma_1=\sup_{\beta\in\mathbb{R}}\sum_{\substack{|\eta|\le E/Y\\ Y(\eta+\beta)\in\mathbb{Z}}}F_{E'}\Bigl(D' V^2\beta+D' V^2\eta\Bigr)\le\sup_{\beta'\in\mathbb{R}}\sum_{a\le 2E}F_{E'}\Bigl(\beta'+\frac{D' V^2 a}{Y}\Bigr).
\]
Since $(d_1d_2d_3,D')=d_3$ and $(q_1q_2,d)=1$, as $a'$, $b_1$ and $b_2$ go through all residue classes $\Mod{q_1q_2}$, $\Mod{d_1}$ and $\Mod{d_2}$ respectively subject to $(a',q_1q_2)=(b_1+d_1b_2,d_1d_2)=1$, we see that $D'\beta_2$ goes through all values of $c/q_1q_2d_1d_2\Mod{1}$ for $0< c< q_1q_2d_1d_2$ with $(c,q_1q_2d_1d_2)=1$, and each value is attained exactly once. Similarly, since $(d_1d_2d_3,D' V)=d_2d_3$, we see that $\beta_3$ goes through every value of $c/q_1q_2d_1\Mod{1}$ with $0 < c< q_1q_2d_1$ and $(c,q_1q_2d_1)=1$ exactly once as $a$ goes through the values $\Mod{q_1q_2}$ and $b_1$ goes through the values $\Mod{d_1}$ with $(a,q_1q_2)=(b_1,d_1)=1$.

Thus we have
\begin{align*}
 \sum_{\substack{a'< q_1q_2\\ (a',q_1q_2)=1}}\underset{(b_1+d_1b_2+d_1d_2b_3,d)=1}{\sum_{b_1< d_1}\,\sum_{b_2< d_2}\,\sum_{b_3< d_3}}\sup_{|\gamma|\le E/Y}F_{D'}\Bigl(\beta_2+\frac{b_3}{d_3}+\gamma\Bigr)F_V\Bigl(D'\beta_2+D' \gamma\Bigr)^2\ll \Sigma_2\Sigma_3,\nonumber\\
\sum_{\substack{a'< q_1q_2\\ (a',q_1q_2)=1}}\underset{(b_1+d_1b_2+d_1d_2b_3,d)=1}{\sum_{b_1< d_1}\,\sum_{b_2< d_2}\,\sum_{b_3< d_3}}\sup_{|\gamma|\le E/Y}F_{D'}\Bigl(\beta_2+\frac{b_3}{d_3}+\gamma\Bigr)F_V\Big(\beta_3+D' V\gamma\Bigr)^2\ll \Sigma_4\Sigma_5,
\end{align*}
where 
\begin{align*}
\Sigma_2&=\sup_{\beta\in\mathbb{R}}\sum_{b_3< d_3}\sup_{|\gamma|\le E/Y}F_{D'}\Bigl(\frac{b_3}{d_3}+\beta+\gamma\Bigr),\\
\Sigma_3&=\sum_{\substack{a_1< d_1d_2q_1q_2\\ (a_1,d_1d_2q_1q_2)=1}}\sup_{|\gamma|\le E/Y}F_{V}\Bigl(\frac{a_1}{d_1d_2q_1q_2}+D'\gamma\Bigr)^2,\\
\Sigma_4&=\sup_{\beta\in\mathbb{R}}\sum_{b'< d_2d_3}\sup_{|\gamma|\le E/Y}F_{D'}\Bigl(\frac{b'}{d_2d_3}+\beta+\gamma\Bigr),\\
\Sigma_5&=\sum_{\substack{a_2< d_1q_1q_2\\ (a_2,d_1q_1q_2)=1}}\sup_{|\gamma|\le E/Y}F_V\Bigl(\frac{a_2}{d_1q_1q_2}+D' V\gamma\Bigr)^2.
\end{align*}
We note that only $\Sigma_3$ and $\Sigma_5$ depend on $q_2$. Thus, summing over $q_2\sim Q_2$ with $(q_2,10)=1$ we obtain
\begin{equation}
\sum_{\substack{q_2\sim Q_2\\ (q_2,10)=1}}\,\sum_{\substack{a< d q_1q_2\\ (a,d q_1q_2)=1}}\,\sum_{\substack{|\eta|\le E/Y\\ (\eta+a/d q_1q_2)Y\in\mathbb{Z}}}F_Y\Bigl(\frac{a}{q_1q_2d}+\eta\Bigr)\le \Sigma_1(\Sigma_2\Sigma_3'+\Sigma_4\Sigma_5'),
\label{eq:HybridBound}
\end{equation}
where $\Sigma_1$, $\Sigma_2$ and $\Sigma_4$ are as above and $\Sigma_3'$ and $\Sigma_5'$ are given by
\begin{align*}
\Sigma_3'&=\sum_{\substack{q_2\sim Q_2\\ (q_2,10)=1}}\sum_{\substack{a_1< d_1d_2q_1q_2\\ (a_1,d_1d_2q_1q_2)=1}}\sup_{|\gamma|\le E/Y}F_{V}\Bigl(\frac{a_1}{d_1d_2q_1q_2}+D'\gamma\Bigr)^2,\\
\Sigma_5'&=\sum_{\substack{q_2\sim Q_2\\ (q_2,10)=1}}\sum_{\substack{a_2< d_1q_1q_2\\ (a_2,d_1q_1q_2)=1}}\sup_{|\gamma|\le E/Y}F_V\Bigl(\frac{a_2}{d_1q_1q_2}+D' V\gamma\Bigr)^2.
\end{align*}
Since $Y/D' V^2\asymp E\asymp E'$, by Lemma \ref{lmm:L1Bound} we have
\begin{equation}
\Sigma_1\ll E^{27/77}. \label{eq:Sig1Bound}
\end{equation}
We have $d_2d_3\le d\le D$ and $DE\ll Y$, so $E/Y\ll 1/d_2d_3$. Thus, by Lemma \ref{lmm:LargeSieve},  we have
\begin{align}
\Sigma_2&\ll d_3^{27/77},\label{eq:Sig2Bound}\\
\Sigma_4&\ll (d_2d_3)^{27/77}.\label{eq:Sig4Bound}
\end{align}
We are left to bound $\Sigma_3'$ and $\Sigma_5'$, which are very similar. Let
\[
\Sigma'=\Sigma'(q_1,d_1,d_2)=\sum_{\substack{q_2\sim Q_2\\ (q_2,10)=1}}\sum_{\substack{a_1< d_1d_2q_1q_2\\ (a_1,d_1d_2q_1q_2)=1}}\sup_{|\gamma|\le D' E V/Y}F_{V}\Bigl(\frac{a_1}{d_1d_2q_1q_2}+\gamma\Bigr)^2.
\]
We note that $\Sigma'(q_1,d_1,d_2)$ is the same as $\Sigma_3'$ except we have increased the range of the supremum, and so we have $\Sigma_3'\le\Sigma'(q_1,d_1,d_2)$. Moreover, we see that $\Sigma_5'$ is a special case of $\Sigma'$ with $d_2=1$, so $\Sigma_5'=\Sigma'(q_1,d_1,1)$. Thus it will suffice to get suitable bounds on $\Sigma'$.

Since $F_R(\theta)\ge F_V(\theta)$ for $R\le V$, we may replace $F_V$ with $F_R$ where $R=10^r$ is the largest power of 10 less than $\min(V,d_1d_2Q_1Q_2^2)$. Since $R\le V$ and $D' E V/Y\ll1/V$, we see all quantities $\gamma$ occurring in the supremum are of size at most $O(1/R)$. Given any choice of reals $\eta_{a,q_2}\ll 1/R$ for $a\le d_1d_2q_1q_2$ and $q_2\sim Q_2$ with $(a,d_1d_2q_1q_2)=1$, the numbers $a/d_1d_2q_1q_2+\eta_{a,q_2}$ can be arranged into $O(d_1d_2Q_1Q_2^2/R)$ sets such that all numbers in any set are separated by $\gg 1/R$. (Recall that $r$ is chosen such that $R\le  d_1d_2Q_1Q_2^2$.) Thus, as in the proof of Lemma \ref{lmm:LargeSieve} (specifically the argument leading up to \eqref{eq:GallagherBound}), we find that
\begin{align*}
\Sigma'&\le \sum_{\substack{q_2\sim Q_2\\ (q_2,10)=1}}\sum_{\substack{a< d_1d_2q_1q_2\\ (a,d_1d_2q_1q_2)=1}}\sup_{|\eta|\ll 1/R}F_{R}\Bigl(\frac{a}{d_1d_2q_1q_2}+\eta\Bigr)^2&\\
&\ll d_1d_2Q_1Q_2^2\int_0^1 F_{R}(t)^2 d t+\frac{d_1d_2Q_1Q_2^2}{R}\int_0^1|F_{R}'(t)|F_{R}(t) d t.
\end{align*}
By Parseval we have
\[
\int_0^1 F_{R}(t)^2d t=\frac{1}{9^{2r}}\sum_{\substack{a\in\mathcal{A}\\ a\le R}}1=\frac{1}{9^{r}},
\]
and
\[
 \int_0^1F_{R}'(t)^2d t=\frac{1}{9^{2r}}\sum_{\substack{a\in\mathcal{A}_1\\ a\le R}}4\pi^2 a^2\ll \frac{10^{2r}}{9^{r}}.
\]
Using Cauchy-Schwarz and the above bounds, we obtain
\[
\int_0^1|F_R'(t)|F_R(t)d t\ll \Bigl(\int_0^1F_R'(t)^2d t\Bigr)^{1/2}\Bigl(\int_0^1F_R(t)^2d t\Bigr)\ll \frac{R}{9^r}.
\]
Putting this together gives
\[
\Sigma'\ll \frac{d_1d_2Q_1Q_2^2}{9^r}.
\]
We recall that $R=10^{r}\sim \min(V,d_1d_2Q_1Q_2^2)$ and $V\asymp (Y/DE)^{1/2}$, and note that $20/21<\log{9}/\log{10}$. This gives 
\begin{equation}
\Sigma'\ll (d_1d_2Q_1Q_2^2)^{1/21}+d_1d_2Q_1Q_2^2\Bigl(\frac{Y}{D E}\Bigr)^{-10/21}.\label{eq:Sig3Bound}
\end{equation}
This gives a bound for $\Sigma_3'$ since $\Sigma_3'\le \Sigma'$, and we obtain an analogous bound for $\Sigma_5'$ with $d_2$ replaced by 1. Combining \eqref{eq:Sig3Bound} with our earlier bounds \eqref{eq:Sig1Bound}, \eqref{eq:Sig2Bound} and \eqref{eq:Sig4Bound} and substituting these into \eqref{eq:HybridBound} gives
\begin{align*}
\sum_{\substack{q_2\sim Q_2\\ (q_2,10)=1}}\,\sum_{\substack{a< d q_1q_2\\ (a,d q_1q_2)=1}}\,&\sum_{\substack{|\eta|\le E/Y\\ (\eta+a/d q_1q_2)Y\in\mathbb{Z}}}F_Y\Bigl(\frac{a}{q_1q_2d}+\frac{b}{d}+\eta\Bigr)\\
&\ll E^{27/77}\Bigl(D^{27/77}(Q_1Q_2^2)^{1/21}+Q_1Q_2^2D\Bigl(\frac{Y}{D E}\Bigr)^{-10/21}\Bigr).
\end{align*}
Simplifying the exponents by noting $1+10/21<3/2$ and $27/77+10/21<5/6$ then gives the result.

The second statement of the lemma is simply the case when $Q_2=1$ and $q=d q_1$.
\end{proof}
%
%
%
%
We see that Lemma \ref{lmm:LargeSieveTypeI} follows immediately from Lemma \ref{lmm:LargeSieve}, and Lemma \ref{lmm:LInfTypeI} is the same as Lemma \ref{lmm:LInfBound}. Thus we are left to establish  Proposition \ref{prpstn:Major}, Proposition \ref{prpstn:Generic} and Proposition \ref{prpstn:Exceptional}, which we do over the next few sections.

%
%
%
%
%
%
%
\section{Major arcs}\label{sec:Major}
%
%
%
%
In this section we establish Proposition \ref{prpstn:Major} using the prime number theorem in arithmetic progressions and short intervals, making use of Lemma \ref{lmm:LInfBound}.
%
%
%
%
\begin{proof}[Proof of Proposition \ref{prpstn:Major}]
We split $\mathcal{M}$ up as three disjoint sets
\[
\mathcal{M}=\mathcal{M}_1\cup\mathcal{M}_2\cup\mathcal{M}_3,
\]
where
\begin{align*}
\mathcal{M}_1&=\Bigl\{a\in\mathcal{M}:\Bigl|\frac{a}{X}-\frac{b}{q}\Bigr|\le\frac{(\log{X})^C}{X}\text{ for some }\,b,\,q\le (\log{X})^C,\,q\nmid X\Bigr\},\\
\mathcal{M}_2&=\Bigl\{a\in\mathcal{M}:\frac{a}{X}
=\frac{b}{q}+\nu\text{ for some }\,b,\,q\le (\log{X})^C,\,q|X,\,0<|\nu|\le \frac{(\log{X})^C}{X}\Bigr\},\\
\mathcal{M}_3&=\Bigl\{a\in\mathcal{M}:\frac{a}{X}
=\frac{b}{q}\text{ for some }\,b,\,q\le (\log{X})^C,\,q|X\Bigr\}.\\
\end{align*}
By Lemma \ref{lmm:LInfBound} and recalling $X$ is a power of 10, we have
\[
\sup_{a\in\mathcal{M}_1}\Bigl|S_{\mathcal{A}}\Bigl(\frac{a}{X}\Bigr)\Bigr|=\#\mathcal{A}\sup_{a\in\mathcal{M}_1}F_X\Bigl(\frac{a}{X}\Bigr)\ll \#\mathcal{A}\exp(-\sqrt{\log{X}}).
\]
Using the trivial bound $S_{\mathcal{R}_X}(\theta)\ll X(\log{X})^{\ell}$, where $\ell\le 2/\eta$ and noting $\#\mathcal{M}_1\ll (\log{X})^{3C}$, we obtain
\begin{equation}
\frac{1}{X}\sum_{a\in\mathcal{M}_1}S_{\mathcal{A}}\Bigl(\frac{a}{X}\Bigr)S_{\mathcal{R}_X}\Bigl(\frac{-a}{X}\Bigr)\ll_{C,\eta} \frac{\#\mathcal{A}}{(\log{X})^C}.
\label{eq:M1}
\end{equation}
This gives the result for $\mathcal{M}_1$.

We now consider $\mathcal{M}_2$. Recalling the definition of $\mathcal{R}_X$, we have that for $n<X$
\begin{equation}
\Lambda_{\mathcal{R}_X}(n)=\sum_{\substack{n=p_1\cdots p_\ell\\ p_j\in (X^{a_j},X^{a_j+\delta}]\,\text{for $j<\ell$}\\ p_\ell\ge X^{\eta/4},X^{1-\sum_ia_i-\ell\delta}}}\prod_{i=1}^\ell \log{p_i}=\sum_{\substack{n=mp\\ p\ge X^{\eta/4}\\ p\ge X^{1-\sum_ia_i-\ell\delta}}}\Lambda_{\mathcal{C}}(m)\log{p},
\label{eq:RXSum}
\end{equation}
 where $\mathcal{C}=(a_1,a_1+\delta]\times\dots\times(a_{\ell-1},a_{\ell-1}+\delta]$ is the projection of $\mathcal{R}_X$ onto the first $\ell-1$ coordinates. We note the crude bound
\begin{equation}
\sum_{m<X}\frac{\Lambda_{\mathcal{C}}(m)}{m}\le \Bigl(\sum_{p\le X}\frac{\log{p}}{p}\Bigr)^{\ell-1} \ll (\log{X})^{\ell-1}.
\label{eq:Crude}
\end{equation}
Let $\Delta=\lceil\log{X}\rceil^{-10C-10\ell}$. We note that if $a\in\mathcal{M}_2$ then $a/X=b/q+c/X$ for some integers $b,q,|c|\le (\log{X})^C$ ($c$ is an integer since $q|X$ for the set $\mathcal{M}_2$). We separate the sum $S_{\mathcal{R}_X}(a/X)$ by putting the prime variable $p$ occurring in \eqref{eq:RXSum} in short intervals of length $\Delta x/m$ and in arithmetic progressions $\Mod{q}$. We note that $\Lambda_{\mathcal{C}}$ is supported on $m\le X^{\sum_i a_i+(\ell-1)\delta}< X^{1-\eta/3}$, so we can drop the constraints $p\ge X^{\eta/4},X^{1-\sum_{i}a_i-\ell\delta}$ at the cost of some terms with $mp<X^{1-\eta/12}+X^{1-\delta}$. Thus we have
\begin{align*}
\sup_{a\in\mathcal{M}_2}S_{\mathcal{R}_X}\Bigl(\frac{a}{X}\Bigr)
&=\sup_{a\in\mathcal{M}_2}\sum_{m<X^{1-\eta/3}}\Lambda_{\mathcal{C}}(m)\sum_{p< X/m}(\log{p})e\Bigl(\frac{a m p}{X}\Bigr)\\
&\qquad +O_\ell\Bigl(\sum_{pm<X^{1-\eta/12}+X^{1-\delta}}(\log{X})^{\ell}\Bigr)\\
&=O_{C,\eta}\Bigl(\frac{X}{(\log{X})^{4C}}\Bigr)+\sup_{\substack{1\le b\le q\\ q\le (\log{X})^C \\ 0<|c|\le (\log{X})^C}}\sum_{m<X^{1-\eta/3}}\Lambda_{\mathcal{C}}(m)\sum_{r=0}^{q-1}\sum_{0\le j< \Delta^{-1}}\\ 
&\qquad\times\sum_{\substack{p\in [j\Delta X/m,(j+1)\Delta X/m) \\ p\equiv r\Mod{q}}}(\log{p})e\Bigl(mp\Bigl(\frac{b}{q}+\frac{c}{X}\Bigr)\Bigr).
\end{align*}

If $mp=j\Delta X+O(\Delta X)$ and $p\equiv r\Mod{q}$ we have
\[
e\Bigl(mp\Bigl(\frac{b}{q}+\frac{c}{X}\Bigr)\Bigr)= e\Bigl(\frac{b r m}{q}\Bigr)e(j c\Delta)+O(\Delta (\log{X})^C).
\]
By the prime number theorem in short intervals and arithmetic progressions \eqref{eq:PNT}, for $m<X^{1-\eta/3}$ and $(r,q)=1$ we have
\[
\sum_{\substack{p\in[j\Delta X/m,(j+1)\Delta X/m) \\ p\equiv r\Mod{q}}}\log{p}=\frac{\Delta X}{m\phi(q)}+O_{C,\eta}\Bigl(\frac{\Delta^2 X}{m\phi(q)}\Bigr)
\]
Thus
\begin{align*}
\sup_{a\in\mathcal{M}_2}S_{\mathcal{R}_X}\Bigl(\frac{a}{X}\Bigr)
&=\Delta X\sup_{\substack{b\le q\\ q\le (\log{X})^C \\ c\le (\log{X})^C}}\sum_{m<X^{1-\eta/3}}\frac{\Lambda_{\mathcal{C}}(m)}{m\phi(q)}
\sum_{\substack{1\le r<q\\ (r,q)=1}} e\Bigl(\frac{b r m}{q}\Bigr)\sum_{1\le j< \Delta^{-1}} e(j\Delta c)\\
&\qquad+O_{C,\eta}\Bigl(\frac{X}{(\log{X})^{4C}}\Bigr).
\end{align*}
Finally, since $c\in\mathbb{Z}$ and $c\ne 0$ and $\Delta^{-1}\in\mathbb{Z}$, we have
\[
\sum_{1\le j< \Delta^{-1}}e(j\Delta c)=-e(c)=-1=O(1).
\]
Using \eqref{eq:Crude}, this gives
\begin{align}
\sup_{a\in\mathcal{M}_2}S_{\mathcal{R}_X}\Bigl(\frac{a}{X}\Bigr)&\ll\Delta X (\log{X})^C\sum_{m<X^{1-\eta/3}}\frac{\Lambda_{\mathcal{C}}(m)}{m}+O_{C,\eta}\Bigl(\frac{X}{(\log{X})^{4C}}\Bigr)\nonumber\\
&\ll_{C,\eta} \frac{X}{(\log{X})^{4C}}.\label{eq:M2Bound}
\end{align}
Note that in the above argument for us to be able to save an arbitrary power of log it was important that we are counting elements with weight $\Lambda_{\mathcal{R}_X}(n)$ rather than $\mathbf{1}_{\mathcal{R}_X}(n)$, and that $X\nu\in\mathbb{Z}$ for $a\in\mathcal{M}_2$. 

Using the trivial bounds $S_{\mathcal{A}}(\theta)\le \#\mathcal{A}$ and $\#\mathcal{M}_2\ll (\log{X})^{3C}$ along with \eqref{eq:M2Bound}, we obtain
\begin{equation}
\frac{1}{X}\sum_{a\in\mathcal{M}_2}S_{\mathcal{A}}\Bigl(\frac{a}{X}\Bigr)S_{\mathcal{R}_X}\Bigl(\frac{-a}{X}\Bigr)\ll_{C,\eta} \frac{\#\mathcal{A}}{(\log{X})^C}.
\label{eq:M2}
\end{equation}
 Finally, we consider $\mathcal{M}_3$. By the prime number theorem in arithmetic progressions as above, we have for $(r,q)=1$ and $q\le (\log{X})^C$ that
\begin{align*}
\sum_{\substack{n< X\\ n\equiv r\Mod{q}}}\Lambda_{\mathcal{R}_X}(n)&=\frac{X}{\phi(q)}\sum_{m<X^{1-\eta/3}}\frac{\Lambda_{\mathcal{C}}(m)}{m}+O_{\eta,C}\Bigl(\frac{X}{(\log{X})^{4C}}\Bigr)\\
&=\frac{1}{\phi(q)}\sum_{n< X}\Lambda_{\mathcal{R}_X}(n)+O_{\eta,C}\Bigl(\frac{X}{(\log{X})^{4C}}\Bigr).
\end{align*}
Thus, for $(a,q)=1$
\begin{align*}
S_{\mathcal{R}_X}\Bigl(\frac{a}{q}\Bigr)&=\sum_{0\le r<q}e\Bigl(\frac{a r}{q}\Bigr)\sum_{\substack{n< X\\ n\equiv r\Mod{q}}}\Lambda_{\mathcal{R}_X}(n)\\
&=\frac{1}{\phi(q)}\Bigl(\sum_{n< X}\Lambda_{\mathcal{R}_X}(n)\Bigr)\Bigl(\sum_{\substack{0\le r<q\\ (r,q)=1}}e\Bigl(\frac{a r}{q}\Bigr)\Bigr)+O_{\eta,C}\Bigl(\frac{X}{(\log{X})^{4C}}\Bigr)\\
&=\frac{\mu(q)}{\phi(q)}\sum_{n< X}\Lambda_{\mathcal{R}_X}(n)+O_{\eta,C}\Bigl(\frac{X}{(\log{X})^{4C}}\Bigr).
\end{align*}
Since $\mu(q)=0$ for $q|10^k=X$ unless $q\in\{1,2,5,10\}$, using the trivial bounds $\#\mathcal{M}_3\ll (\log{X})^{2C}$ and $|S_\mathcal{A}(a/X)|\le \#\mathcal{A}$, we obtain
\begin{align}
\frac{1}{X}\sum_{a\in\mathcal{M}_3}&S_{\mathcal{A}}\Bigl(\frac{a}{X}\Bigr)S_{\mathcal{R}_X}\Bigl(\frac{-a}{X}\Bigr)
=\frac{1}{X}\sum_{0\le b<10}S_{\mathcal{A}}\Bigl(\frac{b}{10}\Bigr)S_{\mathcal{R}_X}\Bigl(\frac{-b}{10}\Bigr)+O_{C,\eta}\Bigl(\frac{\#\mathcal{A}}{(\log{X})^C}\Bigr)\nonumber\\
&=\frac{10}{X}\sum_{m\in\mathcal{A}}\sum_{\substack{n<X\\ n\equiv m\Mod{10}}}\Lambda_{\mathcal{R}_X}(n)+O_{C,\eta}\Bigl(\frac{\#\mathcal{A}}{(\log{X})^C}\Bigr)\nonumber\\
&=\frac{10}{\phi(10)}\Bigl(\frac{1}{X}\sum_{n< X}\Lambda_{\mathcal{R}_X}(n)\Bigr)\#\{m\in\mathcal{A}:(m,10)=1\}+O_{C,\eta}\Bigl(\frac{\#\mathcal{A}}{(\log{X})^C}\Bigr)\nonumber\\
&=\kappa_\mathcal{A}\frac{\#\mathcal{A}}{X}\sum_{n<X}\Lambda_{\mathcal{R}_X}(n)+O_{C,\eta}\Bigl(\frac{\#\mathcal{A}}{(\log{X})^C}\Bigr).
\label{eq:M3}
\end{align}
Thus \eqref{eq:M1}, \eqref{eq:M2} and \eqref{eq:M3} gives the result.
\end{proof}
\begin{rmk}
We have only needed to use the prime number theorem in arithmetic progressions when the modulus is a small divisor of $X$, and so has no large prime factors. This means that our implied constants can be taken to be effectively computable since for such moduli we do not need to appeal to Siegel's theorem.
\end{rmk}

%
%
%
%
\section{Generic minor arcs}\label{sec:Generic}
In this section we establish Proposition \ref{prpstn:Generic} and obtain some bounds on the exceptional set $\mathcal{E}$ by using the distributional estimates of Lemma \ref{lmm:DigitDistribution}.
%
%
%
%
\begin{lmm}[$\ell^2$ bound for primes]\label{lmm:PrimeDistribution}
We have that
\[
\#\Bigl\{0\le a<X:\,\Bigl|S_{\mathcal{R}}\Bigl(\frac{a}{X}\Bigr)\Bigr|\sim\frac{X}{C}\Bigr\}\ll C^2(\log{X})^{O_\eta(1)}.
\]
\end{lmm}
%
%
%
%
\begin{proof}
This follows from the $\ell^2$ bound coming from Parseval's identity.
\begin{align*}
\#\Bigl\{0\le a<X:\Bigl|S_{\mathcal{R}}\Bigl(\frac{a}{X}\Bigr)\Bigr|\ge\frac{X}{10C}\Bigr\}
&\ll \frac{C^{2}}{X^2}\sum_{a<X}\Bigl|S_{\mathcal{R}}\Bigl(\frac{a}{X}\Bigr)\Bigr|^2\\
&=\frac{C^2}{X}\sum_{n<X}\Lambda_{\mathcal{R}}(n)^2\\
&\ll C^2(\log{X})^{O_\eta(1)}.\qedhere
\end{align*}
\end{proof}
%
%
%
%
\begin{lmm}[Generic frequency bounds] \label{lmm:Generic}
Let
\[
\mathcal{E}=\Bigl\{0\le a<X:\, F_X\Bigl(\frac{a}{X}\Bigr)\ge \frac{1}{X^{23/80}}\Bigr\}.
\]
Then
\begin{align*}
\#\mathcal{E}&\ll X^{23/40-\epsilon},\\
\sum_{a\in\mathcal{E}}F_X\Bigl(\frac{a}{X}\Bigr)&\ll X^{23/80-\epsilon},
\end{align*}
and
\[
\frac{1}{X}\sum_{\substack{a<X\\ a\notin\mathcal{E}}}\Bigl|F_X\Bigl(\frac{a}{X}\Bigr)S_{\mathcal{R}}\Bigl(\frac{-a}{X}\Bigr)\Bigr|\ll_\eta\frac{1}{X^{\epsilon}}.
\]
\end{lmm}
%
%
%
%
\begin{proof}
The first bound on the size of $\mathcal{E}$ follows from using Lemma \ref{lmm:DigitDistribution} with $B=X^{23/80}$ and verifying that $(23\times 235)/(80\times 154)+59/433<23/40$. For the second bound we see from Lemma \ref{lmm:DigitDistribution} that
\begin{align*}
\sum_{a\in\mathcal{E}} F_X\Bigl(\frac{a}{X}\Bigr)
&\ll \sum_{\substack{j\ge 0\\ 2^j\le X^{23/80}}} \#\Bigl\{0\le a<X:\,F_X\Bigl(\frac{a}{X}\Bigr)\sim 2^{-j}\Bigr\}\\
&\ll \sum_{\substack{j\ge 0\\ 2^j\le X^{23/80}}} 2^{(235/154-1)j} X^{59/433}\\
&\ll X^{59/433+23\times 235/(80\times 154)-23/80},
\end{align*}
and so the calculation above gives the result.

It remains to bound the sum over $a\notin\mathcal{E}$. We divide the sum into $O(\log{X})^2$ subsums where we restrict to those $a$ such that $F_X(a/X)\sim 1/B$ and $|S_{\mathcal{R}}(a/X)|\sim X/C$ for some $B\ge X^{23/80}$ and $C\le X^2$ (terms with $C>X^2$ makes a contribution $O(1/X)$). This gives
\begin{align*}
\frac{1}{X}\sum_{\substack{a<X\\ a\notin\mathcal{E}}}
&\Bigl|F_{X}\Bigl(\frac{a}{X}\Bigr)S_{\mathcal{R}}\Bigl(\frac{-a}{X}\Bigr)\Bigr|\\
&\ll \sup_{\substack{X^{23/80}\le B \\ 1\le C\le X^2}}\frac{(\log{X})^{2}}{X}\sum_{\substack{a<X\\ F_X(a/X)\sim 1/B \\ S_{\mathcal{R}}(-a/X)\sim X/C}}\Bigl|F_X\Bigl(\frac{a}{X}\Bigr)S_{\mathcal{R}}\Bigl(\frac{-a}{X}\Bigr)\Bigr|+\frac{1}{X^2}.
\end{align*}
We concentrate on the inner sum. Using Lemmas \ref{lmm:DigitDistribution} and \ref{lmm:PrimeDistribution} we see that the sum contributes
\begin{align*}
&\ll \frac{X}{B C}\#\Bigl\{a:F_X\Bigl(\frac{a}{X}\Bigr)\sim \frac{1}{B},\,\Bigl|S_{\mathcal{R}}\Bigl(\frac{-a}{X}\Bigr)\Bigr|\sim \frac{X}{C}\Bigr\}\\
&\ll \frac{X(\log{X})^{O_\eta(1)}}{B C} \min\Bigl(C^2,\,B^{235/154}X^{59/433}\Bigr)\\
&\ll_\eta X^{1+\epsilon}\frac{X^{59/866}}{B^{73/308}}.
\end{align*}
Here we used the bound $\min(x,y)\le x^{1/2}y^{1/2}$ in the last line. In particular, we see this is $O_\eta(X^{1-2\epsilon})$ if $B\ge X^{23/80}$ on verifying that $23/80\times 73/308>59/866$. Substituting this into our bound above gives the result.
\end{proof}
%
%
%
%
%
%
%
\section{Exceptional minor arcs}\label{sec:Bilinear}
In this section we reduce Proposition \ref{prpstn:Exceptional} to the task of establishing Proposition \ref{prpstn:LatticeBound} and Proposition \ref{prpstn:LineBound}, given below. We do this by making use of the bilinear structure of $\Lambda_{\mathcal{R}_X}(n)$ which is supported on integers of the form $n_1n_2$ with $n_1$ of convenient size, and then showing that if these resulting bilinear expressions are large then the Fourier frequencies must lie in a smaller additively structured set. Propositions \ref{prpstn:LatticeBound} and \ref{prpstn:LineBound} then show that we have superior Fourier distributional estimates inside such sets. Thus we conclude that the bilinear sums are always small. To make the bilinear bound explicit, we establish the following lemma, from which Proposition \ref{prpstn:Exceptional} follows quickly.
%
%
%
%
\begin{lmm}[Bilinear sum bound]\label{lmm:Explicit}
Let $N,M,Q\ge 1$ and $E$ satisfy $X^{9/25}\le N\le X^{17/40}$, $Q\le X^{1/2}$, $NM\le 1000X$ and $E\le 100X^{1/2}/Q$, and either $E\ge 1/X$ or $E=0$.
Let $\mathcal{F}=\mathcal{F}(Q,E)$ be given by
\[
\mathcal{F}=\Bigl\{a<X:\, \frac{a}{X}=\frac{b}{q}+\nu\text{ for some }(b,q)=1\text{ with }q\sim Q,\,\nu\sim E/X\Bigr\}.
\]
Then for any complex 1-bounded complex sequences $\alpha_n,\beta_m,\gamma_a$ we have
\[
\sum_{a\in\mathcal{F}\cap\mathcal{E}}\sum_{\substack{n\sim N\\ m\sim M}}F_{X}\Bigl(\frac{a}{X}\Bigr)\alpha_n\beta_m\gamma_a e\Bigl(\frac{-anm}{X}\Bigr)\ll \frac{X(\log{X})^{O(1)}}{(Q+E)^{\epsilon/10}}.
\]
\end{lmm}
\begin{proof}[Proof of Proposition \ref{prpstn:Exceptional} assuming  Lemma \ref{lmm:Explicit}]
By symmetry, we may assume that $\mathcal{I}=\{1,\dots,\ell_1\}$ for some $\ell_1< \ell$. By Dirichlet's theorem on Diophantine approximation, any $a\in[0, X)$ has a representation
\[
\frac{a}{X}=\frac{b}{q}+\nu
\]
for some integers $(b,q)=1$ with $q\le X^{1/2}$ and some real $|\nu|\le 1/X^{1/2}q$. Thus we can divide $[0,X)$ into $O(\log{X})^2$ sets $\mathcal{F}(Q,E)$ as defined by Lemma \ref{lmm:Explicit} for different parameters $Q$, $E$ satisfying $1\le Q\le X^{1/2}$ and $E=0$ or $1/X\le E\le 100 X^{1/2}/Q$. Moreover, if $a\notin\mathcal{M}$ then $a\in\mathcal{F}=\mathcal{F}(Q,E)$ for some $Q$, $E$, with $Q+E\ge (\log{X})^C$. Thus, provided $C$ is sufficiently large compared with $A$ and $\eta$, we see it is sufficient to show that
\begin{equation}
\frac{1}{X}\Bigl|\sum_{a\in\mathcal{F}\cap\mathcal{E}}S_{\mathcal{A}}\Bigl(\frac{a}{X}\Bigr)S_{\mathcal{R}_X}\Bigl(\frac{-a}{X}\Bigr)\Bigr|\ll \frac{\#\mathcal{A}}{(Q+E)^{\epsilon/20}}.
\label{eq:ExcepBound1}
\end{equation}
From the definition \eqref{eq:LambdaDef} of $\Lambda_{\mathcal{R}_X}$ and shape of $\mathcal{R}_X$ given by Proposition \ref{prpstn:Exceptional}, we have that for $n<X$
\[
\Lambda_{\mathcal{R}_X}(n)=\sum_{\substack{n_1n_2p=n\\ X^{\eta/4},X^{1-\sum_ia_i-\ell\delta}\le p}}\Lambda_{\mathcal{R}_1}(n_1)\Lambda_{\mathcal{R}_2}(n_2)\log{p},
\]
where $\mathcal{R}_1$ is the projection of $\mathcal{R}_X$ onto the first $\ell_1$ coordinates, and $\mathcal{R}_2$ is the projection onto the subsequent $\ell-\ell_1-1$ coordinates.

Since $n_1$, $n_2$, $p$ and $X$ are integers, $|\log{((X-1/2)/n_1n_2p)}|\gg 1/X$. Thus, by Perron's formula (see, for example, \cite[Chapter 17]{Davenport}), we have for $n_1,n_2,p<X$
\begin{align*}
\frac{1}{(2\pi i)^2}\int_{1/\log{X}-i X^4}^{1/\log{X}+i X^4}\Bigl(\frac{X-1/2}{n_1n_2p}\Bigr)^{s}\frac{d s}{s}
=\begin{cases}
1+O(X^{-2}),\qquad &\text{if $n_1n_2p< X$,}\\
O(X^{-2}), &\text{otherwise.}
\end{cases}
\end{align*}
We will use this to remove the constraint $n=n_1n_2p<X$ in $S_{\mathcal{R}_X}(-a/X)$. We first put $n_1,n_2,p$ into one of $O(\log{X})^3$ intervals of the form $(Y/10,Y]$, and then apply the above estimate. The $O(X^{-2})$ error term trivially makes a negligible contribution to \eqref{eq:ExcepBound1}. Thus, %
 we see that for $C$ sufficiently large, it suffices to show uniformly over all $s$ with $\Re(s)=1/\log{X}$ and all choices of $N_1,N_2,P$ with $N_1N_2P\le 1000 X$ and $P\ge X^{1-\sum_{i=1}^{\ell-1} a_i-\ell\delta}$ that
\begin{align*}
\frac{1}{X}\sum_{a\in\mathcal{F}\cap\mathcal{E}}S_{\mathcal{A}}\Bigl(\frac{a}{X}\Bigr)\sum_{\substack{n_1\sim N_1\\ n_2\sim N_2\\ p\sim P}}\frac{\Lambda_{\mathcal{R}_1}(n_1)\Lambda_{\mathcal{R}_2}(n_2)c_p}{n_1^s n_2^s p^s} e\Bigl(\frac{-an m p}{X}\Bigr)
\ll \frac{\#\mathcal{A}}{(Q+E)^{\epsilon/15}},
\end{align*}
where $c_p=\log{p}$ if $p\ge X^{\eta/4},X^{1-\sum_i a_i-\ell\delta}$ and 0 otherwise. (The integral over $s$ and the choices of $N_1,N_2,P$ contribute a factor of $O(\log{X})^4$, which is acceptable for establishing \eqref{eq:ExcepBound1} if $C$ is sufficiently large.) 

Since $\Lambda_{\mathcal{R}_1}(n_1)$ is supported on $n_1\in[X^{\sum_{i=1}^{\ell_1}a_i},X^{\sum_{i=1}^{\ell_1}a_i+\ell\delta}]$ and $\Lambda_{\mathcal{R}_2}(n_2)$ is supported on $n_2\ge X^{\sum_{\ell_1+1}^{\ell-1}a_i}$, we only need to consider $N_1N_2P\ge X^{1-\ell\delta}$ and $N_1\in [X^{\sum_{i=1}^{\ell_1}a_i},X^{\sum_{i=1}^{\ell_1}a_i+\epsilon/6}]$. But, by assumption, 
\[
\sum_{i=1}^{\ell_1}a_i\in \Bigl[\frac{9}{25}+\frac{\epsilon}{2},\frac{17}{40}-\frac{\epsilon}{2}\Bigr]\cup\Bigl[\frac{23}{40}+\frac{\epsilon}{2},\frac{16}{25}-\frac{\epsilon}{2}\Bigr],
\]
 so either $N_1$ or $N_2 P$ lie in $[X^{9/25},X^{17/40}]$. Since $\Lambda_{\mathcal{R}_1}(n_1),\Lambda_{\mathcal{R}_2}(n_2),\log{p}\ll_\ell (\log{X})^{\ell-1}$, %
\begin{align}
\frac{1}{X}\sum_{a\in\mathcal{F}\cap\mathcal{E}}S_{\mathcal{A}}\Bigl(\frac{a}{X}\Bigr)\sum_{n\sim N}\alpha_n\sum_{m\sim M}\beta_m e\Bigl(\frac{-an m}{X}\Bigr)
\ll \frac{\#\mathcal{A}}{(Q+E)^{\epsilon/12}}
\label{eq:ExcepBound2}
\end{align}
uniformly over all choices of $N\in [X^{9/25},X^{17/40}]$ and $M\le 1000 X/N$ and uniformly over all 1-bounded complex sequences $\alpha_n,\beta_m$. (Setting $\alpha_n=\Lambda_{\mathcal{R}_1}(n)/(\log{X})^{\ell}$ and $\beta_m=\sum_{p n_2=m, p\sim P, n_2\sim N_2}\Lambda_{\mathcal{R}_2}(n_2)c_p/(\log{X})^{\ell}$ gives the bound when $\sum_{i=1}^{\ell_1}a_i\in [9/25+\epsilon/2,17/40-\epsilon/2]$; the other case is analogous with $\alpha_n$ and $\beta_m$ swapped.)

Finally, let $\gamma_a$ be the 1-bounded sequence satisfying $S_{\mathcal{A}}(a/X)=\#\mathcal{A}\gamma_a F_X(a/X)$. After substituting this expression for $S_\mathcal{A}$, we see that \eqref{eq:ExcepBound2} follows immediately from  Lemma \ref{lmm:Explicit} for $C$ sufficiently large in terms of $\eta$, thus giving the result.
\end{proof}
%
%
%
%
Thus it remains to establish Lemma \ref{lmm:Explicit}. The key estimate constraining Fourier frequencies to additively structured sets is the following lemma.
%
%
%
%
\begin{lmm}[Geometry of numbers] \label{lmm:GeometryOfNumbers}
Let $K_0$ be a sufficiently large constant, let $\mathbf{t}\in\mathbb{R}^3$ with $\|\mathbf{t}\|_2=1$ and let $N>1>\delta>0$. Let
\[
\mathcal{R}=\{\mathbf{v}\in\mathbb{R}^3:\,\|\mathbf{v}\|_2\le N,\,|\mathbf{v}\cdot\mathbf{t}|\le \delta\}
\]
satisfy $\#\mathcal{R}\cap\mathbb{Z}^3\ge \delta K N^2$ for some $K>K_0$. Then there exists a lattice $\Lambda\subset\mathbb{Z}^3$ of rank at most 2 such that
\[
\#\{\mathbf{v}\in\Lambda\cap\mathcal{R}\}\ge \frac{\delta K N^2}{2}.
\]
\end{lmm}
%
%
%
%
If a cuboid $\mathcal{R}\subseteq\mathbb{R}^3$ of volume $V$ lies in a the region $|z|\le \epsilon$, then it can easily contain rather more than $V$ lattice points from the plane $z=0$. Lemma \ref{lmm:GeometryOfNumbers} says that such a situation is essentially the only way a cuboid can contain many lattice points; if any cuboid has substantially more than $V$ lattice points in $\mathcal{R}\cap\mathbb{Z}^3$, then these lattice points must come from some lower dimensional linear subspace. The region $\mathcal{R}$ which we are interested in is a slightly thickened disc through the origin in the plane orthogonal to $\mathbf{t}$.
\begin{proof}[Proof of Lemma \ref{lmm:GeometryOfNumbers}]
Let $\phi:\mathbb{R}^3\rightarrow\mathbb{R}^3$ be the linear map which is a dilation by a factor $N/\delta$ in the $\mathbf{t}$-direction (i.e. $\phi(\mathbf{v})=\mathbf{v}+\mathbf{t}(N/\delta-1)(\mathbf{v}\cdot\mathbf{t})$.) Let $\Lambda_1=\phi(\mathbb{Z}^3)\subset\mathbb{R}^3$ be the lattice which is the image of $\mathbb{Z}^3$ under $\phi$. Since the determinant of a lattice is the volume of the fundamental parallelepiped, we see that $\det(\Lambda_1)=N/\delta$.

Let $\{\mathbf{v}_1,\mathbf{v}_2,\mathbf{v}_3\}$ be a Minkowski-reduced basis of $\Lambda_1$. We recall that this means that any $\mathbf{v}\in\Lambda_1$ can be written uniquely as $n_1\mathbf{v}_1+n_2\mathbf{v}_2+n_3\mathbf{v}_3$ for some $n_1,n_2,n_3\in\mathbb{Z}$, and for any $n_1,n_2,n_3\in\mathbb{Z}$ we have
\[
\|n_1\mathbf{v}_1+n_2\mathbf{v}_2+n_3\mathbf{v}_3\|_2\asymp \sum_{i=1}^3\|n_i\mathbf{v}_i\|_2,
\]
and that $\|\mathbf{v}_1\|_2\|\mathbf{v}_2\|_2\|\mathbf{v}_3\|_2\asymp \det(\Lambda_1)=N/\delta$. Without loss of generality let $\|\mathbf{v}_1\|_2\le \|\mathbf{v}_2\|_2\le\|\mathbf{v}_3\|_2$. 

We now notice that any element of $\mathcal{R}\cap\mathbb{Z}^3$ is mapped injectively by $\phi$ to an element of $\{\mathbf{x}\in\Lambda_1:\,\|\mathbf{x}\|_2\le 2N\}$. Thus for a sufficiently large constant $C$, we have
\begin{align*}
\Bigl\{\mathbf{n}\in\mathbb{Z}^3:\,\sum_{i=1}^3n_i\mathbf{v}_i\in\phi(\mathcal{R})\Bigr\}&\subseteq\Bigl\{\mathbf{n}\in\mathbb{Z}^3:\,\Bigl\|\sum_{i=1}^3n_i\mathbf{v}_i\Bigr\|_2\le 2N\Bigr\}\\
&\subseteq\Bigl\{\mathbf{n}\in\mathbb{Z}^3: \, |n_i|\le C\frac{N}{\|\mathbf{v}_i\|_2}\Bigr\}.
\end{align*}
If $\|\mathbf{v}_3\|_2>C N$, then there are no $\mathbf{n}\in\mathbb{Z}^3$ counted above with $n_3\ne 0$. If instead $\|\mathbf{v}_3\|_2\le C N$ then since $\|\mathbf{v}_1\|_2\le\|\mathbf{v}_2\|_2\le\|\mathbf{v}_3\|_2$, the number of $\mathbf{n}$ is
\[
\ll \frac{C^3N^3}{\prod_{i=1}^3\|\mathbf{v}_i\|_2}\ll \frac{N^3}{\det(\Lambda_1)}\ll\delta N^2. 
\]
Thus in either case there are $O(\delta N^2)$ points with $n_3\ne 0$. However, by assumption of the lemma we have that $K$ is sufficiently large and 
\[
\delta K N^2\le \#\{\mathbf{x}\in\mathbb{Z}^3\cap\mathcal{R}\}=\#\{\mathbf{x}\in\Lambda_1:\,\mathbf{x}\in\phi(\mathcal{R})\}.
\]
This means that most of the contribution must come from terms with $n_3=0$. Indeed, we have
\begin{align*}
\#\{(n_1,n_2)\in\mathbb{Z}^2:\,n_1\mathbf{v}_1+n_2\mathbf{v}_2\in \phi(\mathcal{R})\}&= \#\{\mathbf{x}\in\Lambda_1:\,\phi(x)\in\mathcal{R}\}-O(\delta N^2)\\
&\ge \delta K N^2-O(\delta N^2).
\end{align*}
We may choose $K_0$ such that if $K\ge K_0$ then the right hand side is at least $\delta KN^2/2$. Thus, we see if $\Lambda$ is the lattice $\phi^{-1}(\mathbf{v}_1)\mathbb{Z}+\phi^{-1}(\mathbf{v}_2)\mathbb{Z}$ then $\Lambda\subseteq\mathbb{Z}^3$ and 
\[
\#\{\mathbf{v}\in\Lambda\cap\mathcal{R}\}\ge \delta KN^2/2.\qedhere
\]
\end{proof}
%
%
%
%
We establish Lemma \ref{lmm:Explicit} assuming two key propositions, Proposition \ref{prpstn:LatticeBound} and Proposition \ref{prpstn:LineBound}, given below. These propositions will be proven over the next two sections.
%
%
%
%
\begin{prpstn}[Bound for angles generating lattices] \label{prpstn:LatticeBound}
Let $X,K,N,Q\ge 1$ and $\delta>0$, $E\ge 0$ satisfy $X^{17/40}\le N K$, $\delta\ge N/X$, $E\le 100X^{1/2}/Q$ and $Q\le X^{1/2}$. Let $\mathcal{B}_1=\mathcal{B}_1(N,K,\delta)\subseteq[0,X)^2$ be the set of pairs $(a_1,a_2)\in\mathbb{Z}^2$ such that there is a lattice $\Lambda\subseteq\mathbb{Z}^3$ of rank 2 such that
\[\#\{\mathbf{n}\in\Lambda:\,| n_1a_1+n_2a_2+n_3 X|\le \delta X,\,\|\mathbf{n}\|_2\le N\}\ge \delta K N^2,\]
and not all of these points lie on a line through the origin. Let $\mathcal{F}=\mathcal{F}(Q,E)$ be given by
\[
\mathcal{F}=\Bigl\{a<X:\, \frac{a}{X}=\frac{b}{q}+\nu\text{ for some }(b,q)=1\text{ with }q\sim Q,\,|\nu|\sim E/X\Bigr\}.
\]
Then we have
\[
\sum_{\substack{(a_1,a_2)\in\mathcal{B}_1(N,K,\delta)\\ a_1,a_2\in\mathcal{F}\cap\mathcal{E}}}F_X\Bigl(\frac{a_1}{X}\Bigr)F_X\Bigl(\frac{a_2}{X}\Bigr)
\ll \frac{(\log{X})^5}{(Q+E)^{\epsilon/4}}\frac{X}{N K}.
\]
\end{prpstn}
%
%
%
%
\begin{prpstn}[Bound for angles generating lines] \label{prpstn:LineBound}
Let $N\ge X^{9/25}$, $\delta\ge N/X$ and $K\ge 1$. Let $\mathcal{B}_2=\mathcal{B}_2(N,K,\delta)\subseteq[0,X)^2$ be the set of pairs $(a_1,a_2)\in\mathbb{Z}^2$ such that there exists a line $L$ through the origin such that
\[
\#\{\mathbf{n}\in L\cap\mathbb{Z}^3: |n_1a_1+n_2a_2+n_3X|\le \delta X,\,\|\mathbf{n}\|_2\le N\}\ge \delta N^2 K.
\]
Given $B\le X^{23/80}$, let $\mathcal{E}'=\mathcal{E}'(B)$ be given by
\[
\mathcal{E}'=\Bigl\{a<X:\,F_X\Bigl(\frac{a}{X}\Bigr)\sim \frac{1}{B}\Bigr\}.
\]
Then we have
\[
\sum_{\substack{(a_1,a_2)\in\mathcal{B}_2(N,K,\delta) \\ a_1,a_2\in\mathcal{E}'}}F_X\Bigl(\frac{a_1}{X}\Bigr)F_X\Bigl(\frac{a_2}{X}\Bigr)\ll \frac{X^{1-\epsilon}}{NK}.
\]
\end{prpstn}
%
%
%
%

\begin{proof}[Proof of Lemma \ref{lmm:Explicit} assuming Propositions \ref{prpstn:LatticeBound} and \ref{prpstn:LineBound}]
We split $\mathcal{E}$ into $O(\log{X})$ subsets of the form
\[
\mathcal{E}'=\mathcal{E}'(B)=\Bigl\{a\in[0,X):\,F_X\Bigl(\frac{a}{X}\Bigr)\sim \frac{1}{B}\Bigr\}
\]
for some $B\in [1,X^{23/80}]$. By Cauchy-Schwarz, we have
\begin{align*}
\sum_{a\in\mathcal{F}\cap\mathcal{E}'}\sum_{\substack{n\sim N\\ m\sim M}}F_{X}\Bigl(\frac{a}{X}\Bigr)\alpha_n\beta_m\gamma_a e\Bigl(\frac{-a nm}{X}\Bigr)&\ll \Sigma_1^{1/2}\Sigma_2^{1/2},
\end{align*}
where
\begin{align*}
\Sigma_1&=\sum_{m\ll  X/N}|\beta_m|^2\ll \frac{X}{N},\\
\Sigma_2&=\sum_{m\ll X/N}\Bigl|\sum_{a\in\mathcal{F}\cap\mathcal{E}'}\sum_{n\sim N}\alpha_n \gamma_a F_{X}\Bigl(\frac{a}{X}\Bigr)e\Bigl(\frac{-a n m}{X}\Bigr)\Bigr|^2\\
&= \sum_{a_1,a_2\in\mathcal{F}\cap\mathcal{E}'}F_{X}\Bigl(\frac{a_1}{X}\Bigr)F_{X}\Bigl(\frac{a_2}{X}\Bigr)\sum_{n_1,n_2\sim N}\alpha_{n_1}\overline{\alpha_{n_2}}\gamma_{a_1}\overline{\gamma_{a_2}}\sum_{m\ll  X/N}e\Bigl(\frac{m(a_1n_1-a_2n_2)}{X}\Bigr)\\
&\ll \sum_{a_1,a_2\in\mathcal{F}\cap\mathcal{E}'}F_X\Bigl(\frac{a_1}{X}\Bigr)F_X\Bigl(\frac{a_2}{X}\Bigr)\sum_{n_1,n_2\sim N}\min\Bigl(\frac{X}{N},\Bigl\|\frac{a_1n_1-a_2n_2}{X}\Bigr\|^{-1}\Bigr).
\end{align*}
Thus it suffices to show
\[
\sum_{a_1,a_2\in\mathcal{F}\cap\mathcal{E}'}F_X\Bigl(\frac{a_1}{X}\Bigr)F_X\Bigl(\frac{a_2}{X}\Bigr)\sum_{n_1,n_2\le N}\min\Bigl(\frac{X}{N},\Bigl\|\frac{a_1n_1-a_2n_2}{X}\Bigr\|^{-1}\Bigr)\ll  \frac{N X(\log{X})^{O(1)}}{(Q+E)^{\epsilon/5}},
\]
provided $X^{9/25}\le N\le X^{17/40}$, $Q\le X^{1/2}$ and $E\le 100X^{1/2}/Q$.

Let $\mathcal{G}(K)$ denote the set of pairs $(a_1,a_2)\in\mathcal{F}\cap\mathcal{E}'$ such that
\[
\sum_{n_1,n_2\le N}\min\Bigl(\frac{X}{N},\Bigl\|\frac{n_1a_1-n_2a_2}{X}\Bigr\|^{-1}\Bigr)\sim N^2 K.
\]
We consider $1\le K\le X/N$ taking values which are integral powers of 10, and split the contribution of our sum according to these sets. We see it is therefore sufficient to show that for each $K$
\[
\sum_{\substack{(a_1,a_2)\in\mathcal{G}(K)\\ a_1,a_2\in\mathcal{F}\cap\mathcal{E}'}}F_X\Bigl(\frac{a_1}{X}\Bigr)F_X\Bigl(\frac{a_2}{X}\Bigr)\ll \frac{X(\log{X})^{O(1)}}{(Q+E)^{\epsilon/5} N K}.
\]
Let $\mathcal{G}(K,\delta)$ denote the set of pairs $(a_1,a_2)\in\mathcal{F}\cap\mathcal{E}'$ such that
\[
\#\Bigl\{\mathbf{n}\in\mathbb{Z}^3:\,\Bigl|\frac{n_1a_1-n_2a_2-n_3 X}{X}\Bigr|\le \delta,\,\|\mathbf{n}\|_2\le 10 N\Bigr\}\ge \delta N^2 K.
\]
By considering $\delta=2^{-j}$ and using the pigeonhole principle, we see that if
\[
\sum_{n_1,n_2\le N}\min\Bigl(\frac{X}{N},\Bigl\|\frac{n_1a_1-n_2a_2}{X}\Bigr\|^{-1}\Bigr)\sim N^2 K,
\]
then there is some $\delta\ge N/X$ and some $K/\log{X} \ll K'\le K$ such that
\[
(a_1,a_2)\in\mathcal{G}(K',\delta).
\]
Thus is suffices to show for all $K',\delta$ that
\begin{equation}
\sum_{\substack{(a_1,a_2)\in\mathcal{G}(K',\delta)\\ a_1,a_2\in\mathcal{F}\cap\mathcal{E}'}}F_X\Bigl(\frac{a_1}{X}\Bigr)F_X\Bigl(\frac{a_2}{X}\Bigr)
\ll \frac{X(\log{X})^{O(1)}}{(Q+E)^{\epsilon/5} N K'}. 
\label{eq:BilinearTarget}
\end{equation}
From Lemma \ref{lmm:Generic}, we have the bound
\[
\sum_{\substack{(a_1,a_2)\in\mathcal{G}(K',\delta)\\ a_1,a_2\in\mathcal{F}\cap\mathcal{E}'}}F_X\Bigl(\frac{a_1}{X}\Bigr)F_X\Bigl(\frac{a_2}{X}\Bigr)\ll \Bigl(\sum_{a_1\in\mathcal{E}'}F_X\Bigl(\frac{a_1}{X}\Bigr)\Bigr)^2\ll X^{23/40-2\epsilon},
\]
which gives \eqref{eq:BilinearTarget} in the case when $N K'\ll X^{17/40+\epsilon}$. Thus we may assume that $N K'\gg X^{17/40+\epsilon}$. By assumption, we also have that $N\le X^{17/40}$, so we only consider $K'\gg X^{\epsilon}$. In particular, we may use Lemma \ref{lmm:GeometryOfNumbers} to conclude that either there is a rank 2 lattice $\Lambda\subseteq\mathbb{Z}^3$ such that 
\[
 \#\{\mathbf{n}\in\Lambda:\,\|\mathbf{n}\|_2\le 10 N,\, |n_1a_1+n_2a_2+n_3 X|\le \delta X\}\ge \delta K' N^2/2,
\]
and not all of these points lie on a line through the origin, or there is a line $L\subseteq\mathbb{Z}^3$ such that
\[
 \#\{\mathbf{n}\in L:\,\|\mathbf{n}\|_2\le 10 N,\, |n_1a_1+n_2a_2+n_3 X|\le \delta X\}\ge \delta K' N^2/2.
\]
In either case \eqref{eq:BilinearTarget} follows from Proposition \ref{prpstn:LatticeBound} or Proposition \ref{prpstn:LineBound} (taking `$N$' and `$K$' in the propositions to be $10N$ and $K'/1000\ge 1$ in our notation here).
\end{proof}
%
%
%
%
Thus it remains to establish Proposition \ref{prpstn:LatticeBound} and Proposition \ref{prpstn:LineBound}.
%
%
%
%
%
%
%
\section{Lattice Estimates}\label{sec:Lattice}
%
%
%
%
In this section we establish Proposition \ref{prpstn:LatticeBound}, which controls the contribution from pairs of angles which cause a large contribution to the bilinear sums considered in Section \ref{sec:Bilinear} to come from a lattice. A low height lattice $\Lambda$ makes a significant contribution only if $(a_1,a_2,X)$ is approximately orthogonal to the plane of the lattice, and so only if $(a_1,a_2,X)$ lies close to the line through the origin orthogonal to this lattice. We note that we only make small use of the fact that these angles lie in a small set, but it is vital that the angles lie outside the major arcs.
%
%
%
%
\begin{lmm}[Lattice generating angles have simultaneous approximation]\label{lmm:LatticeApprox}
Let $\delta>0$ and $X,N,K\ge1$ be such that $\delta\ge N/X$. Let $\mathcal{B}_1=\mathcal{B}_1(N,K,\delta)\subseteq[0,X)^2$ be the set of pairs $(a_1,a_2)\in\mathbb{Z}^2$ such that there is a lattice $\Lambda\subseteq\mathbb{Z}^3$ of rank 2 such that
\[\#\{\mathbf{n}\in\Lambda:\,| n_1a_1+n_2a_2+n_3 X|\le \delta X,\,\|\mathbf{n}\|_2\le N\}\ge  \delta K N^2,\]
and moreover the points counted above do not all lie on a line through the origin.

Then all pairs $(a_1,a_2)\in\mathcal{B}_1$ have the simultaneous rational approximations
\begin{align*}
\frac{a_1}{X}&=\frac{b_1}{q}+O\Bigl(\frac{1}{N K q}\Bigr),\\
\frac{a_2}{X}&=\frac{b_2}{q}+O\Bigl(\frac{1}{N K q}\Bigr),
\end{align*}
for some integer $q\ll X/N K$.
\end{lmm}
%
%
%
%
We see Lemma \ref{lmm:LatticeApprox} restricts the pair $(a_1,a_2)$ to lie in a set of size $O(X/N K)^3$, which is noticeably smaller than $X^2$ for the range of $N K$ under consideration. This allows us to obtain superior bounds for the sum over $a_1,a_2$, by exploiting the estimates of Lemma \ref{lmm:Hybrid} which show $F$ is not abnormally large on such a set. 
%
%
%
%
\begin{proof} Clearly we may assume that $N K$ is sufficiently large, since otherwise the result is trivial. By assumption of the lemma, for any pair $(a_1,a_2)\in\mathcal{B}_1$ there is a rank 2 lattice $\Lambda=\Lambda_{a_1,a_2}$ such that $\#(\Lambda\cap\mathcal{H})\ge \delta K N^2$ where  
\[
\mathcal{H}=\{\mathbf{x}\in\mathbb{R}^3:\,| x_1 a_1+x_2 a_2+x_3 X|\le \delta X,\,\|\mathbf{x}\|_2\le N\}.
\]
Moreover, not all the points in $\Lambda\cap\mathcal{H}$ lie in a line through the origin. Let $\mathbf{a}=(a_1,a_2,X)$, and let $\phi:\mathbb{R}^3\rightarrow\mathbb{R}^3$ be a dilation by a factor $N/\delta$ in the $\mathbf{a}$-direction, and let $\Lambda'=\phi(\Lambda)$. Then we see that
\[
\phi(\Lambda\cap\mathcal{H})\subseteq \{\mathbf{x}\in\Lambda':\,\|\mathbf{x}\|_2\le 2N\}.
\] 
Moreover, not all the points on the right hand hand side lie in a line through the origin, since $\phi^{-1}$ preserves lines through the origin. Let $\Lambda'$ have a Minkowski-reduced basis $\{\mathbf{v}_1,\mathbf{v}_2\}$, and let $V_1=\|\mathbf{v}_1\|_2$ and $V_2=\|\mathbf{v}_2\|_2$. Since $\|m_1\mathbf{v}_1+m_2\mathbf{v}_2\|_2\asymp |m_1|V_1+|m_2|V_2$, for a suitably large constant $C$ we have
\begin{align*}
\{\mathbf{x}\in\Lambda':\,\|\mathbf{x}\|_2\le 2 N\}&\subseteq \Bigl\{m_1\mathbf{v}_1+m_2\mathbf{v}_2:\,|m_1|\le \frac{C N}{ V_1},\,|m_2|\le \frac{C N}{V_2}\Bigr\}.
\end{align*}
Since not all of the points in the final set lie in a line through the origin, we see that $V_1,V_2\le C N$. Thus
\[
\delta K N^2\le \#(\Lambda\cap\mathcal{H})=\#(\Lambda'\cap\phi(\mathcal{H}))\ll \frac{N^2}{V_1 V_2}.
\]
In particular, $V_1V_2\ll 1/\delta K$.

Let $\mathbf{w}_1=\phi^{-1}(\mathbf{v}_1)$ and $\mathbf{w}_2=\phi^{-1}(\mathbf{v}_2)$, so $\mathbf{w}_1$ and $\mathbf{w}_2$ are linearly independent vectors in $\Lambda\subseteq\mathbb{Z}^3$. Since $\phi$ can only increase the length of vectors, $\|\mathbf{w}_1\|_2\le V_1$ and $\|\mathbf{w}_2\|_2\le V_2$. Let  $\epsilon_1=|\mathbf{w}_1\cdot \mathbf{a}|$ and $\epsilon_2=|\mathbf{w}_2\cdot \mathbf{a}|$. Trivially we have $|\mathbf{v}_1\cdot \mathbf{a}|\ll V_1X$ and $|\mathbf{v}_2\cdot\mathbf{a}|\ll V_2X$, and so recalling that $\phi$ is a dilation by a factor $N/\delta$ in the $\mathbf{a}$-direction, we see that $\epsilon_1\ll \delta X V_1/N$ and $\epsilon_2\ll \delta X V_2/N$.

Putting this together, we see that for any pair $(a_1,a_2)\in\mathcal{B}_1$ there are linearly independent vectors $\mathbf{w}_1,\mathbf{w}_2\in\mathbb{Z}^3$ and quantities $V_1,V_2$ such that
\begin{align*}
V_1V_2\ll \frac{1}{\delta K},\qquad \|\mathbf{w}_1\|_2\le V_1,\qquad \|\mathbf{w}_2\|_2\le V_2,\\
|\mathbf{a}\cdot \mathbf{w}_1|\ll \frac{\delta X V_1}{N},\qquad |\mathbf{a}\cdot \mathbf{w}_2|\ll \frac{\delta X V_2}{N}.
\end{align*}
This puts considerable constraints on the possibilities for $(a_1,a_2)$, since it must lie in an infinite cylinder with axis parallel to $\mathbf{w}_1\times\mathbf{w}_2$ with short radius, for some low height vectors $\mathbf{w}_1,\mathbf{w}_2$. (Here $\times$ is the standard cross product on $\mathbb{R}^3$.) Explicitly, let $\mathbf{e}_1,\mathbf{e}_2,\mathbf{e}_3$ be an orthonormal basis of $\mathbb{R}^3$ with $\mathbf{e}_1$ orthogonal to $\mathbf{w}_1$ and $\mathbf{w}_2$, and with $\mathbf{e}_2$ orthogonal to $w_2$. Then we see that $\mathbf{e}_1\propto \mathbf{w}_1\times\mathbf{w}_2$, $\mathbf{e}_2\propto \mathbf{w}_2\times \mathbf{e}_1$ and $\mathbf{e}_3\propto \mathbf{w}_2$. In particular, we have that $|\mathbf{e}_3\cdot \mathbf{w}_2|=\|\mathbf{w}_2\|_2$, and
\[
|\mathbf{e}_2\cdot \mathbf{w}_1|=\frac{|\mathbf{w}_1\cdot (\mathbf{w}_2\times(\mathbf{w}_1\times\mathbf{w}_2))|}{\|\mathbf{w}_2\|_2\|\mathbf{w}_1\times\mathbf{w}_2\|_2}=\frac{\|\mathbf{w}_1\times\mathbf{w}_2\|_2}{\|\mathbf{w}_2\|_2}.
\]
(Here we used the identity $\mathbf{a}\cdot (\mathbf{b}\times \mathbf{c})=\mathbf{c}\cdot(\mathbf{a}\times\mathbf{b})$.) Thus, if $\mathbf{x}=x_1\mathbf{e}_1+x_2\mathbf{e}_2+x_3\mathbf{e}_3$ has $|\mathbf{x}\cdot \mathbf{w}_1|\ll \delta X V_1/N$ and $|\mathbf{x}\cdot \mathbf{w}_2|\ll \delta X V_2/N$, then
\begin{align*}
\frac{\delta X V_2}{N}&\gg |\mathbf{x}\cdot \mathbf{w}_2|=|x_3|\,\|\mathbf{w}_2\|_2,\\
\frac{\delta X V_1}{N}&\gg |\mathbf{x}\cdot \mathbf{w}_1|=\frac{|x_2|\,\|\mathbf{w}_1\times \mathbf{w}_2\|_2}{\|\mathbf{w}_2\|_2}+O\Bigl(|x_3|\, \|\mathbf{w}_1\|_2\Bigr).
\end{align*}
Since $\|\mathbf{w}_1\|_2\ll V_1$, $\|\mathbf{w}_2\|_2\ll V_2$ and $\|\mathbf{w}_1\times\mathbf{w}_2\|_2\le \|\mathbf{w}_1\|_2\|\mathbf{w}_2\|_2$, this implies that
\begin{align*}
|x_3|&\ll \frac{\delta X V_2}{N\|\mathbf{w}_2\|_2}\ll \frac{\delta X V_1 V_2}{N \|\mathbf{w}_1\times \mathbf{w}_2\|_2},\\
|x_2|&\ll \frac{\delta X V_1 V_2}{N \|\mathbf{w}_1\times \mathbf{w}_2\|_2}+\frac{|x_3|\, \|\mathbf{w}_1\|_2 \, \|\mathbf{w}_2\|_2}{\|\mathbf{w}_1\times \mathbf{w}_2\|_2}\ll \frac{\delta X V_1 V_2}{N \|\mathbf{w}_1\times \mathbf{w}_2\|_2}.
\end{align*}
Thus, since $V_1V_2\ll 1/\delta K$, we see that any vector $\mathbf{x}$ with $|\mathbf{x}\cdot \mathbf{w}_1|\ll \delta X V_1/N$ and $|\mathbf{x}\cdot \mathbf{w}_2|\ll \delta X V_2/N$ satisfies 
\[
\mathbf{x}=\lambda (\mathbf{w}_1\times\mathbf{w}_2)+O\Bigl(\frac{X}{N K \|\mathbf{w}_1\times\mathbf{w}_2\|_2}\Bigr)
\]
for some $\lambda\in \mathbb{R}$. We note that the error term is $o(X)$ since $\mathbf{w}_1,\mathbf{w}_2$ are linearly independent integer vectors and $N K$ is assumed sufficiently large.  Let the components of $\mathbf{w}_1\times \mathbf{w}_2$ be $c_1,c_2,c_3$ (with respect to the standard basis of $\mathbb{R}^3$). Since $\mathbf{w}_1,\mathbf{w}_2\in\mathbb{Z}^3$, we have $c_1,c_2,c_3\in\mathbb{Z}$. Thus if $\mathbf{a}$ is of the above form we must have $\mathbf{a}=\lambda(\mathbf{w}_1\times \mathbf{w}_2)+o(X)$ for some $\lambda$. Since $\|\mathbf{a}\|_2\ge X$ and $a_1,a_2\le a_3=X$, we must have that $|c_1|,|c_2|\ll |c_3|$. In particular, $|c_3|\asymp \|\mathbf{w}_1\times\mathbf{w}_2\|_2$. Dividing through by $X=\lambda c_3+O(X/N K |c_3|)$ then gives
\begin{equation}
\Bigl\|
\begin{pmatrix}
a_1/X\\
a_2/X
\end{pmatrix}
-
\begin{pmatrix}
c_1/c_3\\
c_2/c_3
\end{pmatrix}
\Bigr\|_2\ll \frac{1}{N K |c_3|}.\label{eq:LatticeApprox}
\end{equation}
Finally, we note that since $\delta \ge N/X$ and $V_1 V_2\ll 1/\delta K$ we have
\[
c_1,c_2,c_3\le \|\mathbf{w}_1\times\mathbf{w}_2\|_2\le \|\mathbf{w}_1\|_2\|\mathbf{w}_2\|_2\le V_1 V_2 \ll \frac{1}{\delta K} \ll \frac{X}{NK}.
\]
Thus, we see that for any pair $(a_1,a_2)\in\mathcal{B}_1$ there must be integers $c_1,c_2,c_3\ll X/N K$ such that \eqref{eq:LatticeApprox} holds. This gives the result.
\end{proof}
%
%
%
%
\begin{lmm}[Size of rational approximations]\label{lmm:RatSize}
Let  $\mathcal{B}_1(N,K,\delta)$ and $\mathcal{F}=\mathcal{F}(Q,E)$ be as in Proposition \ref{prpstn:LatticeBound}. If $\mathcal{B}_1(N,K,\delta)\cap\mathcal{F}^2\ne\emptyset$ then
\[
Q+E\ll \Bigl(\frac{X}{N K}\Bigr)^2.
\]
\end{lmm}
%
%
%
%
\begin{proof}
By Lemma \ref{lmm:LatticeApprox}, if $(a_1,a_2)\in\mathcal{B}_1(N,K,\delta)$ then
\begin{align*}
\frac{a_1}{X}&=\frac{b_1}{q}+\nu_1,\\
\frac{a_2}{X}&=\frac{b_2}{q}+\nu_2,
\end{align*}
for some $q\ll X/N K$ and $|\nu_1|,|\nu_2|\ll 1/N K q$. By clearing common factors we may assume that $(b_1,b_2,q)=1$. 

If $N K > X^{2/3}$ (and $X$ is sufficiently large) then we see that $b_1/q$ and $b_2/q$ are the best rational approximations to $a_1/X$ and $a_2/X$ with denominator $O(X^{1/3})$, since the error in the approximation is $O(1/(qX^{2/3}))$. Thus if we also have $a_1,a_2\in\mathcal{F}(Q,E)$ then we must have $q\gg Q$ and $|\nu_1|,|\nu_2|\sim E/X$. In particular, we must have $Q+E\ll X/NK$. If instead $N K\le X^{2/3}$ then since $Q+E\ll X^{1/2}$ we have $Q+E\ll (X/NK)^{2}$. Thus in either case we have that there are no such pairs $(a_1,a_2)$ in both $\mathcal{B}_1(N,K,\delta)$ and in $\mathcal{F}\times\mathcal{F}$ unless $Q+E\ll (X/NK)^2$.
\end{proof}
%
%
%
%
\begin{lmm}\label{lmm:LatticeSums}
Let $N K\ge X^{17/40}$, and let $\mathcal{B}_1(N,K,\delta)$, $\mathcal{F}=\mathcal{F}(Q,E)$ and $\mathcal{E}$ be as in Proposition \ref{prpstn:LatticeBound}. Then we have
\[
\sum_{\substack{(a_1,a_2)\in\mathcal{B}_1(N,K,\delta)\\ a_1,a_2\in\mathcal{E}}}F_X\Bigl(\frac{a_1}{X}\Bigr)F_X\Bigl(\frac{a_2}{X}\Bigr)
\ll (\log{X})^5\sup_{\substack{Q_1,G_1,G_2\\ D_0,D_1,E_0}}\sum_{\substack{d_0,d_1\in\mathcal{V}\\ d_0\sim D_0\\ d_1\sim D_1}} \min(S_1S_2,S_1S_3),
\]
where $\mathcal{V}=\{2^u 5^v:u,v\in\mathbb{Z}_{\ge0}\}$, the supremum is over all choices of $Q_1,G_1,G_2,D_0,D_1,E_0\ge 1$ which are powers of 10 and satisfy $Q_1 G_1 G_2 D_0 D_1 E_0\ll X/N K$ and $G_1\ll G_2$, and $S_1,S_2,S_3$ are given by
\begin{align*}
S_1&=\sup_{\substack{q'\sim Q_1\\ (q',10)=1}}\sum_{\substack{g_1'\sim G_1\\ (g_1',10)=1}} \sum_{\substack{b_2'<d_0d_1 q' g_1' \\ (b_2',d_0 d_1 q' g_1')=1}}\sum_{\substack{|\nu_2|\le E_0/X\\ X(b_2'/d_0 d_1 q' g_1'+\nu_2)\in\mathbb{Z}}}F_{X}\Bigl(\frac{b_2'}{d_0 d_1 q' g_1'}+\nu_2\Bigr),\\
S_2&=\sum_{\substack{q'\sim Q_1\\ (q',10)=1}} \sum_{g_2\sim G_2}\sum_{\substack{b_1'<d_0 q' g_2 \\ (b_1',d_0 q' g_2)=1}} \sum_{\substack{|\nu_1|\le E_0/X \\ X(b_1'/d_0 q' g_2+\nu_1)\in\mathbb{Z}}}F_{X}\Bigl(\frac{b_1'}{d_0 q' g_2}+\nu_1\Bigr),\\
S_3&=\sum_{a_1\in\mathcal{E}}F_X\Bigl(\frac{a_1}{X}\Bigr)N(a_1,d_0),\\
N(a,d)&=\#\Bigl\{q\sim Q_1:\exists b,g\text{ s.t. }\Bigl|\frac{a}{X}-\frac{b}{q d g}\Bigr|\le \frac{E_0}{X},\,(b,d q g)=1,\,g \sim G_2\Bigr\}.
\end{align*}
\end{lmm}
%
%
%
%
\begin{proof}
By Lemma \ref{lmm:LatticeApprox} we are considering pairs $(a_1,a_2)\in\mathcal{B}_1(N,K,\delta)$ such that
\begin{align*}
\frac{a_1}{X}&=\frac{b_1}{q}+\nu_1,\\
\frac{a_2}{X}&=\frac{b_2}{q}+\nu_2,
\end{align*}
for some $q\ll X/N K$ and $|\nu_1|,|\nu_2|\ll 1/N K q$.

By clearing common factors we may assume that $(b_1,b_2,q)=1$. We let $g_1=(b_1,q)$ and $g_2=(b_2,q)$. By symmetry we may assume that $g_1\le g_2$. We let $d_1$ be the part of $g_1$ not coprime to $10$ (i.e. $d_1|10^u$ for some integer $u$, and $g_1=g_1'd_1$ for some $(g_1',10)=1$). Similarly we let $d_0$ be the part of $q/g_1g_2$ which is not coprime to $10$. To ease notation we let $b_1'=b_1/g_1$, $b_2'=b_2/g_2$, $q'=q/g_1g_2d_0$ and $g_1'=g_1/d_1$. Thus $q=g_1'g_2d_0d_1q'$, $b_1=b_1'd_1g_1'$ and $b_2=b_2'g_2$ with $(b_1',d_0 q' g_2)=(b_2',d_0 d_1 q' g_1')=1$ and $(q',10)=(g_1',10)=1$.

We split the contribution of pairs $(a_1,a_2)\in\mathcal{B}_1$ into $O(\log{X})^5$ subsets. We consider terms where we have the restrictions $q'\sim Q_1$, $g_1'\sim G_1$, $g_2\sim G_2$, $d_0\sim D_0$ and $d_1\sim D_1$  for some $Q_1,G_1,G_2,D_0,D_1 \ge 1$ all integer powers of 10 with $Q_0:=Q_1 G_1 G_2 D_0 D_1\ll X / N K$. Since $g_1=g_1'd_1\le g_2$ we have $G_1D_1\ll G_2$. We relax the restriction $|\nu_1|,|\nu_2|\ll 1/N K q$ to $|\nu_1|,|\nu_2|\le  E_0/X$ for a suitable power of 10 $E_0\asymp X/N K Q_0$ with $E_0\ge 1$. We see there are $O(\log{X})^5$ sets with such restrictions which cover all possible $(b_1,b_2,q,\nu_1,\nu_2)$ and hence all  $(a_1,a_2)\in\mathcal{B}_1$. For simplicity, the reader might like to consider the special case $G_1=G_2=D_0=D_1=1$ on a first reading.

 To ease notation we let $\mathcal{V}=\{2^u5^v:\,u,v\in\mathbb{Z}_{\ge 0}\}$, and note that we have $d_0,d_1\in\mathcal{V}$. By summing over all possibilities of $q',g_1',g_2,d_0,d_1,b_1',b_2'$, we see that
\begin{align*}
&\sum_{\substack{(a_1,a_2)\in\mathcal{B}_1(N,K,\delta) \\ a_1,a_2\in\mathcal{E} }}F_X\Bigl(\frac{a_1}{X}\Bigr)F_X\Bigl(\frac{a_2}{X}\Bigr)\ll (\log{X})^5\sup_{\substack{Q_1,G_1,G_2\\ D_0,D_1,E_0}} \sum_{\substack{d_0,d_1\in\mathcal{V}\\ d_0\sim D_0\\ d_1\sim D_1}}S_0,
\end{align*}
 where the supremum is over all choices of $Q_1,G_1,G_2,D_0,D_1,E_0\ge 1$ which are powers of 10 and satisfy $Q_1G_1G_2D_0D_1E_0\ll X/N K$ and $G_1D_1\ll G_2$ and $S_0$ is given by
\begin{align*}
S_0&=\sum'_{\substack{q'\sim Q_1 \\ g_1'\sim G_1 \\ g_2\sim G_2}}\sum'_{\substack{b_1'<d_0q' g_2 \\ b_2'<d_0d_1 q' g_1'}}\sum'_{\substack{|\nu_1|\le E_0/X\\ |\nu_2|\le E_0/X}}F_{X}\Bigl(\frac{b_1'}{d_0 q' g_2}+\nu_1\Bigr)F_{X}\Bigl(\frac{b_2'}{d_0 d_1 q' g_1'}+\nu_2\Bigr).
\end{align*}
In $S_0$, we have used $\sum'$ to indicate that the summation is further constrained by the conditions
\begin{align*}
(q',10)=(g_1',10)=(b_1',d_0q' g_2)=(b_2',d_0d_1q' g_1')=1,\\
X(b_1'/d_0 q' g_2+\nu_1)\in\mathbb{Z},\qquad X(b_2'/d_0 d_1 q' g_1'+\nu_2)\in\mathbb{Z},
\end{align*}
which we suppressed for notational simplicity. We see that $g_1',g_2,b_1',b_2',\nu_1,\nu_2$ each occur in only one of the two $F_X$ terms, and so given $d_0,d_1,q'$ the remaining summation in $S_0$ factors into a product of two sums. Taking a supremum over all choices of $q'$ in the first of these then gives
\begin{equation}
\sum_{\substack{(a_1,a_2)\in\mathcal{B}_1(N,K,\delta) \\ a_1,a_2\in\mathcal{F} }}F_X\Bigl(\frac{a_1}{X}\Bigr)F_X\Bigl(\frac{a_2}{X}\Bigr)\ll (\log{X})^5\sup_{\substack{Q_1,G_1,G_2\\ D_0,D_1,E_0}}\sum_{\substack{d_0,d_1\in\mathcal{V}\\ d_0\sim D_0\\ d_1\sim D_1}} S_1S_2,
\label{eq:Q0SmallBound}
\end{equation}
where
\begin{align}
S_1&=\sup_{\substack{q'\sim Q_1\\ (q',10)=1}}\sum_{\substack{g_1'\sim G_1\\ (g_1',10)=1}} \sum_{\substack{b_2'<d_0d_1 q' g_1' \\ (b_2',d_0 d_1 q' g_1')=1}}\sum_{\substack{|\nu_2|\le E_0/X\\ X(b_2'/d_0 d_1 q' g_1'+\nu_2)\in\mathbb{Z}}}F_{X}\Bigl(\frac{b_2'}{d_0 d_1 q' g_1'}+\nu_2\Bigr),\label{eq:S1Def}\\
S_2&=\sum_{\substack{q'\sim Q_1\\ (q',10)=1}} \sum_{g_2\sim G_2}\sum_{\substack{b_1'<d_0 q' g_2 \\ (b_1',d_0 q' g_2)=1}} \sum_{\substack{|\nu_1|\le E_0/X \\ X(b_1'/d_0 q'g_2+\nu_1)\in\mathbb{Z}}}F_{X}\Bigl(\frac{b_1'}{d_0 q' g_2}+\nu_1\Bigr).\label{eq:S2Def}
\end{align}
The bound \eqref{eq:Q0SmallBound} will be useful when $Q_0$ is small, but when $Q_0$ is large it is wasteful to sum over all these possibilities since we have not made use of the fact that $a_1,a_2\in \mathcal{E}$, a small set. To obtain an alternative bound we first sum over all $a_1\in\mathcal{E}$, then all possibilities of $q$, $b_2$, $\nu_2$. This shows that
\begin{align}
\sum_{\substack{(a_1,a_2)\in\mathcal{B}_1(N,K,\delta) \\ a_1,a_2\in\mathcal{E} }}F_X\Bigl(\frac{a_1}{X}\Bigr)F_X\Bigl(\frac{a_2}{X}\Bigr)
\ll (\log{X})^5\sup_{\substack{Q_1,G_1,G_2\\ D_0,D_1,E_0}}\sum_{\substack{d_0,d_1\in\mathcal{V}\\ d_0\sim D_0\\ d_1\sim D_1}}S_0',\label{eq:Q0LargeBound1}
\end{align}
where the supremum has the same constraints as before, and $S_0'$ is given by
\[
S_0'=\sum'_{a_1\in\mathcal{E}}\sum'_{q'\sim Q_1}\sum_{g_1'\sim G_1}'\sum'_{\substack{b_2'<d_0d_1q'g_1'}}\sum'_{|\nu_2|\le E_0/X}F_X\Bigl(\frac{a_1}{X}\Bigr)F_X\Bigl(\frac{b_2'}{d_0 d_1q' g_1'}+\nu_2\Bigr).
\]
Here the summation in $S_0'$ is constrained by
\begin{align*}
&(q',10)=(g_1',10)=(b_2',d_0d_1q' g_1')=1,\\
&X(b_2'/d_0 d_1 q' g_1'+\nu_2)\in\mathbb{Z},\\
&\exists\, b_1',g_2\text{ s.t. }\Bigl|\frac{a_1}{X}-\frac{b_1'}{q' d_0g_2}\Bigr|\le \frac{E_0}{X},\,(b_1',d_0q' g_2)=1,\,g_2\sim G_2.
\end{align*}
Again, taking a supremum over $q'$ and factorizing the summation, we find that
\begin{equation}
S_0'\ll S_1 S_3,
\label{eq:Q0LargeBound2}
\end{equation}
where $S_1$ is as given by \eqref{eq:S1Def} above, and $S_3$ is given by
\begin{equation}
S_3=\sum_{a_1\in\mathcal{E}}F_X\Bigl(\frac{a_1}{X}\Bigr)N(a_1,d_0),\label{eq:S3Def}
\end{equation}
where
\begin{equation*}
N(a_1,d_0)=\#\Bigl\{q'\sim Q_1:\exists\, b_1',g_2\text{ s.t. }\Bigl|\frac{a_1}{X}-\frac{b_1'}{q' d_0 g_2}\Bigr|\le \frac{E_0}{X},\,(b_1',d_0 q'g_2)=1,\,g_2\sim G_2\Bigr\}.
\end{equation*}
Putting together \eqref{eq:Q0SmallBound}, \eqref{eq:Q0LargeBound1}, \eqref{eq:Q0LargeBound2} we obtain
\[
\sum_{\substack{(a_1,a_2)\in\mathcal{B}_1(N,K,\delta) \\ a_1,a_2\in\mathcal{E} }}F_X\Bigl(\frac{a_1}{X}\Bigr)F_X\Bigl(\frac{a_2}{X}\Bigr)\ll (\log{X})^5\sup_{\substack{Q_1,G_1,G_2\\ D_0,D_1,E_0}}\sum_{\substack{d_0,d_1\in\mathcal{V}\\ d_0\sim D_0\\ d_1\sim D_1}} \min(S_1S_2,S_1S_3),
\]
as required.
\end{proof}

%
%
%
%
\begin{lmm}\label{lmm:LatticeSumEstimates}
Let $N K\ge X^{17/40}$ and let $S_1,S_2,S_3$ be as in Lemma \ref{lmm:LatticeSums}. Let $Q_1,G_1,G_2,D_0,D_1,E_0\ge 1$ be powers of 10 which satisfy $Q_1 G_1 G_2 D_0 D_1 E_0\ll X/N K$ and $G_1\ll G_2$. Then we have
\[
\min(S_1S_2,S_1S_3)\ll Q_0^{1-\epsilon}E_0^{1-\epsilon},
\]
where $Q_0=Q_1 G_1 G_2 D_0 D_1$.
\end{lmm}
%
%
%
%
\begin{proof}
We first bound $S_1,S_2,S_3$ individually using Lemma \ref{lmm:Generic}, Lemma \ref{lmm:Hybrid} and Lemma \ref{lmm:Hybrid2}. We will then combine these bounds to give the desired result.

We first consider the quantity $N(a_1,d_0)$ occurring in $S_3$. If $q$ and $q'$ are both counted by $N(a,d)$ then there exists $b,g$ and $b',g'$ such that $(b,q d g)=(b',q' d g')=1$ and
\[
\frac{a}{X}=\frac{b}{q d g}+O\Bigl(\frac{1}{N K Q_0}\Bigr)=\frac{b'}{q' d g'}+O\Bigl(\frac{1}{N K Q_0}\Bigr).
\]
Here we used the fact that $E_0/X\ll 1/N K Q_0$. The variables we consider satisfy $q,q'\sim Q_1\ll Q_0/G_1G_2D_0D_1$ and $g,g'\sim G_2$ and $d\sim D_0$. Thus
\[
b q' g'-b' q g\ll \frac{Q_0}{D_0D_1^2G_1^2 N K} \ll \frac{Q_0}{D_0 D_1 N K}.
\]
Let $h\ll Q_0 / D_0 D_1 N K$ be such that $b q' g'-b' q g=h$. There are $O(1+Q_0/D_0D_1 N K)$ such choices of $h$. Given $q,g,b,h$ with $(q g,b)=1$, we then see 
\begin{align*}
 q' g'\equiv h b^{-1} \Mod{q g},\\
b'\equiv h(q g)^{-1}\Mod{b}.
\end{align*}
Since $q' g'\asymp q g$ and $b'\asymp b$, there are $O(1)$ choices of $b'$ and $q' g'$. Thus there are $O(Q_0^\epsilon)$ such choices of $q',g',b'$ by the divisor bound. Thus we find that
\begin{align*}
N(a_1,d_0)\ll Q_0^\epsilon+\frac{Q_0^{1+\epsilon}}{D_0D_1N K}.
\end{align*}
Combining this with Lemma \ref{lmm:Generic} gives the bound
\begin{equation}
S_3\ll X^{23/80}+\frac{Q_0 X^{23/80}}{D_0D_1N K}.\label{eq:S3Bound}
\end{equation}
We recall $Q_0=Q_1G_1G_2D_0D_1$ is the approximate size of $q$ and that $G_1\ll G_2$, $E_0Q_0\ll X/N K\ll X$. By Lemma \ref{lmm:Hybrid} we have
\begin{align}
S_1&\ll (E_0 D_0 D_1 Q_1 G_1^2)^{27/77}+\frac{E_0 D_0 D_1 Q_1 G_1^2}{X^{50/77}}\nonumber\\
&\ll  Q_0^{27/77} E_0^{27/77},\label{eq:S2Bound1}\\
S_2&\ll (E_0 D_0 Q_1^2G_2^2)^{27/77}+\frac{Q_1^2 G_2^2 E_0 D_0}{X^{50/77}}\nonumber\\
&\ll \Bigl(\frac{ Q_0^{2} E_0}{D_0 D_1^{2} G_1^{2}}\Bigr)^{27/77}+\frac{Q_0^2 E_0}{X^{50/77} D_0 D_1 G_1}.\label{eq:S1Bound}
\end{align}
Alternatively, we may bound $S_1$ using Lemma \ref{lmm:Hybrid2}, which gives
\begin{align}
S_1&\ll (D_0 D_1 E_0)^{27/77}(Q_1G_1^2)^{1/21}+\frac{Q_1 G_1^2 (D_0 D_1)^{3/2} E_0^{5/6}}{X^{10/21}}\nonumber\\
&\ll Q_0^{1/21} (D_0 D_1 E_0)^{27/77} + \frac{Q_0 G_1 (D_0 D_1)^{1/2} E_0^{5/6}}{G_2 X^{10/21}}.\label{eq:S2Bound2}
\end{align}
If the first term in \eqref{eq:S2Bound2} dominates, then since $E_0\ll X/N K Q_0$, the bounds \eqref{eq:S2Bound2} and \eqref{eq:S1Bound} give
\begin{align*}
S_1S_2&\ll E_0^{54/77}Q_0^{54/77+1/21}+\frac{Q_0^{2+1/21} E_0^2}{X^{50/77}}\\
&\ll Q_0^{1-\epsilon} E_0^{1-\epsilon}\Bigl(1+\frac{1}{X^{50/77}}\Bigl(\frac{X}{N K}\Bigr)^{1+1/21+\epsilon}\Bigr).
\end{align*}
This shows $S_1 S_2\ll Q_0^{1-\epsilon}E_0^{1-\epsilon}$ in this case by recalling that $N K\gg X^{17/40}$ and verifying that $22/21\times 23/40<50/77$.

If instead the second term in \eqref{eq:S2Bound2} dominates, then by \eqref{eq:S2Bound1} and \eqref{eq:S2Bound2} (using $G_1\ll G_2$ and replacing $E_0^{5/6}$ with $E_0$ to simplify the expression), we have
\begin{equation}
S_1\ll \min\Bigl(( Q_0 E_0)^{27/77},\frac{Q_0 E_0 (D_0D_1)^{1/2}}{X^{10/21}}\Bigr).\label{eq:S1Combined}
\end{equation}
Combining this with \eqref{eq:S1Bound}, we obtain
\begin{align*}
S_1S_2
&\ll \Bigl(\frac{E_0 Q_0^{2}}{G_1^2 D_0 D_1^2}\Bigr)^{27/77}\Bigl((E_0 Q_0)^{27/77}\Bigr)^{1/3}\Bigl(\frac{Q_0 E_0 (D_0D_1)^{1/2}}{X^{10/21}}\Bigr)^{2/3}\nonumber\\
&\qquad+\frac{Q_0^2 E_0}{X^{50/77}D_0 D_1 G_1}\frac{Q_0E_0 (D_0 D_1)^{1/2}}{X^{10/21}}\\
&\ll \frac{Q_0^{3/2} E_0^{6/5} }{X^{3/10}}+\frac{Q_0^3 E_0^2}{X^{9/8}}.
\end{align*}
Here we have simplified the exponents appearing for an upper bound. We recall that $Q_0 E_0\ll X/N K$ and (by assumption of the lemma) $N K\gg X^{17/40}$. These give
\[
\frac{Q_0^{3/2} E_0^{6/5} }{X^{3/10}}\ll \frac{Q_0 E_0}{X^{3/10}}(X^{23/40})^{1/2}\ll \frac{Q_0 E_0}{X^{1/80}}.
\]
Thus this term is $O(Q_0^{1-\epsilon}E_0^{1-\epsilon})$, and so
\begin{equation}
S_1S_2\ll Q_0^{1-\epsilon}E_0^{1-\epsilon}+\frac{Q_0^3 E_0^2}{X^{9/8}}.\label{eq:S1S2Bound}
\end{equation}
Similarly, we find that combining \eqref{eq:S1Combined} and \eqref{eq:S3Bound} gives
\begin{align*}
S_1 S_3&\ll X^{23/80}(Q_0 E_0)^{27/77}+\frac{Q_0 E_0 (D_0D_1)^{1/2}}{X^{10/21}}\frac{X^{23/80}Q_0}{D_0D_1N K}\nonumber\\
&\ll X^{23/80}(Q_0 E_0)^{27/77}+ \frac{Q_0^{2} E_0}{X^{3/16}N K}.
\end{align*}
Here we used $10/21-23/80>3/16$. Since $Q_0 E_0\ll X/N K$ and $N K\gg X^{17/40}\gg X^{13/32+\epsilon}$, we see that
\[
\frac{Q_0^2 E_0}{X^{3/16}N K}\ll Q_0\frac{X^{13/16}}{(N K)^2}\ll Q_0^{1-\epsilon}\ll Q_0^{1-\epsilon}E_0^{1-\epsilon}.
\]
Thus we have
\begin{equation}
S_1S_3\ll Q_0^{1-\epsilon}E_0^{1-\epsilon}+ X^{23/80}(Q_0 E_0)^{27/77}.
\label{eq:S2S3Bound}
\end{equation}
Combining \eqref{eq:S1S2Bound} and \eqref{eq:S2S3Bound}, we obtain
\begin{align*}
&\min(S_1S_2,S_1S_3)\ll Q_0^{1-\epsilon}E_0^{1-\epsilon}+\min\Bigl(X^{23/80}(Q_0E_0)^{27/77},\frac{Q_0^3 E_0^2}{X^{9/8}}\Bigr).
\end{align*}
We find that
\begin{align*}
\min\Bigl(X^{23/80}(Q_0E_0)^{27/77},\frac{Q_0^3E_0^2}{X^{9/8}}\Bigr)&\ll \Bigl(X^{23/80}(Q_0E_0)^{27/77}\Bigr)^{77/100}\Bigl(\frac{Q_0^3E_0^2}{X^{9/8}}\Bigr)^{23/100}\\
&=\frac{Q_0^{96/100}E_0^{73/100}}{X^{(90-77)\times 23/8000}}\\
&\ll Q_0^{1-\epsilon}E_0^{1-\epsilon}.
\end{align*}
Thus we have $\min(S_1S_2,S_3S_2)\ll Q_0^{1-\epsilon}E_0^{1-\epsilon}$ in all cases, as desired.
\end{proof}
%
%
%
%
Having established the technical Lemma \ref{lmm:LatticeSums} and Lemma \ref{lmm:LatticeSumEstimates}, we are now in a position to prove Proposition \ref{prpstn:LatticeBound}.
%
%
%
%
\begin{proof}[Proof of Proposition \ref{prpstn:LatticeBound}]
We wish to show that
\[
\sum_{\substack{(a_1,a_2)\in\mathcal{B}_1(N,K,\delta)\\ a_1,a_2\in\mathcal{F}\cap\mathcal{E}}}F_X\Bigl(\frac{a_1}{X}\Bigr)F_X\Bigl(\frac{a_2}{X}\Bigr)
\ll \frac{(\log{X})^5}{(Q+E)^{\epsilon/4}}\frac{X}{N K}
\]
in the region $X^{17/40}\le N K$. Since $\mathcal{B}_1(N,K,\delta)\cap \mathcal{F}^2=\emptyset$ unless $Q+E\ll (X/N K)^2$ by Lemma \ref{lmm:RatSize}, we may assume that $Q+E\ll (X/NK)^2$.

By Lemma \ref{lmm:LatticeSums} and Lemma \ref{lmm:LatticeSumEstimates} we have
\begin{align*}
\sum_{\substack{(a_1,a_2)\in\mathcal{B}_1(N,K,\delta)\\ a_1,a_2\in\mathcal{F}\cap\mathcal{E}}}F_X\Bigl(\frac{a_1}{X}\Bigr)F_X\Bigl(\frac{a_2}{X}\Bigr)
&\ll (\log{X})^5\sup_{\substack{Q_1,G_1,G_2\\ D_0,D_1,E_0}}\sum_{\substack{d_0,d_1\in\mathcal{V}\\ d_0\sim D_0\\ d_1\sim D_1}} \min(S_1S_2,S_1S_3)\\
&\ll  (\log{X})^5 \sup_{\substack{Q_1,G_1,G_2\\ D_0,D_1,E_0}} \sum_{\substack{d_0,d_1\in\mathcal{V}\\ d_0\sim D_0\\ d_1\sim D_1}}Q_0^{1-\epsilon}E_0^{1-\epsilon}.
\end{align*}
There are $O(Q_0^{\epsilon/2})$ elements $d_0,d_1\in\mathcal{V}$ with $d_0,d_1\ll Q_0$. Thus, recalling that $Q_0E_0\ll X/N K$, we have
\begin{align*}
\sum_{\substack{(a_1,a_2)\in\mathcal{B}_1(N,K,\delta) \\ a_1,a_2\in\mathcal{F}\cap\mathcal{E} }}F_X\Bigl(\frac{a_1}{X}\Bigr)F_X\Bigl(\frac{a_2}{X}\Bigr)&\ll \sup_{\substack{Q_1,G_1,G_2\\ D_0,D_1,E_0}}(\log{X})^5Q_0^{1-\epsilon/2}E_0^{1-\epsilon}\\
&\ll (\log{X})^5\Bigl(\frac{ X}{ N K}\Bigr)^{1-\epsilon/2}.
\end{align*}
We recall that $Q+E\ll (X/N K)^2$, and so this gives
\[
\sum_{\substack{(a_1,a_2)\in\mathcal{B}_1(N,K,\delta) \\ a_1,a_2\in\mathcal{F}\cap\mathcal{E} }}F_X\Bigl(\frac{a_1}{X}\Bigr)F_X\Bigl(\frac{a_2}{X}\Bigr)\ll \frac{(\log{X})^5 X}{(Q+E)^{\epsilon/4}N K},
\]
as required.
\end{proof}
%
%
%
%
%
%
%
\section{Line Estimates} \label{sec:Line}
%
%
%
%
In this section we establish Proposition \ref{prpstn:LineBound}, which controls the contribution from pairs of angles which cause a large contribution to the bilinear sums considered in Section \ref{sec:Bilinear} to come from a line. If a line $L$ makes a large contribution, then $(a_1,a_2,X)$ must lie close to the low height plane orthogonal to this line. We note that we do not make use of the fact that these angles lie outside the major arcs, but it is vital that the angles are restricted to the small set $\mathcal{E}$.
%
%
%
%
\begin{lmm}[Line angles lie in low height plane]\label{lmm:Line}
Let $0<\delta<1$ and $K,N,X>1$ be reals with $\delta\ge N/X$ and $N K\ge X^{17/40}$. Let $\mathcal{B}_2=\mathcal{B}_2(N,K,\delta)$ be the set of integer pairs $(a_1,a_2)\in[0,X)^2$ such that there is a line $L$ through the origin such that
\[
\#\{\mathbf{n}\in L\cap\mathbb{Z}^3: |n_1a_1+n_2a_2+n_3X|\le \delta X,\,\|\mathbf{n}\|_2\le N\}\gg \delta N^2 K.
\]
Then all pairs $(a_1,a_2)\in\mathcal{B}_2$ satisfy
\[
v_1a_1+v_2 a_2+v_3 X+v_4=0
\]
for some integers $v_1,v_2,v_3,v_4\ll  X/N^2K$ not all zero.
\end{lmm}
%
%
%
%
\begin{proof}
Let $\mathbf{v}=(v_1,v_2,v_3)$ be a non-zero element of $\mathbb{Z}^3\cap L$ of smallest norm, and let $V=\|\mathbf{v}\|_2$ and $\epsilon_1=|v_1a_1+v_2a_2+v_3X|$. Then all of $\mathbb{Z}^3\cap L$ is generated by $\mathbf{v}$, and so
\[
\#\{\mathbf{n}\in L\cap\mathbb{Z}^3: |n_1a_1+n_2a_2+n_3X|\le \delta X,\,\|\mathbf{n}\|_2\le N\}\ll \min\Bigl(\frac{N}{V},\frac{\delta X}{\epsilon_1}\Bigr).
\]
By assumption, this is also $\gg \delta N^2K$, and so we obtain
\[
V\ll \frac{1}{N K\delta}\ll \frac{X}{N^2K},\qquad \epsilon_1\ll \frac{X}{N^2K}.
\]
Letting $v_4=-(v_1a_1+v_2a_2+v_3X)\in\{\pm \epsilon_1\}$ gives the result.
\end{proof}
%
%
%
%
\begin{lmm}[Sparse sets restricted to low height planes] \label{lmm:LineBound}
Let $\mathcal{C}\subseteq[0,X)$ be a set of integers. Then we have for any $V\ge 1$
\[
\begin{split}
\#\Bigl\{(a_1,a_2)\in\mathcal{C}^2:\exists (v_1,v_2,v_3,v_4)\in[-V,V]^4\backslash\{\mathbf{0}\}\text{ s.t. }v_1a_1+v_2a_2+v_3 X+v_4=0\Bigr\}\\
\ll X^{o(1)}\Bigl(\#\mathcal{C}^{5/4}V^2+\frac{\#\mathcal{C}^{3/2}V^{3}}{X^{1/2}}\Bigr).\end{split}
\]
\end{lmm}
%
%
%
%
\begin{proof}
Trivially there are $O(\#\mathcal{C}^2)$ choices of $a_1,a_2\in\mathcal{C}$, which gives the required bound if $V>\#\mathcal{C}^{3/8}$. In particular, we may assume that $V< \#\mathcal{C}\le X$. There are $O(\#\mathcal{C})$ points with $a_1=0$ or $a_2=0$, so we may assume that $a_1,a_2\ne 0$.

We first claim that there are
\begin{equation}
O(\#\mathcal{C} V^2 X^{o(1)})\label{eq:ZeroTerms}
\end{equation}
choices of $v_1$, $v_2$, $v_3$, $v_4$, $a_1$, and $a_2$ satisfying $v_1a_1+v_2a_2+v_3X+v_4=0$ with at least one of $v_1,v_2,v_3,v_4$ equal to $0$ and at least one of $v_1,v_2,v_3,v_4$ non-zero. For example, if $v_1=0$ then there are $O(\#\mathcal{C}V^2)$ choices of $a_1,v_3,v_4$, which then determines $v_2a_2$. Since there are no non-zero solutions to $v_3X+v_4=0$, this is non-zero and so there are $O(X^\epsilon)$ choices of $v_2,a_2$. The other cases are entirely analogous. Thus it suffices to consider pairs $(a_1,a_2)$ such that $v_1a_1+v_2a_2+v_3X+v_4=0$ for some $v_1,v_2,v_3,v_4$ all non-zero. We let $\mathcal{C}_2$ denote the set of such pairs.

Given $a\in\mathbb{Z}$, let $M_a$ be the smallest value of $(c_1^2+c_2^2)^{1/2}$ over all non-zero integers $c_1,c_2$ such that $c_1\equiv c_2 X\Mod{a}$. We divide $\mathcal{C}$ into $O(\log{X})^2$ subsets localizing the size of $a<X$ and $M_a<X$ by considering the sets
\[
\mathcal{C}(A,M)=\{a\in\mathcal{C}:\,a\sim A,\, M_a\sim M\}.
\]
There are $O(M^2)$ choices of $c_1,c_2$ with $(c_1^2+c_2^2)^{1/2}\le M$, and given any such choice with $M<X$ there are $X^{o(1)}$ choices of $a|c_1-c_2 X$ from the divisor bound (noting that this must be non-zero). Thus we have that 
\[
\#\mathcal{C}(A,M)\le X^{o(1)}\min(\#\mathcal{C},M^2).
\]
By Cauchy-Schwarz we have
\begin{align}
\sum_{(a_1,a_2)\in\mathcal{C}_2}1
&\ll X^{o(1)} \#\mathcal{C}^{1/2}\sup_{A,M}\Bigl(\sum_{a_1\in\mathcal{C}}\Bigl(\sum_{\substack{a_2\in\mathcal{C}(A,M) \\ (a_1,a_2)\in\mathcal{C}_2}}1\Bigr)^2\Bigr)^{1/2}\nonumber\\
&\le X^{o(1)}\#\mathcal{C}^{1/2} \sup_{A,M}N_2^{1/2},\label{eq:CauchyCountBound}
\end{align}
where
\begin{align*}
N_2&=\#\Bigl\{(a_2,a_2',a_1)\in\mathcal{C}(A,M)^2\times\mathcal{C}:\,a_1=\frac{v_2a_2}{v_1X}+\frac{v_3}{v_1}+\frac{v_4}{v_1X}=\frac{v_2'a_2'}{v_1'X}+\frac{v_3'}{v_1'}+\frac{v_4'}{v_1'X},\\
&\qquad \text{ for some integers } 0<|v_1|,|v_1'|,|v_2|,|v_2'|,|v_3|,|v_3'|,|v_4|,|v_4'|\le V,\,a_1a_2a_2'\ne 0\Bigr\}.
\end{align*}
We wish to bound $N_2$. Given $v_1,v_1'$, let $d=\gcd(v_1,v_1')$ and $v_1=d\tilde{v}_1$, $v_1'=d\tilde{v}_1'$ so $\gcd(\tilde{v}_1,\tilde{v}_1')=1$. We split the count $N_2$ by considering $\max(\tilde{v}_1,\tilde{v}_1')\sim V_1$ for different choices of $V_1$. Since $V< X$, there are $O(\log{X})$ choices of $V_1$ we need to consider. This gives
\begin{equation}
N_2\ll (\log{X})\sup_{V_1}N_3(V_1),
\label{eq:N2Bound}
\end{equation}
where
\begin{align*}
&N_3(V_1)=\#\Bigl\{(a_2,a_2',a_1,d,\tilde{v}_1,\tilde{v}_1',v_2,v_3,v_4,v_2',v_3',v_4'):\,0<d\le V/V_1, a_1\in\mathcal{C}\backslash\{0\},\\
&\quad a_1d\tilde{v}_1\tilde{v}_1'=\Bigl(\frac{v_2a_2}{X}+v_3+\frac{v_4}{X}\Bigr)\tilde{v}_1'=\Bigl(\frac{v_2'a_2'}{X}+v_3'+\frac{v_4'}{X}\Bigr)\tilde{v}_1,\,a_2,a_2'\in\mathcal{C}(A,M)\backslash\{0\},\\
&\quad 0<|\tilde{v}_1|,|\tilde{v}_1'|,|v_2|,|v_2'|,|v_3|,|v_3'|,|v_4|,|v_4'|\le V,\,\max(|\tilde{v}_1|,|\tilde{v}_1'|)\sim V_1,\,\gcd(\tilde{v}_1,\tilde{v}_1')=1\Bigr\}.
\end{align*}
We wish to show that $N_3(V_1)\ll X^{o(1)}(\#\mathcal{C}^{3/2}V^4+\#\mathcal{C}^2V^6/X)$ for any choice of $0<V_1<V$. By symmetry we may assume $|\tilde{v}_1|\ge |\tilde{v}_1'|$, so $|\tilde{v}_1|\sim V_1$. Let $b_1=\tilde{v}_1'v_2$, $b_2=-\tilde{v}_1v_2'$, $b_3=\tilde{v}_1'v_3-\tilde{v}_1v_3'$ and $b_4=\tilde{v}_1'v_4-\tilde{v}_1v_4'$. We see that any solution counted by $N_3(V_1)$ must give a solution to
\[
b_1a_2+b_2a_2'+b_3X+b_4=0
\]
with $0\le|b_1|,|b_2|,|b_3|,|b_4|\le 2V_1V$ and $b_1,b_2\ne 0$.

There are $O(V_1^3V^3)$ choices of $b_2,b_3,b_4$ and  $O(\#\mathcal{C})$ choices of $a_2'$. Given such a choice of $b_2,b_3,b_4,a_2'$, there are $O(X^{o(1)})$ choices of $b_1$ and $a_2$ by the divisor bound, since $b_1a_2=-b_2a_2'-b_3X-b_4$ and $b_1a_2$ is non-zero. Given $b_1,b_2$ there are $O(X^{o(1)})$ choices of $\tilde{v}_1,\tilde{v}_1',v_2,v_2'$ by the divisor bound (recall $b_1,b_2\ne 0$). Given $\tilde{v}_1,\tilde{v}_1'$ and $b_3$ we see that
\[
v_3\equiv b_3 \tilde{v}_1'{}^{-1}\Mod{\tilde{v}_1}.
\]
Thus there are $O(V/V_1)$ choices of $v_3$ (here we use the fact that $\gcd(\tilde{v}_1,\tilde{v}_1')=1$). Given $v_1,\tilde{v}_1,b_3$ and such a choice of $v_3$ there is just one choice of $v_3'$. Similarly, there are $O(V/V_1)$ choices of $v_4,v_4'$ given $\tilde{v}_1,\tilde{v}_1'$ and $b_4$. Given $\tilde{v}_1,v_2,v_3,v_4,a_2$, there are $O(X^{o(1)})$ choices of $d,a_1$ since $d a_1\tilde{v}_1X=v_2a_2+v_3 X+v_4$ and $d a_1 \tilde{v}_1 X\ne 0$. Putting this all together, we have 
\begin{equation}
N_3(V_1)\ll X^{o(1)}\#\mathcal{C} V_1 V^5.
\label{eq:V1Small}
\end{equation}
 This bound will be good for us if $V_1$ is small, but we need a different argument if $V_1$ is large.

We note that
\[
b_3=-\frac{b_1a_2+b_2a_2'+b_4}{X}\ll \frac{V V_1 A}{X}.
\]
We make a choice of $a_2,a_2',b_1$, for which there are $\ll V V_1 X^{o(1)}\min(M^4,\#\mathcal{C}^2)$ possibilities counted by $N_3(V_1)$. We see that $b_3,b_4$ satisfy
\[
b_3 X+b_4\equiv b_1a_2\Mod{a_2'}.
\]
Let $b_{3,0},b_{4,0}$ be a solution to this congruence with $b_{3,0}^2+b_{4,0}^2$ minimal. We may assume that $b_{3,0}\ll VV_1A/X$ and $b_{4,0}\ll VV_1$ since otherwise there are no possible $b_3,b_4$. All pairs $b_3,b_4$ satisfying the congruence are then of the form  $(b_3,b_4)=(b_{3,0}+b_3',b_{4,0}+b_4')$ for some integers $b_3',b_4'$ satisfying $b_3'X+b_4'\equiv 0\Mod{a_2'}$ and $b_3'\ll V V_1A/X$, $b_4'\ll V V_1$. This forces $b_3'\mathbf{e}_1+b_4'\mathbf{e}_2$ to lie in a lattice $\Lambda\subset\mathbb{Z}^2$ of determinant $a_2'$, where $\mathbf{e}_1,\mathbf{e}_2$ are the standard basis vector of $\mathbb{Z}^2$. Let $\phi:\mathbb{R}^2\rightarrow\mathbb{R}^2$ be the linear map which is a dilation by a factor $X/A$ in the $\mathbf{e}_1$ direction, and $\Lambda'=\phi(\Lambda)$, a lattice in $\mathbb{R}^2$ of determinant $a_2X/A\asymp X$. 

Let $\Lambda'$ have a Minkowski-reduced basis $\{\mathbf{v}_1,\mathbf{v}_2\}$. We recall this means that $\|\mathbf{v}_1\|_2\cdot \|\mathbf{v}_2\|_2\asymp \det(\Lambda)=a_2'X/A\asymp X$ and $\|n_1\mathbf{v}_1+n_2\mathbf{v}_2\|_2\asymp \|n_1\mathbf{v}_1\|_2+\|n_2\mathbf{v}_2\|_2$. From the definition of $M_a$, we see that the smallest non-zero vector in $\Lambda$ has length at least $M/10$, and so since $\phi$ can only increase the length of vectors we have $\|\mathbf{v}_1\|_2,\|\mathbf{v}_2\|_2\ge M/10$. 

The set of vectors $b_3'\mathbf{e}_1+b_4'\mathbf{e}_2$ in $\Lambda$ inside the bounded region $|b_3'|\ll  V V_1 A/X$, $|b_4'|\ll V V_1$ can be injected by $\phi$ into the set $\{\mathbf{x}\in\Lambda':\,\|\mathbf{x}\|_2\le CV V_1\}$ for some suitably large constant $C$. Thus, provided $C$ is sufficiently large so that we also have $\|n_1\mathbf{v}_1+n_2\mathbf{v}_2\|_2\ge \max_i\|n_i\mathbf{v}_i\|_2/C$, we see that the number of pairs $(b_3',b_4')$ is bounded by
\begin{align*}
\#\{\mathbf{x}\in\Lambda':\|\mathbf{x}\|_2\le C V V_1\}&=\#\Bigl\{(n_1,n_2)\in\mathbb{Z}^2:\,\|n_1\mathbf{v}_1+n_2\mathbf{v}_2\|_2\le C V V_1\}\\
&\le \#\Bigl\{(n_1,n_2)\in\mathbb{Z}^2:\,|n_1|\le C^2\frac{V V_1}{\|\mathbf{v}_1\|_2},|n_2|\le C^2\frac{V V_1}{\|\mathbf{v}_2\|_2}\Bigr\}\\
&\ll \Bigl(1+\frac{V V_1}{\|\mathbf{v}_1\|_2}\Bigr)\Bigl(1+\frac{V V_1}{\|\mathbf{v}_2\|_2}\Bigr)\\
&\ll 1+ \frac{V V_1}{M}+\frac{V^2V_1^2}{\det(\Lambda')}\\
&\ll 1+\frac{V V_1}{M}+\frac{V^2 V_1^2}{X}.
\end{align*}
Here we used the fact that $\|\mathbf{v}_1\|_2,\|\mathbf{v}_2\|_2\gg M$ and $\|\mathbf{v}_1\|_2\cdot\|\mathbf{v}_2\|_2\asymp \det(\Lambda')$ in the penultimate line, and $\det(\Lambda')\asymp X$ in the final line.

Given any choice of $a_2,a_2',b_1,b_3,b_4$, we see that $b_2$ is then determined uniquely by $b_1a_2+b_2a_2'=b_3X+b_4$, since we have already chosen all the other terms. As before, given $a_2$, $a_2'$, $b_1$, $b_2$, $b_3$, $b_4$ there are $O(X^{o(1)}V^2/V_1^2)$ choices of $\tilde{v}_1$, $\tilde{v}_1'$, $v_2$, $v_3$, $v_4$, $v_2'$, $v_3'$, $v_4'$, $d$, $a_1$. Putting this all together, we obtain the bound
\[
N_3(V_1)\ll X^{o(1)}\frac{V^3}{V_1}\min(M^4,\#\mathcal{C}^2)\Bigl(1+\frac{V V_1}{M}+\frac{V^2V_1^2}{X}\Bigr).
\]
Since $\min(M^4,\#\mathcal{C}^2)\le \min(M\#\mathcal{C}^{3/2},\#\mathcal{C}^2)$ this gives
\begin{equation}
N_3(V_1)\ll \Bigl(\#\mathcal{C}^2\frac{V^3}{V_1}+\#\mathcal{C}^{3/2}V^4+\frac{\#\mathcal{C}^2V^6}{X}\Bigr)X^{o(1)}.
\label{eq:V1Large}
\end{equation}
Combining \eqref{eq:V1Small} and \eqref{eq:V1Large}, we obtain
\begin{align}
N_3(V_1)&\ll X^{o(1)}\min\Bigl(\#\mathcal{C} V_1 V^5,\#\mathcal{C}^2\frac{V^3}{V_1}+\#\mathcal{C}^{3/2}V^4+\frac{\#\mathcal{C}^2V^6}{X}\Bigr)\nonumber\\
&\ll X^{o(1)}\Bigl(\Bigl(\#\mathcal{C} V_1 V^5\Bigr)^{1/2}\Bigl(\#\mathcal{C}^2\frac{V^3}{V_1}\Bigr)^{1/2}+\#\mathcal{C}^{3/2}V^4+\frac{\#\mathcal{C}^2V^6}{X}\Bigr)\nonumber\\
&\ll X^{o(1)}\Bigl(\#\mathcal{C}^{3/2}V^4+\frac{\#\mathcal{C}^2V^6}{X}\Bigr).
\label{eq:N3Bound}
\end{align}
We substitute \eqref{eq:N2Bound} and \eqref{eq:N3Bound} into \eqref{eq:CauchyCountBound}, and obtain
\[
\sum_{\substack{(a_1,a_2)\in\mathcal{C}_2}}1\ll X^{o(1)}\Bigl(\#\mathcal{C}^{5/4}V^2+\frac{\#\mathcal{C}^{3/2}V^{3}}{X^{1/2}}\Bigr).
\]
We recall from \eqref{eq:ZeroTerms} that terms with $v_1v_2v_3v_4a_1a_2=0$ contribute a total $O(\#\mathcal{C}V^2X^{o(1)})$, which is negligible compared with the $\#\mathcal{C}^{5/4}V^2$ term above. Thus we obtain the result.
\end{proof}
We see that Lemma \ref{lmm:LineBound} improves on the trivial bound $O(X^{o(1)}\min(V^3\#\mathcal{C},\#\mathcal{C}^2))$ if $V^{8/3+\epsilon}\ll \#\mathcal{C}\ll V^{4-\epsilon}+X^{1-\epsilon}$.
%
%
%
%
\begin{proof}[Proof of Proposition \ref{prpstn:LineBound}] We wish to show that
\[
\sum_{\substack{ (a_1,a_2)\in\mathcal{B}_2(N,K,\delta) \\ a_1,a_2\in\mathcal{E}' }}F_X\Bigl(\frac{a_1}{X}\Bigr)F_X\Bigl(\frac{a_2}{X}\Bigr)\ll \frac{X^{1-\epsilon}}{N K}
\]
in the region $N\gg X^{9/25}$.
We recall that
\[
\mathcal{E}'=\Bigl\{a<X:F_X\Bigl(\frac{a}{X}\Bigr)\sim \frac{1}{B}\Bigr\}\subseteq\mathcal{E}
\]
for some $B\ll X^{23/80}$. Trivially, we have that
\[
\sum_{a_1,a_2\in\mathcal{E}'}F_X\Bigl(\frac{a_1}{X}\Bigr)F_X\Bigl(\frac{a_2}{X}\Bigr)\le \frac{(\#\mathcal{E}')^2}{B^2}.
\]
 By Lemma \ref{lmm:DigitDistribution}, we have
\begin{equation}
\#\mathcal{E}'\ll B^{235/154}X^{59/433}.\label{eq:BSetBound}
\end{equation}
 This gives 
 \[
\sum_{a_1,a_2\in\mathcal{E}'}F_X\Bigl(\frac{a_1}{X}\Bigr)F_X\Bigl(\frac{a_2}{X}\Bigr)\ll B^{81/77}X^{118/433}\ll B X^{23/80-\epsilon}
\]
on verifying that $4/77\times 23/80+118/433<23/80$. This gives the required bound if $N K\ll X^{57/80}/B$.

Alternatively, if $N K\gg X^{57/80}/B$, we use Lemmas \ref{lmm:Line} and \ref{lmm:LineBound} to bound $\#(\mathcal{B}_2\cap(\mathcal{E}')^2)$, and obtain
\begin{align}
\sum_{\substack{(a_1,a_2)\in\mathcal{B}_2(N,K,\delta)\\ a_1,a_2\in \mathcal{E}'}}&F_X\Bigl(\frac{a_1}{X}\Bigr)F_X\Bigl(\frac{a_2}{X}\Bigr)\le \frac{\#(\mathcal{B}_2(N,K,\delta)\cap(\mathcal{E}')^2)}{B^2}\nonumber\\
&\hspace{-1cm}\le\frac{1}{B^{2}}\#\Bigl\{a_1,a_2\in\mathcal{E}':\exists \mathbf{v}\in\mathbb{Z}^4\backslash\{\mathbf{0}\}\text{ s.t. }\|\mathbf{v}\|_2\ll \frac{X}{N^2K},\mathbf{v}\cdot\mathbf{a}=0\Bigr\}\nonumber\\
&\hspace{-1cm}\ll \frac{X^{o(1)}}{B^{2}}\Bigl((\#\mathcal{E}')^{5/4}\Bigl(\frac{X}{N^2K}\Bigr)^2+\frac{(\#\mathcal{E}')^{3/2}}{X^{1/2}}\Bigl(\frac{X}{N^2K}\Bigr)^{3}\Bigr).\label{eq:B2Bound2}
\end{align}
Here we have written $\mathbf{a}$ for the vector $(a_1,a_2,X,1)\in\mathbb{Z}^4$. 

Since $N K\gg X^{57/80}/B$, we have $X/N K\ll X^{23/80}B$. Combining this bound with \eqref{eq:BSetBound}, we obtain a bounds for $(\#\mathcal{E}')^{5/4}B^{-2}X/N K$ and $(\#\mathcal{E}')^{3/2}B^{-2}X^{-1/2}(X/N K)^2$ of the form $X^a B^b$ for some $b>0$. Since we are only considering $B\ll X^{23/80}$, these expressions are maximized when $B\asymp X^{23/80}$. When $B\asymp X^{23/80}$ we have $\#\mathcal{E}'\ll X^{23/40}$ and $X/N K\ll X^{23/40}$. Thus we obtain the bounds
\begin{align*}
\frac{(\#\mathcal{E}')^{5/4}}{B^2}\frac{X}{N K}&\ll X^{115/160}=X^{23/32},\\
\frac{(\#\mathcal{E}')^{3/2}}{B^2 X^{1/2}}\Bigl(\frac{X}{N K}\Bigr)^2& \ll X^{75/80}=X^{15/16}.
\end{align*}
Substituting these bounds into \eqref{eq:B2Bound2} gives
\begin{align*}
\sum_{\substack{ (a_1,a_2)\in\mathcal{B}_2(N,K,\delta) \\ a_1,a_2\in\mathcal{E}' }}F_X\Bigl(\frac{a_1}{X}\Bigr)F_X\Bigl(\frac{a_2}{X}\Bigr)\ll \Bigl(\frac{X^{23/32}}{N^2}+\frac{X^{15/16}}{N^{3}}\Bigr)\frac{X^{1+o(1)}}{N K}.
\end{align*}
We can then verify that $2\times 9/25>23/32$ and that $3\times 9/25>15/16$, so for $N\gg X^{9/25}$ this is $O(X^{1-\epsilon}/NK)$, as required.
\end{proof}
%
%
%
%
%
%
%

\section{Modifications for Theorem \ref{thrm:ManyDigits}}\label{sec:ManyPrimes}
Theorem \ref{thrm:ManyDigits} follows from essentially the same overall approach as in Theorem \ref{thrm:MainTheorem}. We only provide a brief sketch the proof, leaving the complete details to the interested reader. When $q$ is large, there is negligible benefit from using the $235/154^{th}$ moment, so we just use $\ell^1$ bounds. For $Y=q^k$ a power of $q$, we let
\[
F_Y(\theta)=Y^{-\log(q+s)/\log{q}}\Bigl|\sum_{n<Y}\mathbf{1}_\mathcal{A}(n)e(n\theta)\Bigr|=\prod_{i=0}^{k-1}\frac{1}{q-s}\Bigl|\sum_{\substack{n_i<q\\ n_i\notin\mathcal{B}}}e(n_i q^i\theta)\Bigr|.
\]
The inner sum is $\le \min(q-s,\,s+2/\|q^i\theta\|)$. Thus, similarly to Lemma \ref{lmm:L1Bound}, we find
\begin{align}
\sum_{t<Y}F_{Y}\Bigl(\frac{t}{Y}\Bigr)&\ll \frac{1}{(q-s)^k}\prod_{i=0}^{k-1}\Bigl|\sum_{t_i<q}\min\Bigl(q-s,\frac{q}{t_i}+\frac{q}{q-t_i}+s\Bigr)\Bigr|\nonumber\\
&= O\Bigl(\frac{q\log{q}+q s}{q-s}\Bigr)^k.\label{eq:qLargeL1}
\end{align}
In particular, for $q$ large enough in terms of $\epsilon$ and $s\le q^{23/80}$, this is $O(Y^{23/80+\epsilon})$. We can use this bound in place of Lemma \ref{lmm:L1Bound} and Lemma \ref{lmm:DigitDistribution} throughout the argument with the same (or stronger) consequences. This gives the first part of Theorem \ref{thrm:ManyDigits}.

For the second part of Theorem \ref{thrm:ManyDigits}, we see that in the special case $\mathcal{B}=\{0,\dots,s-1\}$ we have
\[
\Bigl|\sum_{\substack{n_i<q\\ n_i\notin\mathcal{B}}}e(n_i\theta)\Bigr|=\Bigl|\frac{e((q-s)\theta)-1}{e(\theta)-1}\Bigr|\le \min\Bigl(q-s, \frac{2}{\|\theta\|}\Bigr).
\]
Using this bound, get a corresponding improvement on \eqref{eq:qLargeL1}, which gives
\begin{align}
\sum_{t<Y}F_Y\Bigl(\frac{t}{Y}\Bigr)&\ll \frac{1}{(q-s)^k}\prod_{i=0}^{k-1}\sum_{t_i<q}\min\Bigl(q-s,\frac{q}{t_i}+\frac{q}{q-t_i}\Bigr)\nonumber\\
&=O\Bigl(\frac{q\log{q}+q-s}{q-s}\Bigr)^k.\label{eq:qLargeL12}
\end{align}
If $s\le q-q^{57/80}$ and $q$ is sufficiently large in terms of $\epsilon$, this gives a bound $Y^{23/80+\epsilon}$. As before, using this bound in place of Lemma \ref{lmm:L1Bound} and Lemma \ref{lmm:DigitDistribution} throughout gives the result.

For the results mentioned after Theorem \ref{thrm:ManyDigits}, we find that in the further restricted ranges $s\le q^{1/4-\delta}$ (or $s\le q-q^{3/4+\delta}$ if $\mathcal{B}=\{0,\dots,s-1\}$), the bound \eqref{eq:qLargeL1} (or \eqref{eq:qLargeL12}) give an $\ell^1$ bound of $Y^{1/4-\delta/2}$. Following this through the argument, we obtain a wider Type II range and can estimate bilinear sums provided $N\in[X^{5/16},X^{1/2}]$ instead of $[X^{9/25},X^{17/40}]$. By symmetry, we can then also estimate terms in $N\in [X^{1/2},X^{11/16}]$. This allows us to obtain asymptotic estimates for all the terms in the right hand side of the identity
\[
S(\mathcal{A},X^{1/2})=S(\mathcal{A},X^{3/8-2\epsilon})-\sum_{X^{3/8-2\epsilon}\le p<X^{1/2}}S(\mathcal{A}_p,p),
\]
by the equivalents of Proposition \ref{prpstn:FinalSieve} and Proposition \ref{prpstn:FinalTypeII} adapted to this larger Type II range.
%
%
%
%
%
%
\section{Acknowledgments}
We thank Ben Green for introducing the author to this problem, Xuancheng Shao for useful discussions and Fabian Karwatowski for some important corrections. We also thank the anonymous referee for many helpful suggestions and corrections. The author is supported by a Clay Research Fellowship and a Fellowship by Examination of Magdalen College, Oxford. Part of this work was performed whilst the author was visiting Stanford university, whose hospitality is gratefully acknowledged.
%
%
%
%
\bibliographystyle{plain}
\bibliography{Digits}
\end{document}